\documentclass[11pt]{article}
\usepackage{amsmath,amsfonts,amssymb,amsthm} 
\usepackage{bm} 
\usepackage{enumerate} 
\usepackage{indentfirst}
\usepackage{color} 
\usepackage{float}
\usepackage{graphics}
\usepackage{epsfig}
\usepackage{amsbsy}
\usepackage{mathrsfs}
\usepackage{mathtools}
\usepackage{comment}
\usepackage{empheq}
\usepackage{hyperref}
\hypersetup{colorlinks=true, 
	linkcolor=blue 
}
\usepackage{authblk}
\AtBeginDocument{}

\allowdisplaybreaks
\newtheorem{theorem}{Theorem}[section]
\newtheorem{lemma}[theorem]{Lemma}
\newtheorem{proposition}[theorem]{Proposition}
\newtheorem{corollary}[theorem]{Corollary}

\newtheorem{remark}[theorem]{Remark}

\def\R{{\mathbb{R}}}
\def\H{{I\!\!H}}
\def\N{{I\!\!N}}
\def\CC{{\rm \kern.24em \vrule width.02em height1.4ex depth-.05ex
		\kern-.26emC}}

\def\Xint#1{\mathchoice
	{\XXint\displaystyle\textstyle{#1}}%
	{\XXint\textstyle\scriptstyle{#1}}%
	{\XXint\scriptstyle\scriptscriptstyle{#1}}%
	{\XXint\scriptscriptstyle\scriptscriptstyle{#1}}%
	\!\int}
\def\XXint#1#2#3{{\setbox0=\hbox{$#1{#2#3}{\int}$}
		\vcenter{\hbox{$#2#3$}}\kern-.5\wd0}}
\def\dashint{\Xint-}

\newcommand{\pvint}{\mathop{\mathrlap{\pushpv}}\!\int}
\newcommand{\pushpv}{\mathchoice
	{\mkern22.5mu\rule[.6ex]{.5em}{.5pt}}
	{\mkern2.8mu\rule[.5ex]{.35em}{.8pt}}
	{\mkern2.5mu\rule[.50ex]{.7em}{.7pt}}
	{\mkern2mu\rule[.5ex]{.7em}{.5pt}}
}

\newcommand{\ppvint}{\mathop{\mathrlap{\pushppv}}\!\int}
\newcommand{\pushppv}{\mathchoice
	{\mkern27.9mu\rule[.6ex]{.5em}{.5pt}}
	{\mkern2.8mu\rule[.5ex]{.35em}{.8pt}}
	{\mkern2.5mu\rule[.50ex]{.7em}{.7pt}}
	{\mkern2mu\rule[.5ex]{.7em}{.5pt}}
}

\newcommand{\pppvint}{\mathop{\mathrlap{\pushpppv}}\!\int}
\newcommand{\pushpppv}{\mathchoice
	{\mkern31.9mu\rule[.6ex]{.5em}{.5pt}}
	{\mkern2.8mu\rule[.5ex]{.35em}{.8pt}}
	{\mkern2.5mu\rule[.50ex]{.7em}{.7pt}}
	{\mkern2mu\rule[.5ex]{.7em}{.5pt}}
}

\def\TagOnRight

\def\AA{{it I} \hskip-3pt{\tt A}}

\def\QQ{\rlap {\raise 0.4ex \hbox{$\scriptscriptstyle |$}} {\hskip -0.1em Q}}

\catcode`\@=11

\def\theequation{\@arabic{\c@section}.\@arabic{\c@equation}}
\renewcommand{\div}{{\mathrm{div}}} 

\newcommand{\DT}{\mathbb{D}} 

\newcommand{\D}[1]{\mathcal{D}(#1)}
\newcommand{\vD}[1]{\bm{\mathcal{D}}(#1)}
\newcommand{\vDsol}[1]{\bm{\mathcal{D}}_\sigma(#1)}
\newcommand{\HC}[2]{\mathcal{C}^{{#1},{#2}}} 
\renewcommand{\L}[1]{L^{#1}(\Omega)}
\newcommand{\Lb}[1]{L^{#1}(\Gamma)}
\newcommand{\vL}[1]{\bm{L}^{#1}(\Omega)}
\newcommand{\vLb}[1]{\bm{L}^{#1}(\Gamma)}
\renewcommand{\H}[1]{H^{#1}(\Omega)}
\newcommand{\vH}[1]{\bm{H}^{#1}(\Omega)}

\newcommand{\vHfracb}[2]{\bm{H}^{\frac{#1}{#2}}(\Gamma)}
\newcommand{\vHfracbd}[2]{\bm{H}^{-\frac{#1}{#2}}(\Gamma)}
\newcommand{\W}[2]{W^{#1,#2}(\Omega)}
\newcommand{\vW}[2]{\bm{W}^{#1,#2}(\Omega)}

\newcommand{\vWz}[2]{\bm{W}^{#1,#2}_{0}(\Omega)}
\newcommand{\Wfracb}[2]{W^{#1,#2}(\Gamma)}
\newcommand{\vWfracb}[2]{\bm{W}^{#1,#2}(\Gamma)}
\newcommand{\vVsolT}[1]{\bm{V}^{#1}_{\sigma,\tau}(\Omega)}
\newcommand{\vE}[1]{\bm{E}^{#1}(\Omega)}

\newcommand{\vn}{\bm{n}} 
\newcommand{\vt}{\bm{\tau}} 
\newcommand{\vu}{\bm{u}} 

\newcommand{\KEYWORDS}[1]{\paragraph{Keywords:}#1\par}

\begin{document}
	
\title{Stokes and Navier-Stokes equations with Navier boundary condition}


\author[1]{P.~Acevedo \thanks{paul.acevedo@epn.edu.ec}}
\author[2]{C.~Amrouche \thanks{cherif.amrouche@univ-pau.fr}}
\author[3]{C.~Conca \thanks{cconca@dim.uchile.cl}}
\author[2,4]{A.~Ghosh \thanks{amrita.ghosh@univ-pau.fr}}

\affil[1]{Escuela Polit\'{e}cnica Nacional, Departamento de Matem\'{a}tica,	Facultad de Ciencias, Ladr\'{o}n de Guevara E11-253, P.O.Box 17-01-2759, Quito, Ecuador.}
\affil[2]{LMAP, UMR CNRS 5142, B\^atiment IPRA, Avenue de l'Universit\'e - BP 1155, 64013 Pau Cedex, France.}
\affil[3]{Departamento de Ingenier\'{i}a Matem\'{a}tica, Facultad de Ciencias F\'{i}sicas y Matem\'{a}ticas, Universidad de Chile, Santiago, Chile.}
\affil[4]{Departamento de Matem\'{a}ticas Facultad de Ciencias y Tecnolog\'{i}a, Universidad del Pa\'{i}s Vasco, Barrio Sarriena s/n, 48940 Lejona, Vizcaya, Spain.}
\maketitle

\begin{abstract}
We study the stationary Stokes and Navier-Stokes equations with non-homogeneous Navier boundary condition in a bounded domain $\Omega\subset\R^{3}$ of class $\HC{1}{1}$. We prove the existence, uniqueness of weak and strong solutions in $\vW{1}{p}$ and $\vW{2}{p}$ for all $1<p<\infty$ considering minimal regularity on the friction coefficient $\alpha$. Moreover, we deduce uniform estimates on the solution with respect to $\alpha$ which enables us to analyze the behavior of the solution when $\alpha \rightarrow \infty$.
\end{abstract}


\tableofcontents

\section{Introduction}

Let $\Omega$ be a bounded domain in $\R^{3}$ with boundary $\Gamma$, possibly not connected, of class $\HC{1}{1}$ (additional smoothness of the boundary will be precised whenever needed). Consider the stationary Stokes equations
\begin{equation}
\label{S'}
-\Delta\vu +\nabla \pi=\bm{f},\quad \div\;\vu=\chi \ \text{ in $\Omega$}
\end{equation}
and the stationary Navier-Stokes equations
\begin{equation}
\label{NS'}
-\Delta\vu +(\vu \cdot\nabla) \vu + \nabla \pi=\bm{f},\quad \div\;\vu=\chi \ \text{ in $\Omega$}
\end{equation}
where $\vu$ and $\pi$ are the velocity field and the pressure of the fluid respectively, $\bm{f}$ is the external force acting on the fluid and $\chi$ stands for the compressibility condition.

Concerning these equations, the first thought goes to the classical no-slip Dirichlet boundary condition
\begin{equation}\label{Dbc}
\vu = \bm{0} \ \text{ on } \Gamma .
\end{equation}
This condition was formulated by G. Stokes in 1845. An alternative was suggested by C.L. Navier \cite{Navier} even before, in 1823. Along with the usual impermeability condition
\begin{equation}
\label{impbc}
\vu\cdot\vn=0 \text{ on } \Gamma
\end{equation}
Navier proposed a slip-with-friction boundary condition and claimed that the component of the fluid velocity tangent to the surface, instead of being zero, should be proportional to the rate of strain at the surface i.e.
\begin{equation}
\label{Nbc}
2\left[(\DT\vu)\vn\right]_{\vt}+\alpha\vu_{\vt}=\bm{0} \ \text{ on $\Gamma$}
\end{equation}
where $\vn$ and $\vt$ are the unit outward normal and tangent vectors on $\Gamma$ respectively and $\DT\vu = \frac{1}{2}(\nabla \vu + \nabla\vu^T)$ is the rate of strain tensor. Here, $\alpha$ is the coefficient which measures the tendency of the fluid to slip on the boundary, called \textit{friction coefficient}.

Although the no-slip hypothesis seems to be in good agreement with experiments, it leads to certain rather surprising conclusions, the most striking one being the absence of collisions of rigid bodies immersed in a linearly viscous fluid (see \cite{H}). In contrast with the no-slip condition, Navier's boundary conditions offer more freedom and are likely to provide a physically acceptable solution at least to some of the paradoxical phenomenons, resulting from the no-slip condition (see for instance, \cite{R}). There have been several attempts in the literature to provide a rigorous justification of the no-slip boundary condition, based on the idea that the physical boundary is never smooth but contains small asperities that drive the fluid to rest under the mere impermeability hypothesis! These kind of results have been shown by Casado-Diaz et al. \cite{CD} in the case of periodically distributed asperities and the references therein. Conversely, \cite{M} identified the Navier' slip boundary condition as a suitable approximation of the behavior of a viscous fluid out of a boundary layer created by the no-slip condition, imposed on a rough boundary. It is worth noting that these results are by far not contradictory but reflect two conceptually different approaches in mathematical modelling of viscous fluids. Some interesting remarks on the use of this boundary condition can be found in Serrin \cite{S}. Recently this boundary condition has also been identified as appropriate for some large Eddy simulation models for turbulent flows, see for instance Galdi and Layton \cite{GL}. Note that, in \cite{Coron}, F. Coron has derived rigorously the slip boundary condition \eqref{Nbc} from the boundary condition at the kinetic level (Boltzmann equation) for compressible fluids. The Navier slip with friction conditions are also used for simulations of flows near rough boundaries, such as in aerodynamics, in weather forecast, in haemodynamics etc.

In this work, we study the existence, uniqueness and regularity of solutions of the system \eqref{S'} and \eqref{NS'} with the boundary condition (\ref{impbc})-\eqref{Nbc}. Skipping the extremely extensive literature of the well-known no-slip boundary condition, we give a brief overview of some of the available results on the Navier/Navier-type boundary conditions. Concerning the non-stationary Navier-Stokes equation with Navier boundary condition, there are considerably many works, among other reasons, for studying the limiting viscosity case. For example, in 2D, the existence of solutions are studied assuming $\alpha$ in $C^2(\Gamma)$ (see Mikeli\'{c} et al \cite{CMR}) and for $\alpha$ in $\Lb{\infty}$ (see Kelliher \cite{K}); whereas in 3D, Beir\~{a}o Da Veiga \cite{BDV7} considered the system with $\alpha$ a positive constant and Iftimie and Sueur \cite{IS} have analysed the case with $\alpha$ in $C^2(\Gamma)$. Also we refer to the work of D. Bucur et al \cite{BFNW} in the periodic boundary case (see the references therein also). We mention here the paper of Monniaux \cite{Monniaux} where they studied the similar boundary condition in Lipschitz domain but with $\alpha$ depending on both time and space variables. On the contrary, for the stationary problem, comparatively less works are known. The first paper concerning basic existence and regularity result is by Solonnikov and Scadilov \cite{SS} where they treated the problem for $\alpha = 0$. They considered the stationary Stokes system with Dirichlet boundary condition on some part and Navier boundary conditions \eqref{Nbc} (with $\alpha=0$) on the other part of the boundary and showed existence of weak solution in $\vH{1}$ which is regular (belongs to $\bm{H}^2_{loc}(\Omega)$) upto some part of the boundary (except in the neighbourhood of the intersection of the two part). Also, it is worth mentioning the work of Beir\~{a}o Da Veiga \cite{BDV4} where he proved existence results of weak and strong solution of the Stokes problem in the $L^2$-settings, but for more generalized system and again with positive constant $\alpha$. Also he did not precise in the estimate, the dependence of the constant on $\alpha$. Recently, Berselli \cite{B} gave some result concerning very weak solution in the special case of a flat domain in $\R^{3}$, in general $\vL{p}$ settings and considering $\alpha = 0$ which is based on the regularity theory of Poisson equation. In the paper of Amrouche and Rejaiba \cite{AR}, they proved the existence and regularity of weak, strong and very weak solutions in a bounded domain in $\R^{3}$ for all $1<p<\infty$ for non-smooth data, but for $\alpha = 0$. In the work of Medkov\'{a} \cite{Medkova}, we can find various other forms of Navier problems and the references therein. Furthermore, the numerical study has been done in, e.g. Verf\"{u}rth \cite{V} (though again for $\alpha = 0$) and John \cite{John}.

Note that, as shown in \eqref{26}, the two boundary conditions
\begin{equation}\label{Navier type}
\mathbf{curl} \ \vu \times \vn = \bm{0}
\end{equation}
and
$$
2\left[(\DT\vu)\vn\right]_{\vt}+\alpha\vu_{\vt}=\bm{0}
$$
are very much similar and in the case of flat boundary and $\alpha = 0$, they are actually equal. Hence, there are several studies concerning this Navier type boundary condition \eqref{Navier type} as well. We refer to \cite{AS1} where Amrouche et al. studied the weak, strong and very weak solutions and the references therein.  

In both cases of stationary and evolution problem, all the available works have considered $\alpha$ as either a constant or a smooth function, as mentioned before. In this work, we analyse the possible minimal regularity of $\alpha$ for the existence of weak and strong solutions in $\L{p}$ for all $1<p<\infty$ [see \eqref{def_exponent_tp_alpha}]. Also, it is worth mentioning that the restriction that $\alpha$ is non-negative is usual, in order to ensure the conservation of energy. But mathematically, we can take into account the negative values of $\alpha$ as well. Some authors have studied the evolution system with $\alpha$ negative where there was no mathematical difficulty due to the access of Gronwall inequality. But in the stationary problem, it is not the case.

Another interesting question we try to answer here is the precise dependence of the solution of (\ref{S}) or (\ref{NS}) on $\alpha$ and if we let the function $\alpha$ tends to $\infty$ or $0$ in \eqref{Nbc}, how does the solution behave? As per the authors' knowledge, there is no previous work on that even if $\alpha$ is smooth function or a constant. We prove the estimates in Theorem \ref{L^2estimate} and Theorem \ref{T2} which show that the solution is uniformly bounded with respect to $\alpha$. The proof of Theorem \ref{T2} is interesting in the sense that it exploits the uniform $\bm{L}^2$-estimate and the observation by Z. Shen that for any $p>2$, $\bm{W}^{1,p}$-estimate for (certain) elliptic equations is equivalent to the \textit{weak reverse h\"{o}lder inequality} (\ref{rhi}). Moreover, in Section \ref{the limiting case}, we show that as $\alpha$ converges to $0$, the solution of the Stokes equation with Navier boundary conditions converges strongly to the solution of the Stokes equation corresponding to $\alpha =0$ and if $\alpha$ goes to $\infty$, the solution converges strongly to the Stokes equation with Dirichlet boundary condition. This type of convergence is, in a sense, an inverse of the derivation of the Navier boundary conditions from no-slip boundary condition for rough boundaries as we have explained previously also. In \cite{Conca}, Conca studied similar system in a bounded domain in $\R^{2}$ with smooth viscosity. There he assumed the well-posedness of the problem \eqref{S'} (or, \eqref{NS'}) with (\ref{impbc})-\eqref{Nbc} and proved some convergence results as $\varepsilon$ goes to $0$, based on homogenization theory, where he considered the domain depends on $\varepsilon$ as well. In some sense, our work generalizes the work in \cite{Conca}. We want to mention that the main purpose of this work is to develop a complete $L^{p}$-theory for all $1<p<\infty$ to deal with the well-posedness for the stationary Stokes and Navier-Stokes equations with full Navier boundary condition, considering general non-regular $\alpha$ and study the limiting cases. Hence, our work can be useful to study the evolution problem and the other related issues; for example, the coupled fluid-structure interaction problems are another interesting open problem. Also the estimates obtained in Section \ref{sec:6sub:2}, precisely in Theorem \ref{P3} can be of independent interest. The corresponding well-posedness and limiting behavior for non-linear system is studied in Section \ref{Navier-Stokes equation}. The main results of our work are mentioned below.

\section{Main results}\label{main results}
\setcounter{equation}{0}

Let us briefly discuss here the main results of our paper, referring to the next sections for precise definitions and complete proofs. Since the case $\alpha \equiv 0$ in (\ref{Nbc}) has already been studied in \cite{AR}, here onwards we consider that $\alpha \not \equiv 0$ on $\Gamma$. If we do not precise otherwise, we will always assume
$$
\alpha \ge 0 \quad \text{ on } \Gamma \quad \text{ and }
\alpha> 0 \quad \text{ on some } \Gamma_0\subset \Gamma \quad \text{ with } |\Gamma_0|>0.
$$
Also we denote as 
\begin{equation}
\label{beta}
\bm{\beta}(x) = \bm{b}\times \bm{x}
\end{equation}
when $\Omega$ is axisymmetric with respect to a constant vector $\bm{b} \in\R^{3}$.

Note that we can always reduce the non vanishing divergence problem
\begin{equation*}
\begin{cases}
-\Delta\vu +\nabla \pi=\bm{f} + \div \ \mathbb{F},\quad \div\;\vu=\chi \ &\text{ in $\Omega$}\\
\vu\cdot\vn=g, \quad \left[(2\DT\vu + \mathbb{F})\vn \right]_{\vt}+\alpha\vu_{\vt}=\bm{h} \ &\text{ on $\Gamma$}
\end{cases}
\end{equation*}
to the case where $\div\ \vu = 0$ in $\Omega$ and $\vu\cdot \vn = 0$ on $\Gamma$, by solving the following Neumann problem
\begin{equation*}
\begin{array}{lr}
\Delta \theta = \chi \quad \text{ in }\Omega, \quad \frac{\partial \theta}{\partial \vn} = g \quad \text{ on }\Gamma
\end{array}
\end{equation*}
and hence using the change of unknowns $\bm{w} = \vu - \nabla \theta$ and $\Pi = \pi - \chi$. Therefore, it is sufficient to study the following Stokes problem:
\begin{equation}
\label{S}
\tag{S}
\begin{cases}
-\Delta\vu +\nabla \pi=\bm{f} + \div \ \mathbb{F},\quad \div\;\vu=0 \ &\text{ in $\Omega$}\\
\vu\cdot\vn=0, \quad \left[(2\DT\vu + \mathbb{F})\vn \right]_{\vt}+\alpha\vu_{\vt}=\bm{h} \ &\text{ on $\Gamma$}.
\end{cases}
\end{equation}
The first main result is the existence and uniqueness of weak and strong solution of the Stokes problem (\ref{S}) which is given in Theorem \ref{thm_weak_sol_Stokes_Nbc} in the Hilbert case and in Corollary \ref{thm_W1p_regularity_Stokes_Nbc} and Theorem \ref{thm_W2p_regularity_Stokes_Nbc} for $p\ne 2$. Note that we proved more general existence result in Theorem \ref{31} than the solution of the problem \eqref{S}.

For that, we need the following assumption on $\alpha$:
\begin{equation}
\label{def_exponent_tp_alpha}
\alpha\in\Lb{t(p)} \quad \text{ with }\quad
\begin{cases}
t(p) = 2 & \text{if\quad} p=2\\
t(p) > 2 & \text{if\quad} \frac{3}{2}\leq p\leq 3, p \neq 2\\
t(p)> \frac{2}{3} \max\{p,p^\prime\} & \text{otherwise}
\end{cases}
\end{equation}
and where $t(p) = t(p')$. Also we will always assume, unless stated otherwise, $\mathbb{F}$ is a $3\times 3$ matrix and $\bm{h}\cdot \vn = 0$ on $\Gamma$ and we will not repeat these hypothesis every time.

\begin{theorem}[{\bf Existence of weak and strong solutions of Stokes problem}]
\label{ch4_teo2}
Let $p\in (1,\infty)$.\\
{\bf (i)} If
$$
\bm{f}\in\vL{r(p)},\,\, \mathbb{F}\in\mathbb{L}^p(\Omega),\,\, \bm{h}\in\vWfracb{-\frac{1}{p}}{p} \text{ and } \alpha\in\Lb{t(p)}
$$
where $t(p)$ as above and $r(p)$ is defined in \eqref{def_exponent_rp}, then the Stokes problem \eqref{S} has a unique solution $(\vu,\pi)\in\vW{1}{p}\times L^{p}_0(\Omega)$.\\
{\bf (ii)} Moreover, if $ \mathbb{F} = 0$ and $$
\bm{f}\in\vL{p},\,\, \bm{h}\in\vWfracb{1-\frac{1}{p}}{p} \text{ and } \alpha \in \Wfracb{1-\frac{1}{q}}{q}
$$
with $q>\frac{3}{2}$ if $p\le \frac{3}{2}$ and $q=p$ otherwise,
then the solution $(\vu,\pi)$ belongs to $\vW{2}{p}\times\W{1}{p}$.
\end{theorem}

Also we obtain some uniform bounds for the weak solution of the problem (\ref{S}) in $\vW{1}{p}$ for all $p\in (1,\infty)$, which we believe are quite interesting on its own.
\begin{theorem}[{\bf Stokes estimates}]
\label{main_result_est}
Let $p\in (1,\infty)$ and $(\vu,\pi)\in\vW{1}{p}\times L^{p}_0(\Omega)$ be the solution of the Stokes problem (\ref{S}). Then it satisfies the following estimates:\\
{\bf (i)} if $\Omega$ is not axisymmetric, then
	\begin{equation*}
	\|\vu\|_{\vW{1}{p}} + \|\pi\|_{L^{p}(\Omega)} \leq C_p(\Omega) \left( \|\bm{f}\|_{\vL{r(p)}} + \|\mathbb{F}\|_{\mathbb{L}^p(\Omega)} + \|\bm{h}\|_{\vWfracb{-\frac{1}{p}}{p}} \right) .
	\end{equation*}
	\medskip
{\bf (ii)} if $\Omega$ is axisymmetric and $\alpha \geq \alpha_* >0$, then
	\begin{equation*}
	\|\vu\|_{\vW{1}{p}} + \|\pi\|_{L^{p}(\Omega)} \leq \frac{C_p(\Omega)}{\min\{2,\alpha_*\}} \left( \|\bm{f}\|_{\vL{r(p)}} + \|\mathbb{F}\|_{\mathbb{L}^p(\Omega)}+ \|\bm{h}\|_{\vWfracb{-\frac{1}{p}}{p}} \right) .
	\end{equation*}
\end{theorem}

In Theorem \ref{L^2estimate} and Theorem \ref{13}, we discuss in detail the estimates of weak and strong solutions in the Hilbert case and the proof of Theorem \ref{main_result_est} for general $p$ is done in Theorem \ref{T2}.

The next theorem gives existence of weak and strong solutions for the following Navier-Stokes problem and some estimates.

\begin{equation}
\label{NS}
\tag{NS}
\begin{cases}
-\Delta\vu + \vu\cdot \nabla \vu+ \nabla \pi=\bm{f} + \div \ \mathbb{F},\quad \div\;\vu=0 & \text{in $\Omega$,}\\
\vu\cdot\vn=0, \qquad \ \left[(2\DT\vu + \mathbb{F})\vn\right]_{\vt}+\alpha\vu_{\vt}=\bm{h}& \text{on $\Gamma$.}
\end{cases}
\end{equation} 

\begin{theorem}[{\bf Existence of weak and strong solutions of Navier-Stokes problem and estimates}]
Let $p\in(\frac{3}{2},\infty)$ and
$$
\bm{f}\in \vL{r(p)},\,\, \mathbb{F} \in \mathbb{L}^p(\Omega),\,\, \bm{h}\in \vWfracb{-\frac{1}{p}}{p} \text{ and } \alpha\in \Lb{t(p)}.
$$
{\bf 1.} Then the problem \eqref{NS} has a solution $\left( \vu,\pi\right) \in\vW{1}{p}\times L^p_0(\Omega)$.\\
{\bf 2.} Also for any $p\in (1,\infty)$, if $\mathbb{F} = 0$ and
$$
\bm{f}\in\vL{p},\,\, \bm{h}\in\vWfracb{1-\frac{1}{p}}{p} \text{ and } \alpha\in \Wfracb{1-\frac{1}{q}}{q}
$$
with $q>\frac{3}{2}$ if $p\le \frac{3}{2}$ and $q=p$ otherwise, then $(\vu,\pi)\in\vW{2}{p}\times\W{1}{p}$.\\
{\bf 3.} For $p=2$, the weak solution $\left( \vu,\pi\right)\in \vH{1}\times L^2_0(\Omega)$ satisfies the following estimates:\\

{\bf a)} if $\Omega$ is not axisymmetric, then
\begin{equation*}
\|\vu\|_{\vH{1}}+\|\pi\|_{\L{2}}\leq{C(\Omega)}\left(\|\bm{f}\|_{\vL{\frac{6}{5}}}+ \|\mathbb{F}\|_{\mathbb{L}^2(\Omega)}+\|\bm{h}\|_{\vHfracbd{1}{2}}\right) .
\end{equation*}

{\bf b)} if $\Omega$ is axisymmetric and\\
\indent $\left(\textit{i} \right)$ $\alpha \geq \alpha_* >0$ on $\Gamma$, then
\begin{equation*}
\|\vu\|_{\vH{1}}+\|\pi\|_{L^{2}(\Omega)}\leq \frac{C(\Omega)}{\min\{2,\alpha_*\}}\left(\|\bm{f}\|_{\vL{\frac{6}{5}}}+ \|\mathbb{F}\|_{\mathbb{L}^2(\Omega)}+\|\bm{h}\|_{\vHfracbd{1}{2}}\right) .
\end{equation*}
\indent $\left(\textit{ii} \right)$ $\bm
{f}, \mathbb{F}$ and $\bm{h}$ satisfy the condition:
\begin{equation*}
\int\displaylimits_{\Omega}{ \bm{f} \cdot \bm{\beta}} -\int\displaylimits_{\Omega}{\mathbb{F}:\nabla \bm{\beta}}+ \left\langle \bm{h}, \bm{\beta}\right\rangle_\Gamma = 0
\end{equation*}
\qquad then, the solution $\vu$ satisfies $\int\displaylimits_{\Gamma}{\alpha \vu \cdot \bm{\beta}} = 0$ and
\begin{equation*}
\|\DT\vu\|_{\mathbb{L}^2(\Omega)}^2+\int\displaylimits_{\Gamma}{\alpha |\vu_{\vt}|^2} + \|\pi\|^2_{\L{2}}\leq C(\Omega)\left(\|\bm{f}\|_{\vL{\frac{6}{5}}}+ \|\mathbb{F}\|_{\mathbb{L}^2(\Omega)}+\|\bm{h}\|_{\vHfracbd{1}{2}}\right)^2 .
\end{equation*}
\indent In particular, if $\alpha$ is a constant, then $\int\displaylimits_{\Gamma}{\vu \cdot \bm{\beta}} = 0$ and
\begin{equation*}
\|\vu\|_{\vH{1}}+\|\pi\|_{\L{2}}\leq{C(\Omega)}\left(\|\bm{f}\|_{\vL{\frac{6}{5}}}+ \|\mathbb{F}\|_{\mathbb{L}^2(\Omega)}+\|\bm{h}\|_{\vHfracbd{1}{2}}\right) .
\end{equation*}
\end{theorem}

We refer to Theorem \ref{39} and Corollary \ref{reg_NS} for the proof of above result, which is similar to that of Stokes problem. Note that in the above theorem, the existence of weak solution in $\vW{1}{p}$ for $\frac{3}{2}<p<2$ is not trivial.

The last interesting result to mention here, is the strong convergence of \eqref{NS} to the Navier-Stokes equations with no-slip boundary condition when $\alpha$ grows large (see Theorem \ref{NS inf}). The proof is essentially based on the estimates obtained above.

\begin{theorem}[{\bf Limit case for Navier-Stokes problem}]
Let $p\ge 2$, $\alpha$ be a constant and $(\vu_\alpha, \pi_\alpha)$ be a solution of \eqref{NS} where
	$$
	\bm{f}\in\vL{r(p)},\,\, \mathbb{F}\in\mathbb{L}^p(\Omega) \text{ and } \bm{h}\in \vWfracb{-\frac{1}{p}}{p}.
	$$
	Then
	$$
	(\vu_\alpha, \pi _\alpha) \rightarrow (\vu_\infty, \pi_\infty) \quad \text{ in } \quad \vW{1}{p}\times L^{p}_0(\Omega) \quad \text{ as } \quad \alpha \rightarrow \infty
	$$
	where $(\vu_\infty,\pi_\infty)$ is a solution of the Navier-Stokes problem with Dirichlet boundary condition,
	\begin{equation*}
	\begin{aligned}
	\begin{cases} 
	-\Delta\vu_\infty + \vu_\infty \cdot \nabla \vu_\infty + \nabla \pi_\infty=\bm{f} +\div \ \mathbb{F} &\text{ in } \ \Omega ,\\
	\div\;\vu_\infty=0 &\text{ in } \ \Omega ,\\
	\vu_\infty =\bm{0} &\text{ on } \ \Gamma .
	\end{cases}
	\end{aligned}
	\end{equation*}
\end{theorem}

\section{Notations and preliminary results}
\label{sec_not-use-res}
\setcounter{equation}{0}

Before studying the problem \eqref{S'} and (\ref{NS'}) with (\ref{impbc})-\eqref{Nbc}, we review some basic notations and functional framework. We will use the term \textit{axisymmetric} to mean a non-empty set which is generated by rotation around an axis. The vector fields and matrix fields (and the corresponding spaces) defined over $\Omega$ or over $\R^{3}$ are denoted by bold font and blackboard bold font respectively. Unless otherwise stated, we follow the convention that $C$ is an unspecified positive constant that may vary from expression to expression, even across an inequality (but not across an equality). Also $C$ depends on $\Omega$ generally and the dependence of $C$ on other parameters will be specified within parenthesis when necessary. 

Note that the vector-valued Laplace operator of a vector-field $\bm{v} = (v_1,v_2,v_3)$ is equivalently defined as
$$\Delta \bm{v} = 2 \ \div \ \DT\bm{v} - \textbf{grad} \ \div \ \bm{v}.$$
We denote by $\vD{\Omega}$ the set of smooth functions (infinitely differentiable) with compact support in $\Omega$. Define
$$\vDsol{\Omega} := \left\{ \bm{v}\in \vD{\Omega}; \ \div \ \bm{v} = 0 \text{ in } \Omega \right\}$$
and
$$
L^p_0(\Omega) := \left\{ v\in \L{p}; \int\displaylimits_{\Omega}{v}= 0 \right\}.
$$
If $p\in [1,\infty)$, $p'$ denotes the conjugate exponent of $p$ i.e. $\frac{1}{p}+\frac{1}{p'}=1$. For $p,r\in [1,\infty)$, we introduce the following space
$$\bm{H}^{r,p}(\div, \Omega) := \left\{ \bm{v}\in \bm{L}^r(\Omega); \ \div \ \bm{v}\in L^p(\Omega)\right\} $$
equipped with the norm
$$\|\bm{v}\|_{\bm{H}^{r,p}(\div, \Omega)} = \|\bm{v}\|_{\bm{L}^r(\Omega)} + \|\div \ \bm{v}\|_{L^p(\Omega)}.$$
It can be shown that $\vD{\overline{\Omega}}$ is dense in $\bm{H}^{r,p}(\div, \Omega)$ (cf. \cite[Proposition 1.0.2]{Nour}). The closure of $\vD{\Omega}$ in $\bm{H}^{r,p}(\div, \Omega)$ is denoted by $\bm{H}^{r,p}_0(\div, \Omega)$ and can be characterized as
$$\bm{H}^{r,p}_0(\div, \Omega) = \left\{\bm{v}\in \bm{H}^{r,p}(\div, \Omega); \ \bm{v}\cdot \bm{n}=0 \ \text{ on } \ \Gamma\right\}.$$
Also for $p\in (1,\infty)$, the dual space of $\bm{H}^{r,p}_0(\div, \Omega)$, denoted by $[\bm{H}^{r,p}_0(\div, \Omega)]^\prime$, can be characterized as follows (cf. \cite[Proposition 1.0.4]{Nour}):

\begin{proposition}
\label{40}
A distribution $\bm{f}$ belongs to $[\bm{H}^{r,p}_0(\div, \Omega)]^\prime$ iff there exists $\bm{\psi} \in \bm{L}^{r^\prime}(\Omega)$ and $\chi\in L^{p^\prime}(\Omega)$ such that $\bm{f} = \bm{\psi}+ \nabla \chi$. Moreover, we have the estimate :
$$\|\bm{f}\|_{[\bm{H}^{r,p}_0(\div, \Omega)]^\prime} \leq \underset{\bm{f} = \bm{\psi}+ \nabla \chi}{\inf} \ \max\{\|\bm{\psi}\|_{\bm{L}^{r^\prime}(\Omega)}, \|\chi\|_{L^{p^\prime}(\Omega)}\} .$$
\end{proposition}

\noindent We also recall the following useful result (cf. \cite[Theorem 3.5]{AS}):
\begin{proposition}
\label{X}
Let $\bm{v} \in \vL{p}$ with $\div \ \bm{v}\in \L{p}$, $\mathbf{curl} \ \bm{v}\in \vL{p}$ and $\bm{v}\cdot \vn \in \Wfracb{1-\frac{1}{p}}{p}$. Then $\bm{v}\in \vW{1}{p}$ and satisfies the estimate:
$$\|\bm{v}\|_{\vW{1}{p}} \leq C \left( \|\bm{v}\|_{\vL{p}}+\|\mathbf{curl} \ \bm{v}\|_{\vL{p}}+\|\div \ \bm{v}\|_{\L{p}}+\|\bm{v}\cdot \vn \|_{\Wfracb{1-\frac{1}{p}}{p}}\right) . $$
\end{proposition}
\noindent
We need to introduce the following spaces also :
$$\vVsolT{p}:=\left\{\bm{v}\in\vW{1}{p}; \ \div\;\bm{v}=0 \text{\; in $\Omega \ $\;and\;} \ \bm{v}\cdot\vn=0\text{\; on $\Gamma$}\right\}$$
equipped with the norm of $\vW{1}{p}$,
$$
\bm{H}^1_{\vt}(\Omega):= \left\lbrace \bm{v}\in\vH{1}; \ \bm{v}\cdot\vn=0\text{\; on $\Gamma$}\right\rbrace
$$
 and
\begin{equation}
\label{Ep}
\vE{p}:=\left\{\left(\bm{v}, \pi\right) \in\vW{1}{p}\times \L{p}; \ -\Delta\bm{v} + \nabla \pi\in\vL{r(p)}\right\}
\end{equation}
where
\begin{equation}
\label{def_exponent_rp}
\begin{cases}
r(p)= \max \left\lbrace 1,\frac{3p}{p+3} \right\rbrace  &\text{ if } \quad p \ne \frac{3}{2}\\
r(p)>1 &\text{ if } \quad p=\frac{3}{2}
\end{cases}
\end{equation}
which is a Banach space with the norm
$$\|\left( \bm{v}, \pi\right) \|_{\vE{p}} := \|\bm{v}\|_{\vW{1}{p}} + \|\pi\|_{\L{p}} + \|-\Delta\bm{v} + \nabla\pi\|_{\vL{r(p)}} .$$

Let us now introduce some notations to describe the boundary. Consider any point $P$ on $\Gamma$ and choose an open neighbourhood $W$ of $P$ in $\Gamma$, small enough to allow the existence of 2 families of $\mathcal{C} ^2$ curves on $W$ with these properties : a curve of each family passes through every point of $W$ and the unit tangent vectors to these curves form an orthogonal system (which we assume to have the direct orientation) at every point of $W$. The lengths $s_1, s_2$ along each family of curves, respectively, are a possible system of coordinates in $W$. We denote by $\bm{\tau}_1, \bm{\tau}_2$ the unit tangent vectors to each family of curves.

With this notations, we have $\bm{v}= \sum_{k=1}^{2}v_k \bm{\tau}_
{k} + (\bm{v}\cdot\vn)\vn$ where $\bm{\tau}_k = (\tau_{k1}, \tau_{k2}, \tau_{k3})$ and $v_k = \bm{v}\cdot \bm{\tau}_k$. In the sequel, for simplicity of notation, we will use
$$\bm{\Lambda} \bm{v} = \sum_{k=1}^2 \left( \bm{v}_{\bm{\tau}}\cdot \frac{\partial \vn}{\partial s_k}\right) \bm{\tau}_k .$$

We recall the following relations which give the equivalence of the two boundary conditions \eqref{Nbc} and \eqref{Navier type} and which will be used extensively to prove some of our main results (for proof, see \cite[Appendix A]{AR}). Note that $\Omega$ being $\HC{1}{1}$ is sufficient and there is a sign change in the second relation, compared to \cite{AR} and it is the corrected formulation.

\begin{lemma}
For any $\bm{v} \in \vW{2}{p}$, we have the following equalities:
$$2\left[(\DT\bm{v})\vn\right]_{\vt} = \nabla_{\vt}(\bm{v}\cdot \vn) + \left( \frac{\partial \bm{v}}{\partial \vn}\right) _{\vt} - \bm{\Lambda}\bm{v} \ ,$$
$$\mathbf{curl} \ \bm{v} \times \vn = - \nabla_{\vt} (\bm{v}\cdot \vn)+ \left( \frac{\partial \bm{v}}{\partial \vn}\right) _{\vt} + \bm{\Lambda}\bm{v} \ .$$
\end{lemma}

\begin{remark}
\label{EF}
\rm{In the particular case $\bm{v}\cdot \vn = 0$ on $\Gamma$, we obtain, for all $\bm{v} \in \vW{2}{p}$,
$$2\left[(\DT\bm{v})\vn\right]_{\vt} = \left( \frac{\partial \bm{v}}{\partial \vn}\right) _{\vt} - \bm{\Lambda}\bm{v} \quad \text{ and } \quad
\mathbf{curl} \ \bm{v} \times \vn = \left( \frac{\partial \bm{v}}{\partial \vn}\right) _{\vt} + \bm{\Lambda}\bm{v} $$
which implies that 
\begin{equation}\label{26}
2\left[(\DT\bm{v})\vn\right]_{\vt} = \mathbf{curl} \ \bm{v} \times \vn - 2\bm{\Lambda}\bm{v} .
\end{equation}
}
\end{remark}

Next, we prove the following Green formula to define the trace of the strain tensor of a vector field.

\begin{lemma}
\label{lem_trace_strain_tensor}
Let $\Omega$ be Lipschitz. Then,\\
\noindent{\bf (i)} $\vD{\overline\Omega}\times \D{\overline{\Omega}}$ is dense in $\vE{p}$ and\\
\noindent{\bf (ii)} The linear mapping $\left( \bm{v},\pi\right) \mapsto\left[(\DT\bm{v})\vn\right]_{\vt}$, defined on $\vD{\overline\Omega}\times \D{\overline{\Omega}}$ can be extended to a linear, continuous map from $\vE{p}$ to $\vWfracb{-\frac{1}{p}}{p}$. Moreover we have the following relation: for all $\left( \bm{v},\pi\right) \in\vE{p}$ and $\bm{\varphi}\in\vVsolT{p'},$
\begin{equation}
\label{green}
 \int\displaylimits_{\Omega}{\left( -\Delta\bm{v} + \nabla\pi\right) \cdot\bm{\varphi}} =2\int\displaylimits_{\Omega}{\DT\bm{v}:\DT\bm{\varphi}}-2
\left\langle \left[(\DT\bm{v})\vn\right]_{\vt},{\bm{\varphi}}\right\rangle_{\vWfracb{-\frac{1}{p}}{p}\times\vWfracb{\frac{1}{p}}{p'}}.
\end{equation}
\end{lemma}

\begin{proof}
{\bf (i)} The proof of the density result is very much similar to \cite[Lemma 4.2.1]{Nour}. Let $P:\vW{1}{p}\rightarrow\bm{W}^{1,p}(\R^{3})$ be the extension operator such that $P\vu\big|_{\Omega}=\vu$. Then for all $\bm{\ell}\in[\vE{p}]'$, there exists $(\bm{\varphi},\lambda,\bm{\psi})\in\bm{W}^{-1,p'}(\R^{3})\times\L{p^\prime}\times\vL{(r(p))^\prime}$ such that for any $\left( \bm{v},\pi\right) \in\vE{p}$,
$$\langle{\bm{\ell},\left( \bm{v},\pi\right) }\rangle=\langle\bm{\varphi},P\bm{v}\rangle_{\bm{W}^{-1,p'}(\R^{3})\times \bm{W}^{1,p}(\R^{3})}+ \int\displaylimits_{\Omega}{\lambda \pi}+\int\displaylimits_{\Omega}{\bm{\psi} \cdot \left( -\Delta\bm{v}+\nabla \pi\right)} .$$
Thanks to the Hahn-Banach theorem, it suffices to show that any $\bm{\ell}$ which vanishes on $\vD{\overline\Omega}\times \D{\overline{\Omega}}$ is actually zero on $\vE{p}$.

Let us suppose that $\bm{\ell}=\bm{0}$ in $\vD{\overline\Omega}\times \D{\overline{\Omega}}$ and let $\widetilde{\lambda}\in L^{p^\prime}(\R^{3})$ and  $\widetilde{\bm{\psi}}\in\bm{L}^{(r(p))^\prime}(\R^{3})$ be the
extension by zero to $\R^{3}$. Then for all $\left( \bm{V}\times \Pi\right) \in\vD{\R^{3}}\times \D{\R^{3}}$, let $\bm{v}=\bm{V}\big|_{\Omega}$ and $\pi = \Pi\big|_{\Omega} $ so that $\left( \bm{v},\pi\right) \in \vD{\overline\Omega}\times \D{\overline{\Omega}}$ and
$$\langle\bm{\varphi},\bm{V}\rangle_{\bm{W}^{-1,p'}(\R^{3})\times \bm{W}^{1,p}(\R^{3})}+\int_{\R^{3}}\widetilde{\lambda} \Pi +\int_{\R^{3}}\widetilde{\bm{\psi}}\cdot\left( -\Delta\bm{V}+\nabla \Pi\right) \mathrm{\;d}x
=0$$
since $\langle\bm{\varphi},\bm{V}\rangle=\langle\bm{\varphi},P\bm{v}\rangle$. It then follows that
\begin{equation*}
\left\langle \bm{\varphi}-\Delta\widetilde{\bm{\psi}},\bm{V} \right\rangle _{\bm{\mathcal{D}}^\prime(\R^3)\times \vD{\R^{3}}} = 0 \quad \text{ and } \quad \left\langle \widetilde{\lambda}- \div \ \widetilde{\bm{\psi}}, \Pi\right\rangle _{\mathcal{D}^\prime(\R^{3})\times \D{\R^{3}}} = 0 .
\end{equation*}
i.e. $\bm{\varphi}-\Delta\widetilde{\bm{\psi}}=\bm{0}$ and $\widetilde{\lambda}- \div \ \widetilde{\bm{\psi}} = 0$
in the sense of distribution in $\R^{3}$. Hence, $\Delta\widetilde{\bm{\psi}}\in\bm{W}^{-1,p'}(\R^{3})$. As a consequence, $\widetilde{\bm{\psi}}\in\bm{W}^{1,p'}(\R^{3})$ and therefore $\bm{\psi}\in\vWz{1}{p'}$. Then by density of $\vD{\Omega}$ in $ \vWz{1}{p'}$, there exists a sequence $(\bm{\psi}_k)_k\subset\vD{\Omega}$ such that $\bm{\psi}_k\rightarrow\bm{\psi}$ as $k\rightarrow\infty$ in $\vW{1}{p'}$. Also $\Delta \widetilde{\bm{\psi}}_k \rightarrow \Delta \widetilde{\bm{\psi}}$ in $\bm{W}^{-1.p^\prime}(\R^{3})$.
Now, for any $\left( \bm{v},\pi\right) \in\vE{p},$ we have,
\begin{align*}
& \quad \left\langle \bm{\ell},\left( \bm{v},\pi\right) \right\rangle\\
&
=\langle\bm{\varphi},P\bm{v}\rangle_{\bm{W}^{-1,p'}(\R^{3})\times \bm{W}^{1,p}(\R^{3})}+\int\displaylimits_{\Omega}{\lambda \pi} +\int\displaylimits_{\Omega}{ \bm{\psi}\cdot\left( -\Delta\bm{v}+\nabla\pi\right) }\\
& = \langle\Delta\widetilde{\bm{\psi}},P\bm{v}\rangle_{\bm{W}^{-1,p'}(\R^{3})\times \bm{W}^{1,p}(\R^{3})}+ \int\displaylimits_{\Omega}{\pi \ \div \ \bm{\psi} }+\left\langle \bm{\psi}, \left( -\Delta\bm{v}+ \nabla \pi\right) \right\rangle _{\vWz{1}{p^\prime}\times \vW{-1}{p}}\\
&=\lim_{k\rightarrow\infty}[\langle\Delta\widetilde{\bm{\psi}}_k,P\bm{v}\rangle_{\bm{W}^{-1,p'}(\R^{3})\times \bm{W}^{1,p}(\R^{3})} + \int\displaylimits_{\Omega}{\pi \ \div \ \bm{\psi}_k} \ +\\
& \qquad \quad + \left\langle \bm{\psi}_k,\left( -\Delta\bm{v}+\nabla\pi\right)\right\rangle_{\vWz{1}{p^\prime}\times \vW{-1}{p}}]\\
& = \lim_{k\rightarrow\infty} [\left\langle \Delta \bm{\psi}_k, \bm{v}\right\rangle_{\vD{\Omega}\times \bm{\mathcal{D}}^\prime(\Omega)} - \left\langle \bm{\psi}_k, \nabla\pi \right\rangle _{\vD{\Omega}\times \bm{\mathcal{D}}^\prime(\Omega)} + \left\langle \bm{\psi}_k, -\Delta\bm{v} +\nabla \pi\right\rangle  _{\vD{\Omega}\times \bm{\mathcal{D}}^\prime(\Omega)} ]\\
&=0 \ .
\end{align*}
Thus $\bm{\ell}$ is identically zero.

{\bf (ii)} The Green formula for $(\bm{v},\pi)\in \vD{\overline\Omega}\times \D{\overline{\Omega}}$ and $\bm{\varphi}\in \vVsolT{p'}$ follows immediately by integration by parts and then we use the density result (cf. \cite[Lemma 2.4]{AR}).
\hfill
\end{proof}

\begin{remark}
\rm{
{\bf 1.} The following Green formula also can be obtained in the same way as (\ref{green}), which will be used later: for $(\bm{v},\pi)\in \vW{1}{p}\times \vL{p}$, $\mathbb{F}\in\mathbb{L}^p(\Omega)$ such that $-\div(2\DT \bm{v}+\mathbb{F}) + \nabla\pi \in \vL{r(p)}$ and $ \bm{\varphi}\in \vVsolT{p'}$,
\begin{equation}
\label{green1}
 \int\displaylimits_{\Omega}{\left( -\div(2\DT \bm{v}+\mathbb{F}) + \nabla\pi\right) \cdot\bm{\varphi}} =2\int\displaylimits_{\Omega}{\DT\bm{v}:\DT\bm{\varphi}}+\int\displaylimits_{\Omega}{\mathbb{F}:\nabla \bm{\varphi}}-
\left\langle \left[(2\DT\bm{v}+\mathbb{F})\vn\right]_{\vt},{\bm{\varphi}}\right\rangle_{\Gamma}.
\end{equation}

{\bf 2.} In fact, we can obtain Lemma \ref{lem_trace_strain_tensor} for any $\bm{v}\in \vE{p}$ where
$$\vE{p}:=\left\{\bm{v} \in\vW{1}{p}; \ \Delta\bm{v} \in[\bm{H}^{r(p)',p'}_0(\div, \Omega)]^\prime\right\} .$$
Thus we can extend \eqref{26} in $\vWfracb{-\frac{1}{p}}{p}$ as follows: let $\Omega$ be $\HC{1}{1}$ and for any $\bm{v}\in \vW{1}{p}$ with $\Delta\bm{v} \in\vL{r(p)}$ and $\bm{v}\cdot \vn = 0$ on $\Gamma$,
\begin{equation}
2\left[(\DT\bm{v})\vn\right]_{\vt} = \mathbf{curl} \ \bm{v} \times \vn - 2\bm{\Lambda}\bm{v} \quad \text{ in } \quad \vWfracb{-\frac{1}{p}}{p} .
\end{equation}}
\end{remark}

The following two propositions give some Korn-type inequalities which will be useful in the context. '$\simeq$' denotes the equivalence of two norms.
\begin{proposition}
\label{prop_equivalent_norms_c2}
	Let $\Omega$ be Lipschitz. Then, for all $\vu\in\vH{1}$ with $\vu\cdot\vn=0$ on $\Gamma$, we have the following equivalence of norms:
\begin{equation}
\label{eqn_equivalence_norms_c1}
	\|\vu\|_{\vH{1}}\simeq\|\DT\vu\|_{\mathbb{L}^2(\Omega)} \quad \text{ if \ $\Omega$ is not axisymmetric} \ ,
\end{equation}
	and
\begin{equation}
\label{eqn_equivalence_norms_c2}
	\|\vu\|_{\vH{1}}\simeq\|\DT\vu\|_{\mathbb{L}^2(\Omega)}+\|\vu_{\vt}\|_{\bm{L}^2(\Gamma)} \quad \text{ if \ $\Omega$ is axisymmetric} \ .
\end{equation}
\end{proposition}

\begin{proof}
Since \eqref{eqn_equivalence_norms_c1} follows from \cite[Lemma 3.3]{AR}, we only prove \eqref{eqn_equivalence_norms_c2}. Also, it is enough to show that there exists $C>0$ such that
$$\|\vu\|_{\vH{1}}\leq C\left(\|\DT\vu\|_{\mathbb{L}^2(\Omega)}+\|\vu_{\vt}\|_{\bm{L}^2(\Gamma)}\right)$$
as the other inequality is obvious.
	
We prove by contradiction. Suppose for any $m\in\N$, there exists $\vu_m\in\vH{1}$ such that $\vu_m\cdot\vn=0$ on $\Gamma$ and
	\begin{equation}
	\label{ineq_hip_reduction}
	\|\vu_m\|_{\vH{1}}> m\left(\|\DT\vu_m\|_{\mathbb{L}^2(\Omega)}+\|(\vu_m)_{\vt}\|_{\bm{L}^2(\Gamma)}\right).
	\end{equation}
Since $\|\vu_m\|_{\vH{1}}>0$, we can define $\bm{v}_m:=\frac{\vu_m}{\|\vu_m\|_{\vH{1}}}$ so that $\|\bm{v}_m\|_{\vH{1}}=1$ for all $m\in\N$. Then from \eqref{ineq_hip_reduction}, we have
	$$
	\|\DT\bm{v}_m\|_{\mathbb{L}^2(\Omega)}+\|(\bm{v}_m)_{\vt}\|_{\bm{L}^2(\Gamma)}<\frac{1}{m}.
	$$
Hence, as $m\rightarrow\infty$,
	\begin{equation*}
	\label{cv_DTvn_L2}
	\DT\bm{v}_m\ {\rightarrow}\ 0 \ \text{ in }\vL{2} \quad \text{ and } \quad (\bm{v}_m)_{\vt}\ {\rightarrow}\ \bm{0} \ \text{ in }\bm{L}^2(\Gamma).
	\end{equation*}
But as $\bm{v}_m\cdot\vn=0$ on $\Gamma$, we get	$\bm{v}_m\ {\rightarrow}\ \bm{0}\text{\;\;in\;\;}\bm{L}^2(\Gamma)$.
On the other hand, since $\{\bm{v}_m\}_m$ is bounded in $\vH{1}$, there exists a subsequence, which we still call $\{\bm{v}_m\}_m$ and $\bm{v}\in\vH{1}$ such that $\bm{v}_{m}\rightharpoonup\bm{v} \ \text{ weakly in } \vH{1}$. Thus
	$$\DT\bm{v}=0 \ \text{ in } \Omega \quad \text{ and } \quad \bm{v}=\bm{0} \ \text{ on } \Gamma$$
which yields $\bm{v}=\bm{0}$ in $\Omega$. But this is a contradiction since from Korn inequality, we have
\begin{equation*}
1 = \|\bm{v}_m\|_{\vH{1}} \le C \left( \|\bm{v}_m\|_{\vL{2}} + \|\DT \bm{v}_m\|_{\mathbb{L}^2(\Omega)}\right) \rightarrow \bm{0}.
\end{equation*}
Thus (\ref{eqn_equivalence_norms_c2}) follows.
\hfill
\end{proof}

\begin{proposition}
	Let $\Omega$ be Lipschitz. For $\Omega$ axisymmetric, we have the following inequalities: for all $\vu\in\vH{1}$ with $\vu\cdot \vn = 0$ on $\Gamma$,
	\begin{equation}\label{43}
	\|\vu\|^2_{\vL{2}} \leq C \left[ \|\DT\vu\|^2_{\mathbb{L}^2(\Omega)}+ \left( \int\displaylimits_{\Omega}{ \vu\cdot \bm{\beta}}\right)^2\right]
	\end{equation}
	and
	\begin{equation}\label{44}
	\|\vu\|_{\vL{2}}^2 \leq C \left[ \|\DT\vu\|_{\mathbb{L}^2(\Omega)}^2 + \left(\int\displaylimits_{\Gamma}{\vu\cdot \bm{\beta}}\right)^2 \right]  .
	\end{equation}
\end{proposition}

\begin{proof}
{\bf (i)}	First recall from (\ref{beta}) that $\bm{\beta} \in \mathcal{C}^\infty (\R^{3})$ and $\DT\bm{\beta} = 0$ in $\R^{3}$. Then \eqref{43} follows from the following result \cite[Lemma 3.3]{AR}:
	\begin{equation}
	\label{55}
	\inf\limits_{\bm{w}\in \bm{\mathcal{T}}(\Omega)} \|\vu + \bm{w} \|^2_{\vL{2}} \leq C(\Omega) \left( \|\DT\vu\|^2_{\mathbb{L}^2(\Omega)} + \int\displaylimits_{\Gamma}{\left| \vu\cdot \vn \right| ^2}\right) .
	\end{equation}
	Since, $\bm{w} = c\bm{\beta}$ for some $c\in \R$, $\inf\limits_{\bm{w}\in \bm{\mathcal{T}}(\Omega)} \|\vu + \bm{w} \|^2_{\vL{2}} = \inf\limits_{c \ \in \ \R} \ \|\vu + c\bm{\beta} \|^2_{\vL{2}}$ and this infimum is attained at
	$$ c =  \frac{1}{\|\bm{\beta}\|^2_{\vL{2}}} \left( \int\displaylimits_{\Omega}{\vu\cdot \bm{\beta}}\right) .$$
	Now,
	\begin{align*}
	& \quad \left\|  \vu - \frac{1}{\|\bm{\beta}\|
		^2_{\vL{2}}} \left( \int\displaylimits_{\Omega}{\vu\cdot \bm{\beta}}\right) 
	\bm{\beta}\right\| ^2_{\vL{2}}\\
	& = \|\vu\|^2_{\vL{2}} - \frac{2}{\|\bm{\beta}\|
		^2_{\vL{2}}}\left( \int\displaylimits_{\Omega}{\vu\cdot \bm{\beta}}\right) ^2 + 
	\frac{1}{\|\bm{\beta}\|^2_{\vL{2}}}\left( \int\displaylimits_{\Omega}{\vu\cdot 
		\bm{\beta}}\right) ^2\\
	& = \|\vu\|^2_{\vL{2}} - \frac{1}{\|\bm{\beta}\|
		^2_{\vL{2}}}\left( \int\displaylimits_{\Omega}{\vu\cdot \bm{\beta}}\right) ^2 .
	\end{align*}
	Hence, we obtain \eqref{43}.\\
	
{\bf (ii)} Now we prove the inequality \eqref{44}, by contradiction.
Let us denote
$$
|||\vu||| := \|\vu\|_{\vL{2}} + \|\DT\vu\|_{\mathbb{L}^2(\Omega)}.
$$
Now assume that for all $m\in\N$, there exists $\vu_m\in\vH{1}$ 
	with $\vu_m\cdot\vn=0$ on $\Gamma$ and $|||\vu_m|||= 1$ such that
	\begin{equation}
	\label{47}
	1> m \left[ \|\DT\vu_m\|_{\mathbb{L}^2(\Omega)}^2 + \left( 
	\int\displaylimits_{\Gamma}{\vu_m \cdot \bm{\beta}}\right) ^2\right] .
	\end{equation}
	So $\{\vu_m\}_m$ is a bounded sequence in $\vH{1}$; Hence 
	there exists a subsequence, we still call $\{\vu_m\}_m$ and 
	$\vu$ in $\vH{1}$ so that $\vu_m \rightharpoonup \vu$ in $\vH{1}$.
	This gives $\vu\cdot \vn = 0$ on $\Gamma$ and $\vu_m 
	\rightarrow \vu$ in $\vL{2}$.
	But from \eqref{47}, we have
	\begin{equation*}
	\DT\vu_m \rightarrow 0 \ \text{ in } \ \vL{2} \quad \text{ and } \quad 	\int\displaylimits_{\Gamma}{\vu_m \cdot \bm{\beta}} \rightarrow 0.
	\end{equation*}
	Then $\DT\vu = 0$ in $\Omega$ which implies $\vu = c\bm{\beta}$ for some $c\in\R$. But also $\vu_m \rightharpoonup \vu$ in $\vHfracb{1}{2}$ and $\vHfracb{1}{2}$ is compactly embedded in $\vLb{2}$ implies $\vu_m \rightarrow \vu$ in $ \vLb{2}$. Therefore, we have $\vu_m \cdot \bm{\beta} \rightarrow \vu\cdot \bm{\beta} \text{ in } \vLb{2}$ which yields $\vu\cdot\bm{\beta} = 0$ on $\Gamma$. Hence, $\vu = \bm{0}$ in $\Omega$. But then
	$$1 = |||\vu_m||| = \|\vu_m\|_{\vL{2}} + \|\DT\vu _m\|_{\mathbb{L}^2(\Omega)} \rightarrow 0$$
	which is a contradiction.
	\hfill
\end{proof}

\section{Stokes equations: $L^2$-theory}
\label{section 3}
\setcounter{equation}{0}
In this section, we study the well-posedness of solutions of the Stokes problem \eqref{S} in the Hilbert space. First we prove the existence and uniqueness of the weak solution.

\begin{theorem}[{\bf Existence in $\boldsymbol{H}^{1}(\Omega)$}]
\label{thm_weak_sol_Stokes_Nbc}
Let $\Omega$ be Lipschitz and
$$
\bm{f}\in\vL{\frac{6}{5}}, \mathbb{F}\in \mathbb{L}^2(\Omega), \bm{h}\in\vHfracbd{1}{2} \text{ and } \alpha\in\Lb{2}
$$
with $\alpha >0$ on $\Gamma_0\subseteq \Gamma$. Then the Stokes problem \eqref{S} has a unique solution $(\vu,\pi)\in \vH{1}\times L^2_0(\Omega)$ which satisfies the estimate:
\begin{equation}
\label{S4estimate}
\|\vu\|_{\vH{1}} + \|\pi\|_{\L{2}} \le C(\alpha) \left( \|\bm{f}\|_{\vL{\frac{6}{5}}} + \|\mathbb{F}\|_{\mathbb{L}^2(\Omega)} +\|\bm{h}\|_{\vHfracbd{1}{2}}\right) .
\end{equation}
\end{theorem}

\begin{proof}
Existence of a unique solution $(\vu,\pi)\in \vH{1}\times L^2_0(\Omega)$ of (\ref{S}) follows from Lax-Milgram theorem. The bilinear form
$$\forall \ \vu, \bm{\varphi}\in\vVsolT{2}, \qquad \bm{a}(\vu, \bm{\varphi}) = 2\int\displaylimits_{\Omega}{\DT\vu:\DT\bm{\varphi}}+\int\displaylimits_{\Gamma}{\alpha\vu_{\vt}\cdot\bm{\varphi}_{\vt}} $$
is clearly continuous. And from Proposition \ref{prop_equivalent_norms_c2},
\begin{align}
\label{S40}
\bm{a}(\vu, \vu) = 2 \|\DT\vu\|^2_{\mathbb{L}^2(\Omega)} + \int\displaylimits_{\Gamma}{\alpha |\vu_{\vt}|^2} \ge C(\alpha) \|\vu\|_{\vH{1}}^2
\end{align}
which shows that it is also coercive on $\vVsolT{2}$. Moreover, the linear form $\bm{\ell} :\vVsolT{2} \rightarrow \R$, defined as
$$\bm{\ell}(\bm{\varphi}) = \int\displaylimits_{\Omega}{\bm{f}\cdot
\bm{\varphi}}-\int\displaylimits_{\Omega}{\mathbb{F}:\nabla \bm{\varphi}}+\langle\bm{h},\bm{\varphi}\rangle_{{\vHfracbd{1}{2}}\times {\bm{H}^\frac{1}{2}(\Gamma)}} $$
is continuous on $\vVsolT{2}$. Hence, Lax-Milgram theorem gives the existence of a unique $\vu\in\vVsolT{2}$ satisfying
\begin{equation}
\label{S41}
\bm{a}(\vu, \bm{\varphi}) =\bm{\ell}(\bm{\varphi}) \qquad \forall  \bm{\varphi}\in\vVsolT{2}.
\end{equation}
This completes the proof since (\ref{S41}) is equivalent to the problem (\ref{S}) (see Proposition \ref{34}). Existence of the pressure term follows from De Rham theorem.

The estimate (\ref{S4estimate}) is obvious from (\ref{S40}) and (\ref{S41}).
\hfill
\end{proof}

\begin{remark}
\rm{ Note that with $\alpha > 0$ on some $\Gamma_0\subseteq \Gamma$ with $|\Gamma_0| > 0$, we get the uniqueness of the solution of the Stokes problem \eqref{S}. But for the case $\alpha \equiv 0$ on $\Gamma$, there is a non-trivial kernel when $\Omega$ is axisymmetric (see \cite[Theorem 3.4]{AR}).
	
Indeed, consider the kernel $\bm{\mathcal{T}}_\alpha (\Omega)$ of the Stokes operator: $(\vu,\pi)\in \vH{1}\times L^2_0(\Omega)$ satisfying (\ref{S}) with $\bm{ f}= \bm{0}$ and $\bm{h} = \bm{0}$. Then we have the energy estimate
\begin{equation*}
2\|\DT\vu\|^2_{\mathbb{L}^2(\Omega)} + \int\displaylimits_{\Gamma}{\alpha |\vu_{ \vt}|^2} = 0
\end{equation*}
with $\alpha \ge 0$ on $\Gamma$. Hence $\DT\bm{v}=0$ in $\Omega$ implies $\bm{u}(\bm{x}) = \bm{b}\times \bm{x} + \bm{c} \text{ for almost all } \bm{x} \in \Omega$ (in fact, for all $x\in \overline{\Omega}$ since $\vu\in\vH{2}\hookrightarrow \bm{C}^0(\overline{\Omega})$) where $\bm{b}, \bm{c}\in\R^3$ are arbitrary constant vectors. But also $\bm{v}\cdot \vn = 0$ on $\Gamma$ gives $\bm{c} = \bm{0}$.\\
{\bf a)} If $\alpha >0$ on $\Gamma_0$, then $\bm{b}\times \bm{x} = \bm{0}, \bm{x}\in \Gamma_0$ and thus $\bm{b} = \bm{0}$ \textit{i.e.} $\bm{\mathcal{T}}_\alpha (\Omega) = \{\bm{0}\}$.\\
{\bf b)} If $\alpha \equiv 0$ on $\Gamma$, we can verify easily that 

{\bf i)} $\bm{u}(\bm{x}) = \bm{b}\times \bm{x}$ if $\Omega$ is axisymmetric \textit{i.e.} $\bm{b}$ is co-linear to the axis of $\Omega$ and $\dim \bm{\mathcal{T}}_\alpha (\Omega)=1$.

{\bf ii)} $\vu = \bm{0}$ if $\Omega$ is not axisymmetric \textit{i.e.} $\bm{\mathcal{T}}_\alpha (\Omega) = \{\bm{0}\}$.
}
\end{remark}

In the next theorem we improve the estimate (\ref{S4estimate}) with respect to $\alpha$ in some particular cases.

\begin{theorem}[{\bf Estimates in $\boldsymbol{H}^{1}(\Omega)$}]
\label{L^2estimate}
With the same assumption on $\bm{f}, \mathbb{F}, \bm{h}$ and $\alpha$ as in Theorem \ref{thm_weak_sol_Stokes_Nbc}, the solution $(\vu,\pi)\in \vH{1}\times L^{2}_0(\Omega)$ of the Stokes problem \eqref{S} satisfies the following estimates:\\
{\bf a)} if $\Omega$ is not axisymmetric, then
\begin{equation}
\label{51}
\|\vu\|_{\vH{1}}+\|\pi\|_{L^{2}(\Omega)}\leq {C(\Omega)}\left(\|\bm{f}\|_{\vL{\frac{6}{5}}}  +\|\mathbb{F}\|_{\mathbb{L}^2(\Omega)}+\|\bm{h}\|_{\vHfracbd{1}{2}}\right) .
\end{equation}
{\bf b)} if $\Omega$ is axisymmetric and \\
\indent $\left(\textit{i} \right)$ $\alpha \geq \alpha_* >0$ on $\Gamma$, then
\begin{equation}
\label{52}
\|\vu\|_{\vH{1}}+\|\pi\|_{L^{2}(\Omega)}\leq \frac{C(\Omega)}{\min\{2,\alpha_*\}}\left(\|\bm{f}\|_{\vL{\frac{6}{5}}}+\|\mathbb{F}\|_{\mathbb{L}^2(\Omega)}+\|\bm{h}\|_{\vHfracbd{1}{2}}\right) .
\end{equation}
\indent $\left(\textit{ii} \right)$ $\bm{f}, \mathbb{F}$ and $\bm{h}$ satisfy the condition:
\begin{equation}
\label{50}
\int\displaylimits_{\Omega}{ \bm{f} \cdot \bm{\beta}} -\int\displaylimits_{\Omega}{\mathbb{F}:\nabla \bm{\beta}}+ \left\langle \bm{h}, \bm{\beta}\right\rangle_\Gamma = 0
\end{equation}
\qquad then, the solution $\vu$ satisfies $\int\displaylimits_{\Gamma}{\alpha \vu \cdot \bm{\beta}} = 0$ and
\begin{equation}
\label{53}
\|\DT\vu\|_{\mathbb{L}^2(\Omega)}^2+\int\displaylimits_{\Gamma}{\alpha |\vu_{\vt}|^2} + \|\pi\|^2_{\L{2}}\leq C(\Omega)\left(\|\bm{f}\|_{\vL{\frac{6}{5}}}+\|\mathbb{F}\|_{\mathbb{L}^2(\Omega)}+\|\bm{h}\|_{\vHfracbd{1}{2}}\right)^2 .
\end{equation}
\qquad In particular, if $\alpha$ is a non-zero constant, then $\int\displaylimits_{\Gamma}{\vu \cdot \bm{\beta}} = 0$ and
\begin{equation}
\label{54}
\|\vu\|_{\vH{1}}+ \|\pi\|_{L^{2}(\Omega)}\leq C(\Omega)\left(\|\bm{f}\|_{\vL{\frac{6}{5}}}+\|\mathbb{F}\|_{\mathbb{L}^2(\Omega)}+\|\bm{h}\|_{\vHfracbd{1}{2}}\right) .
\end{equation}
\end{theorem}

\begin{remark}
	\rm{Note that in the case of $\Omega$ axisymmetric, if $\alpha$ is a non-zero constant, we can use the estimate (\ref{52}) with $\alpha = \alpha_*$. In particular if $\alpha = \frac{1}{n}, n\in \mathbb{N}^*$, the corresponding solution $(\vu_n,\pi_n)$ satisfy
	\begin{equation*}
		\label{S4RE1}
		\|\vu_n\|_{\vH{1}} + \|\pi_n\|_{\L{2}} \le nC(\Omega) \left( \|\bm{f} \|_{\vL{\frac{6}{5}}}+\|\mathbb{F}\|_{\mathbb{L}^2(\Omega)} + \|\bm{h}\|_{\vHfracbd{1}{2}}\right) .
	\end{equation*}
But this estimate is not optimal, when we suppose (\ref{50}). In fact because of $\int\displaylimits_{\Gamma}{\vu_{n}\cdot \bm{\beta}} = 0$, we have by (\ref{54}) the better estimate
	\begin{equation*}
		\label{S4RE2}
		\|\vu_n\|_{\vH{1}} + \|\pi_n\|_{\L{2}} \le C(\Omega) \left( \|\bm{f} \|_{\vL{\frac{6}{5}}}+\|\mathbb{F}\|_{\mathbb{L}^2(\Omega)} + \|\bm{h}\|_{\vHfracbd{1}{2}}\right)
	\end{equation*}
where $C(\Omega)$ does not depend on $n$. That means if $\alpha\rightarrow 0$, (\ref{54}) is better estimate than (\ref{52}).
	}
\end{remark}

\begin{proof}
The solution $\vu$ satisfies:
\begin{align}
2\int\displaylimits_{\Omega}{|\DT\vu|^2}+\int\displaylimits_{\Gamma}{\alpha|\vu_{\vt}|^2}
& \leq C(\Omega) \left(\|\bm{f}\|_{\vL{\frac{6}{5}}}+\|\mathbb{F}\|_{\mathbb{L}^2(\Omega)}+\|\bm{h}
\|_{\vHfracbd{1}{2}}\right) \|\vu\|_{\vH{1}}
\label{3} .
\end{align}
{\bf a)} If $\Omega$ is not axisymmetric, estimate \eqref{eqn_equivalence_norms_c1} shows that the norm $\|\DT\vu\|_{\mathbb{L}^2(\Omega)}$ is equivalent to the norm $\|\vu\|_{\vH{1}}$ and hence from \eqref{3}, it follows 
\begin{equation}
\label{7}
\|\vu\|_{\vH{1}}\leq C(\Omega)\left(\|\bm{f}\|_{\vL{\frac{6}{5}}}+\|\mathbb{F}\|_{\mathbb{L}^2(\Omega)}+\|\bm{h}\|_{\vHfracbd{1}{2}}\right) .
\end{equation}
On the other hand,
\begin{align}
\|\pi\|_{L^{2}(\Omega)}\leq\|\nabla\pi\|_{\vH{-1}}
 & \leq C(\Omega)\left(\|\bm{f}\|_{\vL{\frac{6}{5}}}+\|\mathbb{F}\|_{\mathbb{L}^2(\Omega)}+\|\Delta\vu\|_{\vH{-1}}\right)\notag\\
 & \leq C(\Omega)\left(\|\bm{f}\|_{\vL{\frac{6}{5}}}+\|\mathbb{F}\|_{\mathbb{L}^2(\Omega)}+\|\vu\|_{\vH{1}}\right)\notag\\
 & \leq C(\Omega)\left(\|\bm{f}\|_{\vL{\frac{6}{5}}}+\|\mathbb{F}\|_{\mathbb{L}^2(\Omega)}+\|\bm{h}\|_{\vHfracbd{1}{2}}\right).
 \label{ineq_est_pressure_L2}
\end{align}
The estimate \eqref{51} then follows from \eqref{7} and \eqref{ineq_est_pressure_L2}.\\
\\
{\bf b)} If $\Omega$ is axisymmetric and\\
\indent $\left(\textit{i} \right)$ $\alpha \geq \alpha_* >0$, estimate \eqref{eqn_equivalence_norms_c2} gives,
\begin{align}
 \|\vu\|_{\vH{1}}^2 
& \leq \frac{C(\Omega)}{\min\{2,\alpha_*\}} \left( 2\|\DT\vu\|_{\mathbb{L}^2(\Omega)}^2+\alpha_*\|\vu_{\vt}\|_{\vLb{2}}^2\right)  \notag\\
& \leq \frac{C(\Omega)}{\min\{2,\alpha_*\}} \left(  2\int\displaylimits_{\Omega}{|\DT\vu|^2}+\int\displaylimits_{\Gamma}{\alpha|\vu_{\vt}|^2} \right) \label{4} ;
\end{align}

Hence estimate \eqref{52} follows from \eqref{3}.\\

\indent $\left(\textit{ii} \right)$ $\bm{f}, \mathbb{F}$ and $\bm{h}$ satisfy the condition \eqref{50}, then from \eqref{S41}, we get
\begin{align*}
2\int\displaylimits_{\Omega}{|\DT\vu|^2}+\int\displaylimits_{\Gamma}{\alpha|\vu_{\vt}|^2}
& = \int\displaylimits_{\Omega}{\bm{f} \cdot \vu} - \int\displaylimits_{\Omega}{\mathbb{F}:\nabla \vu} + \langle \bm{h}, \vu \rangle_{\Gamma}\\
& = \int\displaylimits_{\Omega}{\bm{f} \cdot (\vu+k\bm{\beta})} - \int\displaylimits_{\Omega}{\mathbb{F}:\nabla (\vu + k\bm{\beta})}+ \langle \bm{h}, \vu + k\bm{\beta}\rangle_{\Gamma} \qquad \forall
 k\in\R\\
& \leq C(\Omega)\left(\|\bm{f}_{\vL{\frac{6}{5}}} + \|\mathbb{F}\|_{\mathbb{L}^2(\Omega)}+ \|\bm{h}\|_{\vHfracbd{1}{2}} \right) \inf\limits_{k\in\R} \|\vu+k\bm{\beta}\|_{\vH{1}} .
\end{align*}

Also from Korn inequality and the inequality \eqref{55}, we know,
$$\inf\limits_{k\in\R} \|\vu+k\bm{\beta}\|_{\vH{1}}^2 \leq C(\Omega) \left(\inf\limits_{k\in\R} \|\vu+k\bm{\beta}\|_{\vL{2}}^2 + \|\DT\vu\|_{\mathbb{L}^2(\Omega)}^2 \right)\leq C(\Omega) \ \|\DT\vu\|_{\mathbb{L}^2(\Omega)}^2 .$$
which yields
$$2\int\displaylimits_{\Omega}{|\DT\vu|^2}+\int\displaylimits_{\Gamma}{\alpha|\vu_{\vt}|^2} \leq C(\Omega) \left(\|\bm{f}_{\vL{\frac{6}{5}}}+\|\mathbb{F}\|_{\mathbb{L}^2(\Omega)}+ \|\bm{h}\|_{\vHfracbd{1}{2}} \right) \|\DT\vu\|_{\mathbb{L}^2(\Omega)}.$$
This in turn implies
$$\|\DT\vu\|_{\mathbb{L}^2(\Omega)} \leq C(\Omega) \left(\|\bm{f}_{\vL{\frac{6}{5}}}+\|\mathbb{F}\|_{\mathbb{L}^2(\Omega)}+ \|\bm{h}\|_{\vHfracbd{1}{2}} \right)$$
 and then
$$\int\displaylimits_{\Gamma}{\alpha|\vu_{\vt}|^2} \leq C(\Omega) \left(\|\bm{f}_{\vL{\frac{6}{5}}}+\|\mathbb{F}\|_{\mathbb{L}^2(\Omega)}+ \|\bm{h}\|_{\vHfracbd{1}{2}} \right)^2 $$
and hence the inequality \eqref{53}.\\

\indent Moreover, if $\alpha$ is a non-zero constant, the variational formulation \eqref{S41} gives,
$$\int\displaylimits_{\Gamma}{\vu\cdot \bm{\beta}} = 0 .$$
So now \eqref{44} shows that the norm $\|\DT\vu\|_{\mathbb{L}^2(\Omega)}$ is equivalent to the full norm $\|\vu\|_{\vH{1}}$  and \eqref{54} is a consequence of\eqref{3}.
\hfill
\end{proof}

Next we discuss the strong solution of the system \eqref{S} and the corresponding bounds, not depending on $\alpha$.

\begin{theorem}[{\bf Existence and estimate in $\boldsymbol{H}^{2}(\Omega)$}]
\label{13}
Assume that $\alpha$ is a constant. Then if
$$
\bm{f} \in \vL{2} \text{ and } \bm{h}\in\vHfracb{1}{2},
$$
the solution $(\vu,\pi)$ of the Stokes problem \eqref{S} with $\mathbb{F}=0$ belongs to $\vH{2}\times \H1$. Also it satisfies the following estimates:\\
{\bf (i)} if $\Omega$ is not axisymmetric, then
\begin{equation}
\label{5}
\|\vu\|_{\vH{2}} + \|\pi\|_{\H{1}} \leq C(\Omega) \left(\|\bm{f}\|_{\vL{2}} + \|\bm{h}\|_{\vHfracb{1}{2}}\right).
\end{equation}
 {\bf (ii)} if $\Omega$ is axisymmetric, then
\begin{equation}
\label{63}
\|\vu\|_{\vH{2}} + \|\pi\|_{\H{1}} \leq \frac{C(\Omega)}{\min\{2,\alpha\}} \left(\|\bm{f}\|_{\vL{2}} + \|\bm{h}\|_{\vHfracb{1}{2}}\right).
\end{equation}

If moreover, $\bm{f}, \bm{h}$ satisfy the condition:
$$\int\displaylimits_{\Omega}{\bm{f}\cdot \bm{\beta}} + \left\langle \bm{h}, \bm{\beta}\right\rangle _{\Gamma} = 0$$
\indent then
\begin{equation}
\label{6}
\|\vu\|_{\vH{2}} + \|\pi\|_{\H{1}} \leq C(\Omega) \left(\|\bm{f}\|_{\vL{2}} + \|\bm{h}\|_{\vHfracb{1}{2}}\right).
\end{equation}
\end{theorem}

\begin{remark}
	\rm{
	{\bf 1}. We show in Theorem \ref{thm_W2p_regularity_Stokes_Nbc} the existence of $\vu\in \vH{2}$ for more general $\alpha$, not necessarily constant.

{\bf 2}. It is not sensible to consider non-zero $ \mathbb{F}\in \mathbb{H}^1(\Omega)$ for the strong solution as we are considering any $\vL{2}$ data $\bm{f}$ in the RHS.}
\end{remark}

\begin{proof}
\textbf{Method I:} If $\alpha$ is a constant and $\bm{f}\in \vL{2}$ and $\bm{h}\in \vHfracb{1}{2}$, then $\vu \in \vH{1}$ and therefore $\alpha \vu_{\vt} \in \vHfracb{1}{2}$. So using the regularity result for strong solution \cite[Theorem 4.1]{AR}, we get that $\vu\in \vH{2}$.

But concerning the estimate, with this method, using the results in \cite{AR}, we can not obtain the bound on $\vu$, independent of $\alpha$. Thus we need to consider the fundamental but long method, explained below.
\\
\textbf{Method II:} Here we follow the method of difference quotient as in the book of L.C.Evans \cite{Evans}. Without loss of generality, we consider $\bm{h} = \bm{0}$, for ease of notation. Also, let denote the difference quotient by,
	\begin{equation*}
	D^h_k \vu (x) = \frac{\vu(x+h\bm{e}_k) - \vu(x)}{h}, \quad k = 1,2,3, \quad h\in \R.
	\end{equation*}
	\\
\textbf{Interior regularity:} The unique solution $(\vu,\pi)$ in $\vH{1}\times L^{2}_0(\Omega)$ of \eqref{S} belongs to $\bm{H}^2_{loc}(\Omega)\times H^1_{loc}(\Omega)$ with the corresponding local estimate \eqref{5}-(\ref{6}), can be shown in the same way using difference quotient as for the Dirichlet boundary condition, with the help of Proposition \ref{L^2estimate}, since it does not depend on which boundary condition is considered. Thus we do not repeat it.\\
	\\
	\textbf{Boundary regularity:} The solution $(\vu,\pi)$ satisfies the variational formulation, for all $\bm{\varphi} \in \bm{H}^1_{\vt}(\Omega)$,
	\begin{equation}
	\label{56}
	2 \int\displaylimits_{\Omega}{\DT\vu:\DT\bm{\varphi}} + \int\displaylimits_{\Gamma}{\alpha \vu_{\vt}\cdot \bm{\varphi_{\vt}} } - \int\displaylimits_{\Omega}{\pi \ \div \bm{\varphi}} = \int\displaylimits_{\Omega}{\bm{f} \cdot \bm{\varphi}} .
	\end{equation}
\textbf{Case 1. $\Omega  = B(0,1) \cap \R^3_+$ :} First we consider the case when $\Omega$ is a half ball.	Set $V := B(0,\frac{1}{2}) \cap \R^3_+$ and choose a cut-off function $\zeta \in \D{\R^3}$ such that
	\begin{align*}
	\begin{cases}
	& \zeta \equiv 1 \text{ on } B(0,\frac{1}{2}), \ \zeta \equiv 0 \text{ on } \R^3 \setminus B(0,1),\\
	& 0\leq \zeta \leq 1 .
	\end{cases}
	\end{align*}
	So $\zeta \equiv 1$ on $V$ and vanishes on the curved part of $\Gamma$.
	
\textbf{i) Tangential regularity of velocity:} Let $h>0$ be small and $\bm{\varphi} = -D^{-h}_k(\zeta^2 D^h_k \vu), k= 1,2$. Clearly, $\bm{\varphi} \in \bm{H}^1_{\vt}(\Omega)$. So substituting $\bm{\varphi}$ into the identity \eqref{56} we obtain,
	\begin{equation}
	\label{57}
	\begin{aligned}
	&2\int\displaylimits_{\Omega}{\zeta^2 |D^h_k \DT \vu|^2} + 2 \int\displaylimits_{\Omega}{D^h_k \DT \vu : 2\zeta \nabla \zeta D^h_k \vu } + \int\displaylimits_{\Gamma}{ \alpha \zeta ^2 |D^h_k \vu_{\vt}|^2} \\
	&- \int\displaylimits_{\Omega}{\pi \ \div (-D^{-h}_k(\zeta^2 D^h_k \vu))} = \int\displaylimits_{\Omega}{\bm{f}\cdot ( -D^{-h}_k(\zeta^2 D^h_k \vu))} .
	\end{aligned}
	\end{equation}
	Now we estimate the different terms. In this proof, from here onwards, the constant $C$ might depend on $\zeta$ which we do not mention. For the second term in the left hand side, using Cauchy's inequality with $\epsilon$, we get 
	\begin{equation}
	\label{58}
	\begin{aligned}
	|\int\displaylimits_{\Omega}{D^h_k \DT \vu : 2\zeta \nabla \zeta D^h_k \vu }| & \leq C \int\displaylimits_{\Omega}{2\zeta |D^h_k \DT \vu| |D^h_k \vu| }\\
	& \leq C \left[ \epsilon \int\displaylimits_{\Omega}{\zeta ^2 |D^h_k\DT \vu |^2} + \frac{1}{\epsilon} \int\displaylimits_{\Omega}{|D^h_k \vu|^2}\right] .
	\end{aligned}
	\end{equation}
	Similarly, for the fourth term in the left hand side, we write,
	\begin{equation*}
	|\int\displaylimits_{\Omega}{\pi \ \div (-D^{-h}_k(\zeta^2 D^h_k \vu))}| \leq \epsilon \int\displaylimits_{\Omega}{|\div (-D^{-h}_k(\zeta^2 D^h_k \vu))|^2} + \frac{C}{\epsilon} \int\displaylimits_{\Omega}{|\pi|^2}
	\end{equation*}
	But note that,
	\begin{align*}
	\div (D^{-h}_k(\zeta^2 D^h_k \vu)) &= D^{-h}_k\div (\zeta^2 D^h_k \vu)\\ &= D^{-h}_k (2\zeta \nabla \zeta \cdot D^h_k \vu) + D^{-h}_k (\zeta ^2 \underbrace{\div (D^h_k \vu)}_{= 0})\\
	&= D^{-h}_k (2\zeta \nabla \zeta) \cdot D^h_k \vu(x-he_k) + 2\zeta \nabla \zeta \cdot D^{-h}_k D^h_k \vu
	\end{align*}
	which means
	\begin{align*}
	\int\displaylimits_{\Omega}{|\div (-D^{-h}_k(\zeta^2 D^h_k \vu))|^2} & \leq C \left( \int\displaylimits_{\Omega}{|D^h_k \vu|^2} + \int\displaylimits_{\Omega}{\zeta^2|D^{-h}_kD^h_k \vu|^2}\right) \\
	& \leq C \left( \int\displaylimits_{\Omega}{|D^h_k \vu|^2} + \int\displaylimits_{\Omega}{\zeta^2|\nabla D^h_k \vu|^2}\right) .
	\end{align*}
	Therefore, 
	\begin{equation}
	\label{59}
	|\int\displaylimits_{\Omega}{\pi \ \div (D^{-h}_k(\zeta^2 D^h_k \vu))}| \leq \epsilon \left( \int\displaylimits_{\Omega}{|D^h_k \vu|^2} + \int\displaylimits_{\Omega}{\zeta^2|\nabla D^h_k \vu|^2}\right)  + \frac{C}{\epsilon} \int\displaylimits_{\Omega}{|\pi|^2} .
	\end{equation}
	And for the right hand side, proceeding in the same way, we derive,
	\begin{equation*}
	|\int\displaylimits_{\Omega}{\bm{f}\cdot ( -D^{-h}_k(\zeta^2 D^h_k \vu))}| \leq \epsilon \int\displaylimits_{\Omega}{|D^{-h}_k(\zeta^2 D^h_k \vu)|^2} + \frac{C}{\epsilon} \int\displaylimits_{\Omega}{|\bm{f}|^2}.
	\end{equation*}
	But, since
	\begin{align*}
	\int\displaylimits_{\Omega}{|D^{-h}_k(\zeta^2 D^h_k \vu)|^2} \leq C \int\displaylimits_{\Omega}{|\nabla (\zeta^2 D^h_k \vu)|^2} \leq C \left( \int\displaylimits_{\Omega}{|D^h_k \vu|^2} + \int\displaylimits_{\Omega}{ \zeta ^2|\nabla D^h_k \vu|^2}\right)
	\end{align*}
	we get,
	\begin{equation}
	\label{60}
	|\int\displaylimits_{\Omega}{\bm{f}\cdot ( -D^{-h}_k(\zeta^2 D^h_k \vu))}| \leq \epsilon \left( \int\displaylimits_{\Omega}{|D^h_k \vu|^2} + \int\displaylimits_{\Omega}{ \zeta ^2|\nabla D^h_k \vu|^2}\right) + \frac{C}{\epsilon} \int\displaylimits_{\Omega}{|\bm{f}|^2} .
	\end{equation}
	Hence, incorporating \eqref{58}, \eqref{59} and \eqref{60} in \eqref{57} yields,
	\begin{equation}
	\label{61}
	\begin{aligned}
	& 2\int\displaylimits_{\Omega}{\zeta^2 |D^h_k \DT \vu|^2} + \int\displaylimits_{\Gamma}{ \alpha \ \zeta ^2 |D^h_k \vu_{\vt}|^2}\\
	\leq \ & \epsilon \left(\int\displaylimits_{\Omega}{ \zeta ^2|\DT D^h_k \vu|^2} + \int\displaylimits_{\Omega}{ \zeta ^2|\nabla D^h_k \vu|^2}\right) + \frac{C_1}{\epsilon} \left( \int\displaylimits_{\Omega}{|\bm{f}|^2} + \int\displaylimits_{\Omega}{|\pi|^2} \right)\\
	& \qquad \qquad \qquad \qquad \qquad \qquad \qquad \qquad \qquad \qquad + C_2 \int\displaylimits_{\Omega}{|D^h_k \vu|^2}\\
	\leq \ & \epsilon \int\displaylimits_{\Omega}{ \zeta ^2|\nabla D^h_k \vu|^2}+ \frac{C_1}{\epsilon} \left( \int\displaylimits_{\Omega}{|\bm{f}|^2} + \int\displaylimits_{\Omega}{|\pi|^2} \right)+ C_2 \int\displaylimits_{\Omega}{|D^h_k \vu|^2}.
	\end{aligned}
	\end{equation} 
	Furthermore, we see that
	\begin{equation*}
	\begin{aligned}
	\|\zeta D^h_k \vu \|^2_{\vH{1}} & \leq C\left( \|\zeta D^h_k \vu \|^2_{\vL{2}} + \|\DT(\zeta D^h_k \vu )\|^2_{\mathbb{L}^2(\Omega)}\right) \\
	& \leq C\left( \|\zeta D^h_k \vu \|^2_{\vL{2}} + \|\nabla \zeta D^h_k \vu\|^2_{\vL{2}} + \|\zeta \DT D^h_k \vu\|^2_{\mathbb{L}^2(\Omega)}\right) \\
	& \leq C\left( \|D^h_k \vu\|^2_{\vL{2}} + \|\zeta \DT D^h_k \vu\|^2_{\mathbb{L}^2(\Omega)}\right)
	\end{aligned}
	\end{equation*} 
	and
	\begin{equation*}
	\begin{aligned}
	\|\zeta \nabla D^h_k \vu\|^2_{\mathbb{L}^2(\Omega)} = \| \nabla(\zeta D^h_k \vu) - \nabla \zeta D^h_k \vu\|^2_{\mathbb{L}^2(\Omega)}
	& \leq \| \nabla(\zeta D^h_k \vu)\|^2_{\mathbb{L}^2(\Omega)} + C \ \|D^h_k \vu\|^2_{\vL{2}} \\
	& \leq C\left( \|\zeta D^h_k \vu\|^2_{\vH{1}} + \|D^h_k \vu\|^2_{\vL{2}}\right) .
	\end{aligned}
	\end{equation*}
	Combining these inequalities with \eqref{61}, we have,
	\begin{equation}
	\label{16}
	\begin{aligned}
	\|\zeta D^h_k \vu\|^2_{\vH{1}} \leq \epsilon \|\zeta D^h_k \vu\|^2_{\vH{1}} + \frac{C_1}{\epsilon} \left( \|\bm{f}\|^2_{\vL{2}} + \|\pi\|^2_{\vL{2}}\right) + C_2 \|D^h_k \vu\|^2_{\vL{2}}.
	\end{aligned}
	\end{equation}
	Choosing $\epsilon$ small, we obtain,
	\begin{equation*}
	\|D^h_k \vu\|^2_{\bm{H}^1(V)} \leq \|\zeta D^h_k \vu\|^2_{\vH{1}} \leq C \left( \|\bm{f}\|^2_{\vL{2}} + \|\pi\|^2_{\vL{2}} + \|D^h_k \vu\|^2_{\vL{2}}\right)
	\end{equation*} 
	for $k=1,2$ and sufficiently small $|h|\neq 0$; which yields that $\partial^2 \vu / \partial x_i \partial x_j$ belongs to $\bm{L}^2(V)$ for all $i,j = 1,2,3$ except for $i=j=3$ with the corresponding estimates, using the estimates in Proposition \ref{L^2estimate} for $(\vu,\pi)$ in $\vH{1}\times \vL{2}$. 

\textbf{ii)	Tangential regularity of pressure:} Now we deduce the tangential regularity of the pressure in terms of the above derivatives of $\vu$. Indeed, for $i=1,2$, from the Stokes equation, we get,
	\begin{equation*}
	\frac{\partial}{\partial x_i}(\nabla \pi) = \frac{\partial}{\partial x_i} (\bm{f} + \Delta \vu) = \frac{\partial \bm{f}} {\partial x_i} + \div (\nabla \frac{\partial \vu}{\partial x_i}) .
	\end{equation*}
	Since there is no term of the form $\partial^2 \vu/ \partial x_3^2$, by preceding arguments, we obtain
	\begin{equation*}
	\nabla \frac{\partial \pi}{\partial x_i} = \frac{\partial}{\partial x_i}(\nabla \pi) \in \bm{H}^{-1}(V) .
	\end{equation*}
	Furthermore, as we already know $\frac{\partial \pi}{\partial x_i}\in {H}^{-1}(V)$, Ne\v{c}as inequality implies $\frac{\partial \pi}{\partial x_i}\in L^2(V)$ which also satisfies the usual estimate.
	
\textbf{iii) Normal regularity:} For the complete regularity of the solution, it remains to study the derivatives of $\vu$ and $\pi$ in the direction $\bm{e}_3$. Differentiating the divergence equation with respect to $x_3$ gives,
	\begin{equation*}
	\frac{\partial ^2 u_3}{\partial x_3^2} = -\sum_{i=1}^{2} \frac{\partial ^2 u_i}{\partial x_i \partial x_3} \in L^2(V) .
	\end{equation*}
Also from the 3rd component of the Stokes equation, we can write,
	\begin{equation*}
	\frac{\partial \pi}{\partial x_3} = f_3 + \Delta u_3 \in L^2(V)
	\end{equation*}
	which proves that $\pi\in H^1(V)$. Finally, for $i=1,2$, writing the \textit{i}th equation of the system in the form
	\begin{equation*}
	\frac{\partial^2 u_i}{\partial x_3^2}= -\sum_{j=1}^{2} \frac{\partial^2 u_i}{\partial x_j^2} - f_i + \frac{\partial \pi}{\partial x_i} \in L^2(V) 
	\end{equation*}
gives $u_i \in H^2(V)$. Hence apart from the regularity of $\vu$ and $\pi$, we obtain the existence of a constant $C= C(\Omega)>0$ independent of $\alpha$ such that
	\begin{equation*}
	\|\vu\|_{\bm{H}^{2}(V)} + \|\pi\|_{H^{1}(V)} \leq C \|\bm{f}\|_{\vL{2}} .
	\end{equation*}
	
\textbf{Case 2. General domain:} Now we drop the assumption that $\Omega$ is a half ball and consider the general case. Since $\Gamma$ is $\HC{1}{1}$, for any $x_0 \in \Gamma$, we can assume, upon relabelling the coordinate axes,
	\begin{equation*}
	\Omega\cap B(x_0,r) = \left\lbrace x\in B(x_0,r) : x_3>H(x^\prime)\right\rbrace 
	\end{equation*} 
	for some $r>0$ and $H : \R^2 \rightarrow \R$ of class $\HC{1}{1}$. We denoted here $x^\prime = (x_1,x_2)$. Let us now introduce the change of variable,
	\begin{equation*}
	y = (x_1, x_2, x_3 - H(x^\prime)) := \phi(x)
	\end{equation*}
	i.e.
	\begin{equation*}
	x = (y_1,y_2,y_3+H(y^\prime)) := \phi^{-1}(y) .
	\end{equation*}
	Choose $s>0$ small so that the half ball $\Omega^\prime := B(0,s) \cap \R^3_+$ lies in $\phi(\Omega\cap B(x_0,r))$. Also define $V^\prime := B(0,s/2) \cap \R^3_+$ and $\vu^\prime (y) = \vu (\phi^{-1}(y))$ for $y\in \Omega^\prime$. It is easy to see
	$$\vu^\prime \in \bm{H}^1(\Omega^\prime)$$
	and
	$$\vu^\prime \cdot \vn = 0 \ \text{ on } \ \partial \Omega^\prime \cap \partial \R^3_+ .$$
	Note that, with this transformation, it follows, for $i=1,2,3$ and $j=1,2$,
	\begin{equation*}
	\frac{\partial u_i}{\partial x_j} = \frac{\partial u_i^\prime}{\partial y_j} - \frac{\partial H}{\partial y_j} \frac{\partial u_i^\prime}{\partial y_3}
	\end{equation*}
	and
	\begin{equation*}
	\frac{\partial u_i}{\partial x_3} = \frac{\partial u_i^\prime}{\partial y_3} .
	\end{equation*}
	So,
	\begin{equation*}
	\div \ \vu = \div \ \vu^\prime - \sum_{j=1}^{2} \frac{\partial H}{\partial y_j} \frac{\partial u_j^\prime}{\partial y_3}
	\end{equation*}
	and
	\begin{equation*}
	\DT\vu: \DT\bm{\varphi} = \DT\vu^\prime: \DT\bm{\varphi}^\prime + \partial H \ \nabla \vu^\prime : \nabla \bm{\varphi}^\prime
	\end{equation*}
	where $\partial H$ denotes the product of first order derivatives of $H$ i.e. the form $\frac{\partial H}{\partial y_i}\frac{\partial H}{\partial y_j}$.
	Therefore, \eqref{56} becomes under this change of variable
	\begin{equation}
	\label{2}
	\begin{aligned}
	&2 \int\displaylimits_{\Omega^\prime}{\DT\vu^\prime:\DT\bm{\varphi}^\prime} + \int\displaylimits_{\Gamma^\prime}{\alpha \vu_{\vt}^\prime\cdot \bm{\varphi}^\prime_{\vt} } - \int\displaylimits_{\Omega^\prime}{\pi^\prime \ \div \bm{\varphi}^\prime} \\
	&= \int\displaylimits_{\Omega^\prime}{\bm{f}^\prime \cdot \bm{\varphi}^\prime} - \int\displaylimits_{\Omega^\prime}{\pi^\prime \frac{\partial H}{\partial y_j} \frac{\partial u_j^\prime}{\partial y_3}} + \int\displaylimits_{\Omega^\prime}\partial H \ \nabla \vu^\prime :\nabla \bm{\varphi}^\prime.
	\end{aligned}
	\end{equation}
Now choosing the test function $\bm{\varphi}^\prime = - D^{-h}_k(\zeta^2 D^h_k \vu^\prime)$ as before, we estimate the extra terms. For the second term in the right hand side, it follows
	\begin{equation*}
	|\int\displaylimits_{\Omega^\prime}{\pi^\prime \frac{\partial H}{\partial y_j} \frac{\partial u_j^\prime}{\partial y_3}}| \leq C \left( \|\pi^\prime\|^2_{L^2(\Omega^\prime)} + \|\nabla\vu^\prime\|^2_{\mathbb{L}^2(\Omega')}\right) .
	\end{equation*}
	And for the last term in the right hand side,
	\begin{equation*}
	\begin{aligned}
	& |\int\displaylimits_{\Omega^\prime}\partial H \ \nabla \vu^\prime :\nabla (D^{-h}_k(\zeta^2 D^h_k \vu^\prime)) | \\
	= \ & |\int\displaylimits_{\Omega^\prime}\partial H \ \nabla \vu^\prime : D^{-h}_k \nabla (\zeta^2 D^h_k \vu^\prime) |\\
	= \ & |\int\displaylimits_{\Omega^\prime}D^h_k(\partial H \ \nabla \vu^\prime) : \nabla (\zeta^2 D^h_k \vu^\prime)|\\
	= \ & |\int\displaylimits_{\Omega^\prime}{\left( D^h_k\partial H \ \nabla \vu^\prime (y+he_k) + \partial H \ D^h_k \nabla \vu^\prime\right) : \left( 2\zeta \nabla \zeta D^h_k \vu^\prime + \zeta ^2 \nabla D^h_k \vu^\prime\right) }|\\
	\leq \ & C \left( \int\displaylimits_{\Omega^\prime} {|\nabla \vu^\prime|^2 \mathrm{\;d}y} + \epsilon \int\displaylimits_{\Omega^\prime}{\zeta ^2 |\nabla D^h_k \vu^\prime|^2\mathrm{\;d}y} + \frac{1}{\epsilon}\int\displaylimits_{\Omega^\prime}{|D^h_k\vu^\prime|^2\mathrm{\;d}y} + \int\displaylimits_{\Omega^\prime}{\zeta ^2 |D^h_k \nabla \vu^\prime |^2|\partial H| }\right) .
	\end{aligned}
	\end{equation*}
Note that, in the last line, we have used the fact that $H$ is $\HC{1}{1}$ and thus $ D^h_k\partial H$ is bounded. Hence, accumulating all these inequalities, we obtain from \eqref{2},
	\begin{equation*}
	\begin{aligned}
	&\|\zeta D^h_k \vu^\prime\|^2_{\bm{H}^1(\Omega^\prime)}\\
	\leq & C\left( \|\nabla\vu^\prime\|^2_{\mathbb{L}^2(\Omega')} + \|\pi^\prime \|^2_{L^2(\Omega^\prime)} + \|\bm{f}^\prime\|^2_{\bm{L}^2(\Omega^\prime)} + \epsilon\|\zeta D^h_k \vu^\prime\|^2_{\bm{H}^1(\Omega^\prime)} + \int\displaylimits_{\Omega^\prime} {\zeta^2 |\partial H| |\nabla D^h_k \vu^\prime|^2 }\right) .
	\end{aligned}
	\end{equation*}
But $|\partial H|$ is small for sufficiently small $s>0$, since
	\begin{equation*}
	\frac{\partial H}{\partial y_1}(0,0) = 0 = \frac{\partial H}{\partial y_2}(0,0)
	\end{equation*}
	which gives,
	\begin{equation*}
	\|D^h_k \vu^\prime \|^2_{\bm{H}^1(V^\prime)} \leq \|\zeta D^h_k \vu^\prime\|^2_{\bm{H}^1(\Omega^\prime)} \leq C \|\bm{f}\|^2_{\vL{2}}.
	\end{equation*}
	Then proceeding as in the case of half ball, we can deduce,
	\begin{equation*}
	\vu^\prime \in \bm{H}^2(V^\prime) \ \text{ and } \ \|\vu^\prime\|_{\bm{H}^2(V^\prime)} \leq C(\Omega) \|\bm{f}\|_{\vL{2}}.
	\end{equation*}
	Consequently,
	\begin{equation*}
	\|\vu\|_{\bm{H}^2(V)} \leq C(\Omega) \|\bm{f}\|_{\vL{2}}
	\end{equation*}
	where $V = \phi^{-1} (V')$.
	
	Now as $\Gamma$ is compact, we can cover $\Gamma$ with finitely many sets $\{V_i\}$ as above. Thus summing the resulting estimates, along with the interior estimate, we get $\vu \in \vH{2}$ with the bounds, as mentioned in the Theorem.
	\hfill
\end{proof}

\section{Stokes equations: $L^p$-theory}
\setcounter{equation}{0}
\subsection{General solution in $W^{1,p}(\Omega)$}
In this subsection, we study the regularity of weak solution of the Stokes problem (\ref{S}). Before discussing the $L^p$-regularity, we want to show the equivalent formulation of the problem \eqref{S}.

\begin{lemma}
\label{23}
Let $p\in(1,\infty)$. For $\alpha\in\Lb{t(p)}$ with $t(p)$ defined in (\ref{def_exponent_tp_alpha}), $\vu\in \vW{1}{p}$ and $\bm{\varphi}\in \vW{1}{p^\prime}$, the integral over the boundary $\int\displaylimits_{\Gamma}{\alpha \vu_{\vt} \cdot \bm{\varphi}_{\vt}}$ is well-defined.
\end{lemma}

\begin{proof}
	For convenience, from now on, we consider $\alpha\in\Lb{t(p)}$ where we recall that
	\begin{align*}
	t(p) = \begin{cases}
    2 & \text{if\quad} p=2\\
	2+\varepsilon & \text{if\quad} \frac{3}{2}\leq p\leq3, p \neq 2\\
	\frac{2}{3}\max\{p,p'\}+\varepsilon  & \text{otherwise}
	\end{cases}
	\end{align*}
	for some arbitrary $\varepsilon >0$ sufficiently small. Since $\bm{\varphi}\in\bm{W}^{1,p^\prime}(\Omega)$, we have  $\bm{\varphi}_{\vt}\in\vWfracb{1-\frac{1}{p'}}{p'}\hookrightarrow\vLb{m}$ where
\begin{equation}
\label{S5Em}
\frac{1}{m}=\begin{cases}
1 - \frac{3}{2p} & \text{\quad if\;}p>\frac{3}{2},\\
\text{any positive real number}<1 & \text{\quad if\;}p=\frac{3}{2},\\
0 & \text{\quad if\;}p<\frac{3}{2}.
\end{cases}
\end{equation}
Similarly, for $\vu\in\vW{1}{p}$, \ $\vu_{\vt}\in\vWfracb{1-\frac{1}{p}}{p}\hookrightarrow\vLb{s}$ with
\begin{equation}
\label{S5Es}
\frac{1}{s}=\begin{cases}
\frac{3}{2p}-\frac{1}{2} & \text{\quad if\;}p<3,\\
\text{any positive real number} <1 & \text{\quad if\;}p=3,\\
0 & \text{\quad if\;}p>3.
\end{cases}
\end{equation}
We want to show that $\alpha \vu_{\vt} \in \vLb{m^\prime}$.\\
\\
 \textbf{(i) } $p = 2$: Since $\alpha \in L^2(\Gamma)$, $\alpha\vu_{\vt}\in\vLb{q} \text{\; with\; }\frac{1}{q}=\frac{1}{2}+\frac{1}{4} = \frac{3}{4}$ by H\"older inequality. But $\frac{1}{m^\prime} = 1-\frac{1}{m} = \frac{3}{4}$  i.e. $q = m^\prime$. So the integral is well-defined.\\
\\
\textbf{(ii) } $\frac{3}{2} \leq p\leq 3, \ p\neq 2$: As $\alpha\in L^{2+\varepsilon}(\Gamma)$, $\alpha\vu_{\vt}\in\vLb{q} \text{\; with\; }\frac{1}{q}=\frac{1}{2+\varepsilon}+\frac{1}{s} .$\\
So, for $\frac{3}{2}<p<3$, \ $\frac{1}{q} = \frac{3}{2p}-\frac{1}{2}+\frac{1}{2+\varepsilon}$. Then clearly, $q > m^\prime$.\\
For $p=3$, \ $\frac{1}{q} = \frac{1}{2+\varepsilon}+\frac{1}{s}$ where $s$ is any number $>1$. Then, we can choose $s$ suitably such that $q > m^\prime$.\\
For $p = \frac{3}{2}, \ \frac{1}{q} = \frac{1}{2+\varepsilon}+ \frac{1}{2}$ \ and  we can choose $m>1$ suitably so that $q>m^\prime$. Hence, in all cases, the integral is well-defined.\\
\\
\textbf{(iii) } $p> 3$: Since $\alpha\in L^{\frac{2}{3}p+\varepsilon}(\Gamma)$, $\alpha\vu_{\vt}\in\vLb{q} \text{\; with\; }\frac{1}{q}=\frac{1}{\frac{2}{3}p+\varepsilon}$. So clearly $q>m^\prime$ and thus the integral is well-defined.\\
\\
\textbf{(iv) } $p<\frac{3}{2}$: As $\alpha \in  L^{\frac{2}{3}p^\prime+\varepsilon}(\Gamma)$, $\alpha\vu_{\vt}\in\vLb{q}$ with $\frac{1}{q}=\frac{3}{2p}-\frac{1}{2}+ \frac{1}{\frac{2}{3}p^\prime+\varepsilon}$. Again $q>m^\prime$ and the integral is well-defined.
\hfill
\end{proof}

\begin{proposition}
\label{34}
Let $p\in (1,\infty)$ and
$$ \bm{f}\in\vL{r(p)}, \mathbb{F}\in \mathbb{L}^p(\Omega), \bm{h}\in\vWfracb{-\frac{1}{p}}{p} \text{ and } \alpha\in\Lb{t(p)}$$ with $r(p)$ and $t(p)$ defined by \eqref{def_exponent_rp} and \eqref{def_exponent_tp_alpha} respectively. Then the following two problems are equivalent:
	
\noindent{\bf (i)} find $(\vu,\pi)\in\vW{1}{p}\times\L{p}$ satisfying \eqref{S} in the sense of distributions, and
	
\noindent{\bf (ii)} find $\vu\in\vVsolT{p}$ such that for all $\bm{\varphi}\in\vVsolT{p'},$
	\begin{equation}
	\label{varfor_Stokes_Nbc}
	2\int\displaylimits_{\Omega}{\DT\vu:\DT\bm{\varphi}}+\int\displaylimits_{\Gamma}{\alpha\vu_{\vt}\cdot\bm{\varphi}_{\vt}}=\int\displaylimits_{\Omega}{\bm{f}\cdot
		\bm{\varphi}} - \int\displaylimits_\Omega{\mathbb{F}:\nabla \bm{\varphi}}+\left\langle \bm{h},\bm{\varphi}\right\rangle _{\Gamma}.
	\end{equation}
\end{proposition}

\begin{proof}
Let $\vu\in\vVsolT{p}$ satisfies the variational formulation \eqref{varfor_Stokes_Nbc}. We want to show it implies {\bf (i)}. Choosing $\bm{\varphi}\in\bm{\mathcal{D}}_{\sigma}(\Omega)$ as a test function in \eqref{varfor_Stokes_Nbc}, we have
	$$\langle -\Delta\vu,\bm{\varphi}\rangle_{\bm{\mathcal{D}}^\prime(\Omega)\times \vD{\Omega}} = 2\int\displaylimits_{\Omega}{\DT\vu:\DT\bm{\varphi}} = \int\displaylimits_{\Omega}{\bm{f}\cdot \bm{\varphi}} - \int\displaylimits_\Omega{\mathbb{F}:\nabla \bm{\varphi}} .$$	
So by De Rham's theorem, there exists $\pi \in L^p(\Omega)$, defined uniquely up to an additive constant such that 
	\begin{equation}
	\label{1}
	-\Delta\vu + \nabla \pi = \bm{f} + \div \ \mathbb{F}\quad \text{ in } \Omega.
	\end{equation}
Also $\vu\in\vVsolT{p}$ implies $\div \ \vu=0 \text{\; in $\Omega$\;and\;}\vu\cdot\vn=0\text{\; on $\Gamma$}.$ Thus it remains to prove the Navier boundary condition. Multiplying equation \eqref{1} by $\bm{\varphi}\in\vVsolT{p'}$ and using the Green's formula \eqref{green1}, we obtain from \eqref{varfor_Stokes_Nbc},
\begin{equation}
\label{S4E0}
	 \forall \ \bm{\varphi}\in\vVsolT{p'}, \qquad \left\langle \left[(2\DT\vu+\mathbb{F})\vn\right]_{\vt},\bm{\varphi}\right\rangle _{\Gamma} + \int\displaylimits_{\Gamma}{\alpha\vu_{\vt}\cdot\bm{\varphi}_{\vt}}=\left\langle \bm{h},\bm{\varphi}\right\rangle_\Gamma .
\end{equation}
	Now, let $\mu \in \bm{W}^{\frac{1}{p},p^\prime}(\Omega)$. There exists $ \bm{\varphi}\in\bm{W}^{1,p^\prime}(\Omega)$ such that $ \div \ \bm{\varphi}=0$ in $\Omega$ and $\bm{\varphi}=\bm{\mu_\tau}$ on $\Gamma$. Then $\bm{\varphi}\in\vVsolT{p'}$ and using (\ref{S4E0}),
	\begin{align*}
	\langle [(2\DT\vu+\mathbb{F})\vn]_{\vt}+\alpha\vu_{\vt}-\bm{h},{\bm{\mu}}\rangle_\Gamma &= \langle [(2\DT\vu+\mathbb{F})\vn]_{\vt}+\alpha\vu_{\vt}-\bm{h},{\bm{\mu}_{\vt}}\rangle_\Gamma\\
	&= \langle [(2\DT\vu+\mathbb{F})\vn]_{\vt}+\alpha\vu_{\vt}-\bm{h},{\bm{\varphi}}\rangle_\Gamma = 0.
	\end{align*}
	Hence, $$[(2\DT\vu+\mathbb{F})\vn]_{\vt}+\alpha\vu_{\vt}=\bm{h} \qquad \text{ on } \ \Gamma.$$
	Conversely, using the Green formula \eqref{green1}, we can easily deduce that every solution of \eqref{S} also solves \eqref{varfor_Stokes_Nbc}.
	\hfill
\end{proof}
Now we are in the position to study the $L^p$-regularity of the solution of \eqref{S}. We begin with recalling some useful results. For the following theorem, see \cite[Theorem 4.2]{AS}.

\begin{theorem}
\label{infsup}
Let $X$ and $M$ be two reflexive Banach spaces and $X^\prime$ and $M^\prime$ be their dual spaces. Let $a$ be the continuous bilinear form defined on $X\times M$, $A\in \mathcal{L}(X;M^\prime)$ and $A^\prime\in \mathcal{L}(M;X^\prime)$ be the operators defined by
$$\forall v\in X, \quad \forall w\in M, \quad a(v,w)=\left\langle Av, w\right\rangle  = \left\langle v,A^\prime w\right\rangle $$
and $V= \text{Ker } A$. Then the following statements are equivalent :

\noindent(i) There exists $C = C(\Omega) >0$ such that
\begin{equation}\label{15}
\underset{\underset{w \neq 0}{w\in M}}{\inf} \quad \underset{\underset{v \neq 0}{v \in X}}{\sup} \quad \frac{a(v,w)}{\|v\|_X \ \|w\|_M} \geq C .
\end{equation}
(ii) The operator $A:X/V \mapsto M^\prime$ is an isomorphism and $\frac{1}{C}$ is the continuity constant of $A^{-1}$.

\noindent (iii) The operator $A^\prime: M \mapsto X^\prime \perp V$ is an isomorphism and $\frac{1}{C}$ is the continuity constant of $(A^\prime)^{-1}$.
\end{theorem}

Next we introduce the kernel:
$$\bm{K}^p_T (\Omega) =\left\lbrace \bm{v}\in\vL{p}; \ \div \ \bm{v}= 0, \ \mathbf{curl} \ \bm{v} =\bm{0} \text{ in } \Omega, \ \bm{v}\cdot \vn =0 \text{ on } \Gamma \right\rbrace. $$
Thanks to \cite[Corollary 4.1]{AS}, we know that this kernel is trivial iff $\Omega$ is simply connected. Otherwise, it is of finite dimension and spanned by the functions $\widetilde{\mathbf{grad}} \ q_j^T, 1\leq j\leq J$, where $q^T_j$ is the unique solution up to an additive constant of the problem:
\begin{align*}
\begin{cases}
-\Delta q_j^T &= 0 \text{ in } \Omega^o,\\
\partial_n q_j^T &= 0 \text{ on } \Gamma,\\
[q_j^T]_k &= \text{ constant \quad and } \quad [\partial_n q_j^T]_k = 0, \quad 1\leq k \leq J,\\
\left\langle \partial_n q_j^T,1\right\rangle _{\Sigma_k} &= \delta_{jk}, \quad 1\leq k\leq J .
\end{cases}
\end{align*}
Recall that $\Sigma_j$ are the cuts in $\Omega$ such that the open set $\Omega^0 = \Omega \backslash \bigcup\limits_{j=1}^{J} \Sigma_j$ is simply connected. For more details, see \cite{AS}.

Also, recall the following inf-sup condition (see \cite[Lemma 4.4]{AS}):
\begin{lemma}
There exists a constant $C > 0$, depending only on $\Omega$ and $p$ such that
\begin{equation}
\label{29}
\underset{\underset{\bm{\varphi\neq 0}}{\bm{\varphi}\in\bm{V}^{p^\prime}(\Omega)}}{\inf} \ \ \underset{\underset{\bm{\xi}\neq 0}{\bm{\xi}\in \vVsolT{p}}}{\sup}\frac{\int\displaylimits_{\Omega}{\mathbf{curl} \ \bm{\xi} \cdot  \mathbf{curl} \ \bm{\varphi}}}{\|\bm{\xi}\|_{\vVsolT{p}}\|\bm{\varphi}\|_{\bm{V}^{p^\prime}(\Omega)}} \geq C
\end{equation}
where
$$ \bm{V}^{p^\prime}(\Omega) := \left\lbrace \bm{v}\in \vVsolT{p^\prime}; \ \left\langle \bm{v}\cdot \vn , 1\right\rangle _{\Sigma_j} = 0 \quad \forall \ 1\leq j\leq J\right\rbrace .$$
\end{lemma}

\noindent Let us now define the bilinear form: for $\vu\in \vVsolT{p}$ and $\bm{\varphi}\in \vVsolT{p^\prime}$,
\begin{equation}
\label{a}
a(\vu,\bm{\varphi}) = 2\int\displaylimits_{\Omega}{\DT\vu : \DT\bm{\varphi}} + \int\displaylimits_{\Gamma}{\alpha \vu_{\vt}\cdot \bm{\varphi}_{\vt}} .
\end{equation}

\begin{theorem}
\label{31}
Let $p\in (1,\infty)$, $\bm{\ell}\in [\vVsolT{p^\prime}]^\prime$ and $\alpha\in \Lb{t(p)}$. Then the problem:
\begin{equation}
\label{30}
\text{ find } \vu\in\vVsolT{p} \ \text{such that for any } \ \bm{\varphi}\in\vVsolT{p^\prime}, \quad a(\vu,\bm{\varphi}) = \left\langle \bm{\ell}, \bm{\varphi}\right\rangle 
\end{equation}
has a unique solution.
\end{theorem}

\begin{proof} 
Let us consider first $p \geq 2$. Since $[\vVsolT{p^\prime}]^\prime \hookrightarrow [\vVsolT{2}]^\prime$, by Lax-Milgram theorem there exists a unique $\vu\in\vVsolT{2}$ satisfying
\begin{equation}
\label{28}
\forall \ \bm{\varphi} \in \vVsolT{2}, \quad a(\vu,\bm{\varphi}) = \left\langle \bm{\ell}, \bm{\varphi}\right\rangle_{[\vVsolT{2}]^\prime\times \vVsolT{2} } .
\end{equation}	
Now we want to show that $\vu\in\vW{1}{p}$. Since the inf-sup condition \eqref{15} is known for the bilinear form
$$b(\vu,\bm{\varphi}) = \int\displaylimits_{\Omega}{\mathbf{curl} \ \vu \cdot \mathbf{curl} \ \bm{\varphi}}$$
with suitable spaces $X$ and $M$ (see \eqref{29}), we will use another formulation of problem \eqref{28}. For that, in particular for $\bm{\varphi}\in \bm{H}^2(\Omega)$ with $\div \ \bm{\varphi} = 0$ in $\Omega$ and $\bm{\varphi}\cdot \vn = 0$ on $\Gamma$, we have the relation
$$a(\vu,\bm{\varphi}) = \left\langle \bm{\ell}, \bm{\varphi}\right\rangle  .$$
Therefore, using \eqref{26} and integration by parts, we get, for any $\bm{\varphi} \in \bm{H}^2(\Omega)$ with $\div \ \bm{\varphi} = 0$ in $\Omega$ and $\bm{\varphi}\cdot \vn = 0$ on $\Gamma$,
\begin{equation}
\label{17}
\int\displaylimits_{\Omega}{\mathbf{curl} \ \vu \cdot \mathbf{curl} \ \bm{\varphi}} = \left\langle \bm{\ell}, \bm{\varphi}\right\rangle_{[\vVsolT{2}]^\prime\times \vVsolT{2}} - \int\displaylimits_{\Gamma}{\alpha \vu_{\vt} \cdot \bm{\varphi}_{\vt}} + 2\int\displaylimits_{\Gamma}{\bm{\Lambda} \vu \cdot \bm{\varphi}} .
\end{equation}
But since $\Omega$ is $\HC{1}{1}$, using the density result
$$
 \left\lbrace \bm{v}\in \bm{\mathcal{D}}(\overline{\Omega}), \div \ \bm{v}=0 \text{ in } \Omega, \bm{v}\cdot \vn = 0 \text{ on } \Gamma\right\rbrace \text{ is dense in } \vVsolT{2},
 $$
we have the relation \eqref{17} for all $\bm{\varphi}\in \vVsolT{2}$. Now we are in position to prove that $\vu \in \vW{1}{p}$ and for that we consider different cases.\\
\textbf{(i) } $2<p\leq3$ :\\
\\
\textbf{ 1\textsuperscript{st} Step:} Since $\vu_{\vt}\in\vLb{4}$ and $\alpha\in\Lb{2+\varepsilon}$, we have $\alpha\vu_{\vt}\in\vLb{q_1}$ with $\frac{1}{q_1}=\frac{1}{4}+\frac{1}{2+\varepsilon}$. But, $\vLb{q_1}\hookrightarrow\vWfracb{-\frac{1}{p_1}}{p_1}$ with $p_1=\frac{3}{2}q_1 \ > 2$ i.e.
$\frac{1}{p_1}=\frac{2}{3}\left(\frac{1}{4}+\frac{1}{2+\varepsilon}\right).$\\
Therefore, as $\vWfracb{\frac{1}{p_1}}{p_1^\prime}\hookrightarrow \vLb{q_1^\prime}$ with $\frac{4}{3}< q_1^\prime < 4$ and $\bm{\Lambda}\vu \in \vLb{4}$, the mapping
\begin{equation}
\label{S5e0}
\left\langle \bm{L},\bm{\varphi}\right\rangle = \left\langle \bm{\ell}, \bm{\varphi}\right\rangle_{[\vVsolT{s'_1}]'\times \vVsolT{s'_1}} - \int\displaylimits_{\Gamma}{\alpha \vu_{\vt}\cdot \bm{\varphi}_{\vt} } + 2\int\displaylimits_{\Gamma}{\bm{\Lambda} \vu\cdot \bm{\varphi}} \quad \text{ for } \ \bm{\varphi}\in \bm{V}^{s_1^\prime}(\Omega)
\end{equation}
defines an element in the dual space of $\bm{V}^{s_1^\prime}(\Omega)$ with $s_1 = \text{min }\{p_1,p\}$. Now from the inf-sup condition \eqref{29} and using Theorem \ref{infsup}, there exists a unique $\bm{v}\in \vVsolT{s_1}$ such that
\begin{equation}
\label{18}
\forall \bm{\varphi}\in \bm{V}^{s_1^\prime}(\Omega) \ , \quad \int\displaylimits_{\Omega}{\mathbf{curl} \ \bm{v} \cdot \mathbf{curl} \ \bm{\varphi}} = \left\langle \bm{L}, \bm{\varphi}\right\rangle_{[\bm{V}^{s_1^\prime}(\Omega)]^\prime\times \bm{V}^{s_1^\prime}(\Omega)} .
\end{equation}
We will show that $\mathbf{curl} \ \bm{v} = \mathbf{curl} \ \vu$. For that first we extend \eqref{18} to any test function $\bm{\varphi}\in \vVsolT{s_1^\prime}$. Since \ $\widetilde{\nabla} q_j^T \in \vVsolT{2} \hookrightarrow \vVsolT{s_1^\prime}$, using \eqref{17} we get
\begin{align*}
\left\langle \bm{L},\widetilde{\nabla} q_j^T\right\rangle &= \left\langle \bm{\ell}, \widetilde{\nabla} q_j^T\right\rangle - \int\displaylimits_{\Gamma}{ \alpha \vu_{\vt}\cdot {(\widetilde{\nabla} q_j^T)}_{\vt}} + 2\int\displaylimits_{\Gamma}{\bm{\Lambda} \vu \cdot \widetilde{\nabla} q_j^T}\\
& = \int\displaylimits_{\Omega}{\mathbf{curl} \ \vu \cdot \mathbf{curl} \ \widetilde{\nabla} q_j^T\ }= 0 .
\end{align*}
Hence, for any $\bm{\varphi} \in \vVsolT{s_1^\prime}$, we set
$\bm{\tilde{\varphi}} = \bm{\varphi} - \underset{j}{\Sigma}\left\langle \bm{\varphi}\cdot \vn, 1\right\rangle _{\Sigma _j} \widetilde{\nabla} q_j^T $
which implies
$$\left\langle \bm{L}, \bm{\varphi} \right\rangle  = \left\langle \bm{L}, \bm{\tilde{\bm{\varphi}}}\right\rangle + \underset{j}{\Sigma}\left\langle \bm{\varphi}\cdot \vn, 1\right\rangle _{\Sigma _j} \left\langle \bm{L}, \widetilde{\nabla} q_j^T\right\rangle = \left\langle \bm{L}, \bm{\tilde{\bm{\varphi}}}\right\rangle $$
and also $\widetilde{\bm{\varphi}} \in \bm{V}^{s_1^\prime}(\Omega)$.
Therefore \eqref{18} yields,
$$\int\displaylimits_{\Omega}{\mathbf{curl} \ \bm{v} \cdot \mathbf{curl} \ \bm{\varphi}} = \int\displaylimits_{\Omega}{\mathbf{curl} \ \bm{v} \cdot \mathbf{curl} \ \bm{\tilde{\bm{\varphi}}}} = \left\langle \bm{L}, \bm{\tilde{\bm{\varphi}}}\right\rangle = \left\langle \bm{L}, \bm{\varphi} \right\rangle  . $$
So finally, we get that $\bm{v}\in \vVsolT{s_1}$ satisfies
\begin{equation}
\label{19}
\forall \bm{\varphi}\in \vVsolT{s_1^\prime}, \qquad \int\displaylimits_{\Omega}{\mathbf{curl} \ \bm{v} \cdot \mathbf{curl} \ \bm{\varphi}} = \left\langle \bm{L}, \bm{\varphi}\right\rangle .
\end{equation}
Now as $\vVsolT{2}\hookrightarrow \vVsolT{s_1^\prime}$, we deduce from \eqref{17} that
\begin{equation}
\label{S5e1}
\forall \bm{\varphi}\in \vVsolT{2}, \quad \int\displaylimits_{\Omega}{\mathbf{curl} \ \bm{v} \cdot \mathbf{curl} \ \bm{\varphi}} = \int\displaylimits_{\Omega}{\mathbf{curl} \ \vu \cdot \mathbf{curl} \ \bm{\varphi}}
\end{equation}
which gives,
\begin{equation}
\label{S5e2}
\mathbf{curl} \ \vu = \mathbf{curl} \ \bm{v} \quad \text{ in } \ \Omega .
\end{equation}
Therefore, as $\vu \in \vL{6}\hookrightarrow \vL{s_1}, \mathbf{curl} \ \vu\in \vL{s_1}, \div \ \vu = 0$ in $\Omega$ and $\vu \cdot \vn = 0$ on $\Gamma$; from Proposition \ref{X}, we deduce $\vu \in \vW{1}{s_1}$. If $s_1 = p$, the proof is complete. Otherwise, $s_1 = p_1$ and we proceed to the next step.\\
\\
\textbf{ 2\textsuperscript {nd} Step:} Now $\vu\in\vW{1}{p_1}$ implies $\vu_{\vt}\in\vLb{m}$ where $\frac{1}{m}=\frac{3}{2p_1}-\frac{1}{2} (\text{ since } p_1 \leq 3)$. Then $\alpha\vu_{\vt}\in\vLb{q_2}$ where $\frac{1}{q_2}=\frac{1}{m}+\frac{1}{2+\varepsilon}$. But, $\vLb{q_2}\hookrightarrow\vWfracb{-\frac{1}{p_2}}{p_2}$ with $p_2=\frac{3}{2}q_2 \ > p_1$ i.e. $\frac{1}{p_2}=\frac{2}{3}\left(\frac{1}{4}+\frac{1}{2+\varepsilon}-\frac{1}{2}+\frac{1}{2+\varepsilon}\right)
=\frac{2}{3}\left(\frac{2}{2+\varepsilon}-\frac{1}{2}+\frac{1}{4}\right)$. So $\vWfracb{\frac{1}{p_2}}{p_2^\prime}\hookrightarrow \vLb{q_2^\prime}$ with $m^\prime < q_2^\prime$ and $\bm{\Lambda}\vu\in \vLb{m}$ and hence the mapping $\bm{L}$ in (\ref{S5e0}) where now $\left\langle \bm{\ell},\bm{\varphi}\right\rangle $ is the duality between $[\vVsolT{s'_2}]'$ and $\vVsolT{s'_2}$,
defines an element in the dual of $\bm{V}^{s_2^\prime}(\Omega)$ with $s_2 = \text{min } \{p_2, p\}$. Therefore, as in the previous step, there exists a unique $\bm{v}\in \vVsolT{s_2}$ such that (\ref{19}) holds for any $\bm{\varphi}\in \vVsolT{s'_2}$ and then (\ref{S5e2}).
Thus we get, $\vu \in \vL{p_1^*}\hookrightarrow \vL{s_2}, \mathbf{curl} \ \vu\in \vL{s_2}, \div \ \vu = 0$ in $\Omega$ and $\vu \cdot \vn = 0$ on $\Gamma$ which gives $\vu \in \vW{1}{s_2}$. If $s_2 = p$, we are done. Otherwise, $s_2 = p_2$ and we proceed next.\\
\\
\textbf{(k+1)\textsuperscript{th} Step:} Proceeding similarly, we get $\vu \in \vVsolT{p_{k+1}}$ with
$\frac{1}{p_{k+1}}=\frac{2}{3}\left(\frac{k+1}{2+\varepsilon}-\frac{k}{2}+\frac{1}{4}\right)$ (where in each step, we assumed that $p_k<3$) which also satisfies
\begin{equation*}
\forall \ \bm{\varphi}\in \vVsolT{p_{k+1}^\prime}, \quad \int\displaylimits_{\Omega}{\mathbf{curl} \ \vu \cdot \mathbf{curl} \ \bm{\varphi}} = \left\langle \bm{\ell}, \bm{\varphi}\right\rangle - \int\displaylimits_{\Gamma}{ \alpha \vu_{\vt}\cdot \bm{\varphi}_{\vt}} + 2\int\displaylimits_{\Gamma}{\bm{\Lambda} \vu \cdot \bm{\varphi}} .
\end{equation*}
Now choose $k=[\frac{1}{\varepsilon}-\frac{1}{2}]+1$ such that $p_{k+1}\geq3\geq p$ (where $[ a]$ stands for the greatest integer less than or equal to $a$). Hence $\vu\in\vW{1}{p}$. i.e. for $2< p \leq 3$, there exists a unique $\vu\in \vVsolT{p}$ such that for any $\bm{\varphi}\in \vVsolT{p^\prime}$, we have (\ref{17}) where the duality bracket $\left\langle \bm{\ell}, \bm{\varphi}\right\rangle_{[\vVsolT{2}]' \times \vVsolT{2}}$ is replaced by $\left\langle \bm{\ell}, \bm{\varphi}\right\rangle_{[\vVsolT{p'}]' \times \vVsolT{p'}} $.\\
\\
\textbf{(ii)} $p > 3$ : From the previous case, we have that $\vu\in\vW{1}{3}$ which implies $\vu_{\vt}\in\vLb{s}$ for all $s\in(1,\infty)$. Now $\alpha\in\Lb{\frac{2}{3}p+\varepsilon}$ gives $\alpha\vu_{\vt}\in \vLb{q}$ where $\frac{1}{q} = \frac{1}{s} + \frac{1}{\frac{2}{3}p+ \epsilon}$. Choosing $s>1$ suitably, we can get $q = \frac{2}{3}p$ and hence $\vLb{q}\hookrightarrow\vWfracb{-\frac{1}{p}}{p}$. Since $\vWfracb{\frac{1}{p}}{p^\prime}\hookrightarrow \vLb{q^\prime}$ with $s^\prime < q^\prime$ and $\bm{\Lambda}\vu\in \vLb{s}$, the mapping $\bm{L}$ in (\ref{S5e0})
defines an element in the dual of $\vVsolT{p^\prime}$. So there exists a unique $\bm{v}\in \vVsolT{p}$ such that (\ref{19}) holds for any $\bm{\varphi}\in \vVsolT{p^\prime}$ and thus (\ref{S5e2}).
Therefore we obtain similarly $\vu \in \vW{1}{p}$. Hence, $\vu\in \vVsolT{p}$ solves the problem \eqref{30} for all \ $2\leq p <\infty$ .

Finally consider the operator $A\in \mathcal{L}(\vVsolT{p}, (\vVsolT{p^\prime})^\prime)$, associated to the bilinear form $a$ in \eqref{a}, defined as $\left\langle A\xi, \varphi\right\rangle = a(\xi, \varphi)$ . As proved above, for $p \geq 2$, the operator $A$ is an isomorphism from $\vVsolT{p}$ to $(\vVsolT{p^\prime})^\prime$. Then the adjoint operator, which is equal to $A$ is an isomorphism from $\vVsolT{p^\prime}$ to $(\vVsolT{p})^\prime$ for $p^\prime \leq 2$. This means that the operator $A$ is an isomorphism for $p\leq 2$ also, which ends the proof.
\hfill
\end{proof}

As a consequence, the above theorem yields the next important inf-sup condition.
\begin{proposition}
For all $p\in (1,\infty)$ and $\alpha\in \Lb{t(p)}$, there exists a constant $\gamma= \gamma(\Omega, p, \alpha)>0$ such that
\begin{equation}
\label{infsup-ineq}
\underset{\underset{\bm{\varphi}\neq 0}{\bm{\varphi}\in \vVsolT{p^\prime}}}{\inf} \ \ \underset{\underset{\bm{u}\neq 0}{\bm{u}\in \vVsolT{p}}}{\sup} \frac{2\int\displaylimits_{\Omega}{\DT\bm{u}:\DT\bm{\varphi}}+ \int\displaylimits_{\Gamma}{\alpha \bm{u}_{\vt}\cdot \bm{\varphi}_{\vt}}}{\|\bm{u}\|_{\vVsolT{p}} \ \|\bm{\varphi}\|_{\vVsolT{p^\prime}}} \geq \gamma .
\end{equation}
Also for any $\bm{\ell}\in [\vVsolT{p^\prime}]^\prime$, the unique solution $\vu\in \vVsolT{p}$ of the variational problem:
\begin{equation*}
\forall \ \bm{\varphi}\in\vVsolT{p^\prime}, \quad 2\int\displaylimits_{\Omega}{\DT\vu : \DT\bm{\varphi}} + \int\displaylimits_{\Gamma}{\alpha \vu_{\vt}\cdot \bm{\varphi}_{\vt}} = \left\langle \bm{\ell}, \bm{\varphi}\right\rangle 
\end{equation*}
given by Theorem \ref{31}, satisfies the following estimate:
\begin{equation}
\label{S42}
\|\vu\|_{\vW{1}{p}} \le \frac{1}{\gamma} \|\bm{\ell}\|_{[\vVsolT{p'}]'} .
\end{equation}
\end{proposition}

\begin{remark}
\rm{The inf-sup condition (\ref{infsup-ineq}) will be improved in Theorem \ref{P3} where we obtain that the above continuity constant $ \gamma$ does not depend on $\alpha$.}
\end{remark}

\begin{proof}
Using the equivalence (i) and (ii) in Theorem \ref{infsup}, we obtain the inf-sup condition (\ref{infsup-ineq}) from Theorem \ref{31}. The estimate (\ref{S42}) follows immediately from (\ref{infsup-ineq}).
\hfill
\end{proof}

Finally Theorem \ref{31} enables us to obtain the existence of weak solution for the Stokes problem for all $1<p<\infty$.

\begin{corollary}[{\bf Existence in $\boldsymbol{W}^{1,p}(\Omega)$}]
\label{thm_W1p_regularity_Stokes_Nbc}
Let $p\in (1,\infty)$ and
$$\bm{f}\in\vL{r(p)}, \mathbb{F}\in\mathbb{L}^p(\Omega), \bm{h}\in\vWfracb{-\frac{1}{p}}{p}
\text{ and } \alpha\in\Lb{t(p)}.$$
Then the Stokes problem \eqref{S} has a unique solution $(\vu,\pi)\in\vW{1}{p}\times L^{p}_0(\Omega)$ which satisfies the estimate:
\begin{equation}
\label{S43}
\|\vu\|_{\vW{1}{p}} + \|\pi\|_{\L{p}} \le C(\Omega, \alpha,p) \left( \|\bm{f}\|_{\vL{r(p)}} + \|\mathbb{F}\|_{\mathbb{L}^p(\Omega)} + \|\bm{h}\|_{\vWfracb{-\frac{1}{p}}{p}}\right) .
\end{equation}
\end{corollary}

\begin{remark}
	\rm{The existence of $\vu\in\vW{1}{p}$ and corresponding estimate as above can be deduced directly using the regularity result in \cite[Theorem 3.7]{AR} taking $\alpha \vu_{ \vt}$ as the source term in the right hand side, but only for $p>2$. We need Theorem \ref{31} to obtain the existence of solution for $p<2$.}
\end{remark}

\begin{proof}
Let consider
$$
\left\langle \bm{\ell}, \bm{\varphi}\right\rangle  = \int\displaylimits_{\Omega}{ \bm{f}\cdot \bm{\varphi}} -\int\displaylimits_{\Omega}{\mathbb{F}:\nabla \bm{\varphi}} + \left\langle \bm{h}, \bm{\varphi}\right\rangle _\Gamma \text{ for all } \bm{\varphi}\in \vVsolT{p'}.
$$
Clearly $\bm{\ell}\in [\vVsolT{p^\prime}]^\prime$. Then Theorem \ref{31} yields the existence of a unique $\vu \in \vVsolT{p}$ which satisfies the variational formulation \eqref{varfor_Stokes_Nbc}.

The estimate (\ref{S43}) follows from (\ref{S42}) and
$$
\|\bm{\ell}\|_{[\vVsolT{p'}]'} \le C(\Omega) \left( \|\bm{f}\|_{\vL{r(p)}} + \|\mathbb{F}\|_{\mathbb{L}^p(\Omega)}+ \|\bm{h}\|_{\vWfracb{-\frac{1}{p}}{p}}\right).
$$
\hfill
\end{proof}

\begin{remark}
\rm{
{\bf i)} All the previous and following results where we have assumed $\bm{f}\in\vL{r(p)}$ hold true also for $\bm{f}\in [\bm{H}_0^{(r(p))^\prime,p^\prime}(\div ,\Omega)]^\prime$ which is clear from the characterization of the space in Proposition \ref{40}.
	
{\bf ii)} We also want to emphasize that in this work, our assumption on $\alpha$ is quite steep. We need this regularity in order to ensure that $\alpha \vu_{\vt}\in\vWfracb{-\frac{1}{q}}{q}$ for some $q$ so that eventually we can use our tools. But we will see later (Subsection \ref{less regular}) that we may suppose $\alpha$ less regular in some cases.

{\bf iii)} Note that in the case $\alpha \equiv 0$, we are considering here more general Stokes problem than in \cite{AR}. But all the existence results (and the corresponding estimates) hold for that as well. }
\end{remark}

\subsection{Strong solution in $W^{2,p}(\Omega)$}
Concerning the existence of a strong solution, we prove the following regularity result.

\begin{theorem}[{\bf Existence in $\boldsymbol{W}^{2,p}(\Omega)$}]
\label{thm_W2p_regularity_Stokes_Nbc}
Let $p\in (1,\infty)$. Then, for
$$
\bm{f}\in\vL{p}, \bm{h}\in\vWfracb{1-\frac{1}{p}}{p} \text{ and } \alpha \in \Wfracb{1-\frac{1}{q}}{q}
$$
with $q>\frac{3}{2}$ if $p\le \frac{3}{2}$ and $q=p$ otherwise, the solution $(\vu,\pi)$ of the Stokes problem \eqref{S} with $\mathbb{F}=0$, given by Corollary \ref{thm_W1p_regularity_Stokes_Nbc} belongs to $\vW{2}{p}\times\W{1}{p}$ which also satisfies the estimate:
\begin{equation*}
\|\vu\|_{\vW{2}{p}} + \|\pi\|_{\W{1}{p}} \le C(\Omega, \alpha,p) \left( \|\bm{f}\|_{\vL{p}} + \|\bm{h} \|_{\vWfracb{1-\frac{1}{p}}{p}} \right) .
\end{equation*}
\end{theorem}

\begin{proof}
The proof is done essentially using the existence of weak solution and bootstrap argument. Clearly, the data $\bm{f}, \bm{h}$ and $\alpha$ satisfy the hypothesis of Corollary \ref{thm_W1p_regularity_Stokes_Nbc}. Hence there exists a unique solution $(\vu,\pi)\in\vW{1}{p}\times L^{p}_0(\Omega)$ of \eqref{S}.\\
\\
\textbf{(i) $1< p \leq \frac{3}{2}$:} We also have the following embeddings:\\
$\vL{p} \hookrightarrow \vL{r(q)}, \vWfracb{1-\frac{1}{p}}{p} \hookrightarrow \vWfracb{-\frac{1}{q}}{q} $ and $\Wfracb{1-\frac{1}{\frac{3}{2}+\epsilon}}{\frac{3}{2}+\epsilon} \hookrightarrow \Lb{2+\varepsilon} $
where $q = p^*$ i.e. $\frac{1}{q} = \frac{1}{p} - \frac{1}{3}$ with $q\in (\frac{3}{2},3]$ which show that $(\vu,\pi)\in\vW{1}{q}\times\L{q}$ again using Corollary \ref{thm_W1p_regularity_Stokes_Nbc}. Now $\vu \in \vW{1}{q} \hookrightarrow \vL{q^*}$ and $\nabla \vu \in \vL{q}$. Also as $\alpha \in \Wfracb{1-\frac{1}{\frac{3}{2}+\epsilon}}{\frac{3}{2}+\epsilon}$, we can consider $\alpha \in \W{1}{\frac{3}{2}+\varepsilon}$, using the lift operator. Hence from Sobolev inequality $\alpha \in \L{(\frac{3}{2}+\varepsilon)^*}$ and $\nabla \alpha \in \vL{\frac{3}{2}+\varepsilon}$. All these implies, for all $ i, j =1, 2, 3, $
$$\alpha \frac{\partial u_i}{\partial x_j} \in \L{q_1} \text{ where }\frac{1}{q_1} = \frac{1}{\frac{3}{2}+\varepsilon} - \frac{1}{3} + \frac{1}{q} $$
and $$\frac{\partial \alpha}{\partial x_j} u_i
 \in \L{q_2} \text{ where } \frac{1}{q_2} = \frac{1}{\frac{3}{2}+\varepsilon} + \frac{1}{q^*} \ .$$
But $q_1 = q_2 > p$ and thus $\frac{\partial}{\partial x_j}(\alpha u_i)=\frac{\partial\alpha}{\partial x_j}u_i+\alpha\frac{\partial u_i}{\partial x_j} \in \L{p}$. This implies $\alpha \vu \in \vW{1}{p}$ or in other words $\alpha\vu_{\vt} \in \vWfracb{1-\frac{1}{p}}{p}$. Therefore the (general) regularity result as \cite[Theorem 4.1]{AR} gives $(\vu, \pi )\in \vW{2}{p}\times \W{1}{p}$. Note that it is possible to prove the strong existence result \cite[Theorem 4.1]{AR} only for $\HC{1}{1}$ domain since the problem (\ref{S}) takes the form of an uniformly elliptic operator with complementing boundary conditions in the sense of Agmon-Douglis-Nirenberg \cite{agmon}. \\
\\
\textbf{(ii) $p > \frac{3}{2}$:} First we assume $p< 3$. We have
$\vu\in\vW{2}{\frac{3}{2}}\hookrightarrow\vL{s}$ for all $s\in (1,\infty)$ and $\nabla\vu\in\vW{1}{\frac{3}{2}}\hookrightarrow\vL{3}$.
Also, since $\alpha \in \Wfracb{1-\frac{1}{p}}{p}$, we can consider $\alpha \in \W{1}{p}$. Then $\alpha \in \L{p^*} \text{ and } \nabla \alpha \in \vL{p} $. Therefore for all $i,j=1,2,3,$
$$\frac{\partial\alpha}{\partial x_j}u_i\in\L{q_2} \ \text{ where } \ \frac{1}{q_2}=\frac{1}{p}+\frac{1}{s}$$
and
$$\alpha\frac{\partial u_i}{\partial x_j}\in\L{q_3} \ \text{ where } \ \frac{1}{q_3}=\frac{1}{p^*}+\frac{1}{3}=\frac{1}{p} \ .$$
Clearly, $q_2<q_3$ and then $\frac{\partial}{\partial x_j}(\alpha u_i)\in\L{q_2}$ where $q_2 \in (\frac{3}{2},p)$. That implies $\alpha\vu\in\vW{1}{q_2}$ and hence $\alpha\vu_{\vt}\in\vWfracb{1-\frac{1}{q_2}}{q_2}$. So again by the regularity result, we have $\vu\in\vW{2}{q_2}$ where $q_2\in(\frac{3}{2},p)$.

Now $\vu\in\vW{2}{q_2}\hookrightarrow\vL{\infty}$ and $\nabla\vu\in\vW{1}{q_2}\hookrightarrow\vL{q_2^*}$. So for all $i,j=1,2,3,$
$$\frac{\partial\alpha}{\partial x_j}u_i\in\L{p} \ \text{ and } \ \alpha\frac{\partial u_i}{\partial x_j}\in\L{q_4} \ \text{ where } \ \frac{1}{q_4}=\frac{1}{p^*}+\frac{1}{q_2^*}=\frac{1}{p} + \frac{1}{q_2}-\frac{2}{3}.$$
As $q_4>p$, $\frac{\partial}{\partial x_j}(\alpha u_i)\in\L{p}$ which implies $\alpha\vu\in\vW{1}{p}$ and thus $\alpha\vu_{\vt}\in\vWfracb{1-\frac{1}{p}}{p}$. Therefore the regularity result as \cite[Theorem 4.1]{AR} gives $(\vu,\pi)\in\vW{2}{p}\times\W{1}{p}$.

The case for $p \geq 3$ follows exactly in the same way.
\hfill
\end{proof}

More generally, we have the following regularity result.

\begin{theorem}[{\bf Existence in $\boldsymbol{W}^{m,p}(\Omega)$}]
Let $m\geq 3$. Assume $\Omega$ is of class $\HC{m}{1}$ and $p\in (1,\infty)$. Moreover, let
$$
\bm{f}\in\vW{m-2}{p},\,\, \bm{h}\in\vWfracb{m-1-\frac{1}{p}}{p} \text{ and } \alpha\in\Wfracb{m-1-\frac{1}{p}}{p}.
$$
Then the solution $(\vu,\pi)$ of the problem (\ref{S}) with $\mathbb{F} = 0$, given by Corollary \ref{thm_W1p_regularity_Stokes_Nbc} belongs to $\vW{m}{p}\times\W{m-1}{p}$.
\end{theorem}

\begin{proof}
For ease of understanding, here we proof only the case $m=3$. For $m\geq 3$, the proof is exactly similar.

First assume $p<3$. Also we assume that $\alpha\in \W{2}{p}$. As $\W{1}{p}\hookrightarrow \L{p^*}$, we obtain $\vu\in\vW{2}{p^*}$ from strong regularity result in Theorem \ref{thm_W2p_regularity_Stokes_Nbc}. Note that $p^* > \frac{3}{2}$, so $\vu\in \vL{\infty}$. Now $\nabla \alpha \in \W{1}{p}$ gives $\frac{\partial \alpha}{\partial x_j}u_i \in \W{1}{p}$. Also as $\alpha\in\W{1}{p^*}$ and $\nabla \vu\in \vW{1}{p^*}$, we get $\alpha \frac{\partial u_i}{\partial x_j}\in \W{1}{p}$. Thus $\frac{\partial}{\partial x_j}(\alpha u_i)=\frac{\partial\alpha}{\partial x_j}u_i+\alpha\frac{\partial u_i}{\partial x_j} \in \W{1}{p}$ shows $\alpha \vu \in \vW{2}{p}$ which implies $\alpha\vu_{\vt} \in \vWfracb{2-\frac{1}{p}}{p}$. So from the regularity result as \cite[Theorem 4.1]{AR}, we deduce $(\vu, \pi )\in \vW{3}{p}\times \W{2}{p}.$

For $p\ge 3$, the result follows from bootstrap argument.
\hfill
\end{proof}

\section{Estimates}
\label{sec:6}
\setcounter{equation}{0}
\subsection{First estimates}
\label{sec:6sub:1}
We can deduce estimates giving precise dependence of the weak solution of (\ref{S}) on the friction coefficient $\alpha$ in some particular cases, which is better than (\ref{S43}). Note that the following result is not optimal with respect to $\alpha$ and will be improved in Theorem \ref{T2}.

\begin{proposition}
\label{estimate}
	Let $p>2$. With the same assumptions on $\bm{f}, \mathbb{F}, \bm{h}$ and $\alpha$ as in Corollary \ref{thm_W1p_regularity_Stokes_Nbc}, the solution $(\vu,\pi) \in \vW{1}{p} \times L^{p}_0(\Omega)$ of problem \eqref{S} satisfies the following bounds:\\
{\bf a)} if $\Omega$ is not axisymmetric, then
	\begin{equation}
	\label{12}
	\|\vu\|_{\vW{1}{p}} + \|\pi\|_{L^{p}(\Omega)} \leq C(\Omega,p)\left( 1 + \|\alpha \|^2_{\Lb{t(p)}}\right) \left( \|\bm{f}\|_{\vL{r(p)}}+ \|\mathbb{F}\|_{\mathbb{L}^p(\Omega)} + \|\bm{h}\|_{\vWfracb{-\frac{1}{p}}{p}} \right) .
	\end{equation}
	\medskip
{\bf b)} if $\Omega$ is axisymmetric and \\
	\indent $\left(\textit{i} \right)$ $\alpha \geq \alpha_* >0$ on $\Gamma$, then
	\begin{equation*}
	\|\vu\|_{\vW{1}{p}} + \|\pi\|_{L^{p}(\Omega)} \leq \frac{C(\Omega,p)}{\min\{2,\alpha_*\}} \left( 1 + \|\alpha \|^2_{\Lb{t(p)}}\right) \left( \|\bm{f}\|_{\vL{r(p)}}+ \|\mathbb{F}\|_{\mathbb{L}^p(\Omega)} + \|\bm{h}\|_{\vWfracb{-\frac{1}{p}}{p}} \right) .
	\end{equation*}
	\indent $\left(\textit{ii} \right)$ $\bm
	{f}, \mathbb{F}$ and $\bm{h}$ satisfy the condition:
	\begin{equation*}
	\int\displaylimits_{\Omega}{ \bm{f} \cdot \bm{\beta}} -\int\displaylimits_{\Omega}{\mathbb{F}:\nabla \bm{\beta}}+ \left\langle \bm{h}, \bm{\beta}\right\rangle_\Gamma = 0
	\end{equation*}
	\qquad then,
	\begin{equation*}
	\|\vu\|_{\vW{1}{p}} + \|\pi\|_{L^{p}(\Omega)} \leq C(\Omega,p)\left( 1 + \|\alpha \|^2_{\Lb{t(p)}}\right) \left( \|\bm{f}\|_{\vL{r(p)}}+ \|\mathbb{F}\|_{\mathbb{L}^p(\Omega)} + \|\bm{h}\|_{\vWfracb{-\frac{1}{p}}{p}} \right) .
	\end{equation*}
\end{proposition}

\begin{proof}
We only prove \eqref{12} since the other inequalities follow in the same way. Assume that $\Omega$ is not axisymmetric.
\\
\textbf{(i) $2< p <3$:} As in Lemma \ref{23}, $\alpha \vu_{\vt} \in \vLb{q}$ with $\frac{1}{q} = \frac{3}{2p}-\frac{1}{2} + \frac{1}{2+\epsilon} < \frac{3}{2p}$ and $\vLb{q}\hookrightarrow \vWfracb{-\frac{1}{p}}{p}$. Therefore, $\alpha \vu_{\vt} \in \vWfracb{-\frac{1}{p}}{p}$. But from the relation,
	$$\vLb{q} \underset{\text{compact}}{\hookrightarrow} \vWfracb{-\frac{1}{p}}{p} \underset{\text{continuous}}{\hookrightarrow} \vHfracbd{1}{2} $$
we have for any $\delta > 0$, there exists a constant $C>0$ (independent of $\delta$) such that
	\begin{equation}
	\label{10}
	\|\bm{v}\|_{\vWfracb{-\frac{1}{p}}{p}} \leq \ \delta \ \|\bm{v}\|_{\vLb{q}} + \frac{C}{\delta} \ \|\bm{v}\|_{\vHfracbd{1}{2}} \ \qquad \forall \ \bm{v}\in \vLb {q} \ . 
	\end{equation}
	Choosing $\bm{v} = \alpha \vu_{\vt}$ in \eqref{10} and using H\"{o}lder inequality and trace theorem, we get
	\begin{align*}
	\|\alpha \vu_{\vt}\|_{\vWfracb{-\frac{1}{p}}{p}} &\leq \ \delta \ \|\alpha \vu_{\vt}\|_{\vL{q}} + \frac{C}{\delta} \ \|\alpha \vu_{\vt}\|_{\vLb{4/3}}\notag\\
	&\leq \ \delta \ \|\alpha\|_{L^{2+\epsilon}(\Gamma)} \|\vu\|_{\vW{1}{p}} +\frac{C}{\delta} \ \|\alpha\|_{L^2(\Gamma)} \|\vu\|_{\vH{1}} \notag .
	\end{align*}
	Now the generalized regularity result as \cite[Corollary 3.8]{AR} yields
	\begin{align*}
	& \quad \ \|\vu\|_{\vW{1}{p}} + \|\pi\|_{L^{p}_0(\Omega)}\\
	&\leq C \left( \|\bm{f}\|_{\vL{r(p)}} + \|\mathbb{F}\|_{\mathbb{L}^p(\Omega)}+ \|\bm{h}\|_{\vWfracb{-\frac{1}{p}}{p}} + \|\alpha\vu_{\vt}\|_{\vWfracb{-\frac{1}{p}}{p}}\right) \\
	&\leq C \left( \|\bm{f}\|_{\vL{r(p)}}+ \|\mathbb{F}\|_{\mathbb{L}^p(\Omega)} + \|\bm{h}\|_{\vWfracb{-\frac{1}{p}}{p}} \right) +\delta \|\alpha\|_{L^{2+\epsilon}(\Gamma)} \|\vu\|_{\vW{1}{p}} \\
	& \qquad \qquad \qquad \qquad \qquad \qquad \qquad \qquad \qquad \qquad \quad + \frac{C}{\delta} \|\alpha\|_{L^2(\Gamma)} \|\vu\|_{\vH{1}} .
	\end{align*}
	Choosing $\delta > 0$ such that $1- \delta C_1 \|\alpha\|_{L^{2+\epsilon}(\Gamma)} = \frac{1}{2}$, we obtain
	\begin{align*}
	& \quad \|\vu\|_{\vW{1}{p}} + \|\pi\|_{L^{p}_0(\Omega)}\\
	&\leq \ C \left( \|\bm{f}\|_{\vL{r(p)}} + \|\mathbb{F}\|_{\mathbb{L}^p(\Omega)}+ \|\bm{h}\|_{\vWfracb{-\frac{1}{p}}{p}} \right) + C \ \|\alpha\|_{L^2(\Gamma)} \|\alpha\|_{L^{2+\epsilon}(\Gamma)} \|\vu\|_{\vH{1}}\\
	& \leq \ C (1 +\|\alpha\|^2_{L^{(2+ \epsilon)}(\Gamma)}) \left( \|\bm{f}\|_{\vL{r(p)}}+ \|\mathbb{F}\|_{\mathbb{L}^p(\Omega)} + \|\bm{h}\|_{\vWfracb{-\frac{1}{p}}{p}} \right).
	\end{align*}
	
	\noindent \textbf{(ii) $p \ge 3$:} The analysis is exactly similar to the previous case.
	\hfill
\end{proof}

\begin{remark}
\rm{We can also extend the above estimates of Proposition \ref{estimate} for $p<2$ by duality argument in the same way as in Proposition \ref{P1} and Proposition \ref{P2}.
}
\end{remark}

\subsection{Second estimates}
\label{sec:6sub:2}
In this subsection, we prove one of the main result of this work. We improve the estimates in Proposition \ref{estimate} with respect to $\alpha$ and for all $p\in (1,\infty)$.

First we discuss the estimate for $p>2$ with $\bm{f}=\bm{0}$ and $\bm{h}=\bm{0}$, similar to (\ref{51}) or (\ref{52}).

\begin{theorem}[{\bf Estimates in \boldmath$\W{1}{p}, p> 2$ for RHS $\mathbb{F}$}]
\label{T1}
	Let $p>2$, $\mathbb{F}\in\mathbb{L}^p(\Omega)$ and $\alpha\in\Lb{t(p)}$. Then the solution $\vu\in\vW{1}{p}$ of (\ref{S}) with $\bm{f}=\bm{0}$ and $\bm{h}=\bm{0}$ satisfies the following estimates:\\
\textbf{(i)} if $\Omega$ is not axisymmetric, then
	\begin{equation}
	\label{Lp}
	\|\vu\|_{\vW{1}{p}} \le C_p(\Omega) \ \|\mathbb{F} \|_{\mathbb{L}^p(\Omega)}
	\end{equation}
\textbf{(ii)} if $\Omega$ is axisymmetric and $\alpha\ge \alpha_* >0$, then	
	\begin{equation}
	\label{Lp.}
	\|\vu\|_{\vW{1}{p}} \le C_p(\Omega,\alpha_*) \ \|\mathbb{F} \|_{\mathbb{L}^p(\Omega)}.
	\end{equation}
\end{theorem}

The proof of the above theorem uses the weak Reverse H\"{o}lder inequality and is similar to the result done for the Laplace-Robin problem in \cite{AGG}. Since $\Omega$ is $\HC{1}{1}$, there exists some $r_0>0$ such that for any $x_0\in\Gamma$, there exists a coordinate system $(x',x_3)$ which is isometric to the usual coordinate system and a $\HC{1}{1}$ function $\psi : \R^2 \rightarrow \R$ so that
$$
B(x_0,r_0)\cap \Omega = \left\lbrace (x',x_3)\in B(x_0,r_0) : x_3 > \psi(x')\right\rbrace
$$
and
$$
B(x_0,r_0)\cap \Gamma = \left\lbrace (x',x_3)\in B(x_0,r_0) : x_3 = \psi(x')\right\rbrace .
$$

In some places, we may write $B$ instead of $B(x_0,r)$ where there is no ambiguity and $aB := B(x_0, a r)$ for $a>0$. Also for any integrable function $f$ on a domain $\omega$, we use the usual notation to denote the average of it by
\begin{equation*}
\dashint \displaylimits_{\omega} {f} = \frac{1}{|\omega|} \int\displaylimits_{\omega} {f}.
\end{equation*}

We require some further results to complete the proof of Theorem \ref{T1}. The following lemma is proved in \cite[Lemma 0.5]{modica}.

\begin{lemma}
	\label{epsilon_lemma}
	Let $f, g, h$ be non-negative functions in $L^1(Q_0)$ where $Q_0$ is a cube in $\R^n$, $Q_R(x_0)$ is a cube centered at $x_0$ with sides $2R$ and let $\beta\in\R^+$. There exists $\delta_0$ such that if for some $\delta\le \delta_0$, the following inequality
	\begin{equation*}
	\int\displaylimits_{Q_R(x_0)} {f} \le C(\delta) \left[ R^{-\beta} \int\displaylimits_{Q_{2R}(x_0)}{g}+ \int\displaylimits_{Q_{2R}(x_0)}{h} \right] + \delta \int\displaylimits_{Q_{2R}(x_0)}{f}
	\end{equation*}
	holds for all $x_0\in Q_0$ and $R< \frac{1}{2}d(x_0, \partial Q_0)$, then there exists a constant $C>0$ such that
	\begin{equation*}
	\int\displaylimits_{Q_R(x_0)} {f} \le C \left[ R^{-\beta} \int\displaylimits_{Q_{2R}(x_0)} {g}+ \int\displaylimits_{Q_{2R}(x_0)}{h} \right]
	\end{equation*}
	for all $x_0\in Q_0$ and all $R < \frac{1}{2}d(x_0, \partial Q_0)$.
\end{lemma}

Next we deduce the Caccioppoli inequality for Stokes problem, up to the boundary.
\begin{lemma}[\textbf{Caccioppoli inequality}]
Let $(\vu,\pi)$ be a weak solution of the problem
\begin{equation}
\label{var_form}
2\int\displaylimits_{\Omega}{\DT\vu: \DT\bm{\varphi}} + \int\displaylimits_{\Gamma}{\alpha \vu_{\vt}\cdot \bm{\varphi}_{\vt}} - \int\displaylimits_{\Omega}{\pi \ \div \ \bm{\varphi}} = -\int\displaylimits_{\Omega}{\mathbb{F}: \nabla\bm{\varphi}}\qquad \forall\, \bm{\varphi}\in \bm{H}^1_{\vt}(\Omega).
\end{equation}
Then there exists a constant $C>0$, independent of $\alpha$ such that for all $x_0\in \overline{\Omega}$ and $ 0<r<\frac{r_0}{2}$, we have
\begin{equation}
\label{11}
\int\displaylimits_{B\cap\Omega}{|\nabla \vu|^2} \le C \left( \frac{1}{r^2} \int\displaylimits_{2B\cap\Omega}{|\vu|^2} + \int\displaylimits_{2B\cap\Omega}{|\mathbb{F}|^2}\right) .
\end{equation}
\end{lemma}

\begin{proof}
	We will use the identity several times that for any 'smooth enough' $\bm{v}$ and any symmetric matrix $\mathbb{M}$, $\int\displaylimits_{ \omega}{\DT \bm{v} : \mathbb{M}} = \int\displaylimits_{ \omega}{\nabla \bm{v} : \mathbb{M}}$. In particular, for any 'smooth enough' $\bm{v}, \bm{\varphi}$,
	$$
	\int\displaylimits_{ \omega}{\DT \bm{v} : \DT\bm{\varphi}} = \int\displaylimits_{ \omega}{\nabla \bm{v} : \DT\bm{\varphi}}.
	$$
{\bf i) Pressure estimate:} Let $\pi_0 = \dashint\displaylimits_{2B\cap \Omega}{\pi}$. Since $\pi-\pi_0\in L^2(2B\cap\Omega)$, there exists a unique $\bm{\psi}\in \bm{H}^1(2B\cap\Omega)$ such that
\begin{equation*}
\begin{cases}
\div \ \bm{\psi}=\pi-\pi_0 &\text{ in } 2B\cap\Omega\\
\bm{\psi} = \bm{0} &\text{ on } \partial(2B\cap\Omega)
\end{cases}
\end{equation*}
satisfying
\begin{equation}
\label{psi}
\|\bm{\psi}\|_{\bm{H}^1(2B\cap\Omega)} \le C(\Omega) \ \|\pi-\pi_0\|_{L^2(2B\cap\Omega)}. 
\end{equation}
Now putting $\bm{\psi}$ as a test function in (\ref{var_form}) (extending by $0$ outside $2B\cap\Omega$, we can consider $\bm{\psi}\in\bm{H}^1_0(\Omega)$) and replacing $\pi$ by $\pi-\pi_0$ (since if $\pi$ satisfies (\ref{S}), $\pi-\pi_0$ also satisfies the same equation), we obtain
\begin{equation*}
2\int\displaylimits_{2B\cap\Omega}{\DT\vu: \DT\bm{\psi}} - \int\displaylimits_{2B\cap\Omega}{(\pi - \pi_0) \ \div \ \bm{\psi}} = -\int\displaylimits_{2B\cap\Omega}{\mathbb{F}: \nabla\bm{\psi}}
\end{equation*}
which implies
\begin{equation*}
\begin{aligned}
\int\displaylimits_{2B\cap\Omega}{|\pi - \pi_0|^2} \le \left( \|\nabla\vu\|_{\mathbb{L}^2(2B\cap\Omega)} + \|\mathbb{F}\|_{\mathbb{L}^2(2B\cap\Omega)}\right) \|\bm{\psi}\|_{\bm{H}^1(2B\cap\Omega)}
\end{aligned}
\end{equation*}
and hence, using (\ref{psi}),
\begin{equation}
\label{pressure}
\|\pi-\pi_0\|_{L^2(2B\cap\Omega)} \le C(\Omega)  \left( \|\nabla\vu\|_{\mathbb{L}^2(2B\cap\Omega)} + \|\mathbb{F}\|_{\mathbb{L}^2(2B\cap\Omega)}\right) .
\end{equation}

{\bf ii) Caccioppoli inequality:}
Consider a cut-off unction $\eta\in C_c^{\infty}(2B)$ such that
\begin{equation}
\label{eta}
0\le \eta \le 1, \quad \eta \equiv 1 \ \text{ on } B \quad \text{ and } \quad |\nabla \eta| \le \frac{C}{r} \ \text{ in } 2B.
\end{equation}
Now choosing $\bm{\varphi} = \eta ^2 \vu$ in (\ref{var_form}), we have,
\begin{equation*}
2\int\displaylimits_{2B \cap\Omega}{\DT\vu: \DT(\eta^2 \vu)} + \int\displaylimits_{2B \cap\Gamma}{\alpha \eta^2 |\vu_{\vt}|^2} - \int\displaylimits_{2B \cap\Omega}{(\pi-\pi_0) \ \div (\eta^2 \vu)} = - \int\displaylimits_{2B \cap\Omega}{\mathbb{F}: \nabla(\eta^2 \vu)}
\end{equation*}
which gives, using the fact that $\div \ \vu =0$ in $\Omega$,
\begin{equation*}
\begin{aligned}
& \quad 2\int\displaylimits_{2B \cap\Omega}{\eta^2 |\DT\vu|^2} +\int\displaylimits_{2B \cap\Gamma}{\alpha \eta^2 |\vu_{\vt}|^2}\\
=& -4\int\displaylimits_{2B \cap\Omega}{\DT\vu:\eta \nabla \eta \vu} + 2\int\displaylimits_{2B \cap\Omega}{(\pi-\pi_0) \eta \nabla \eta \vu} - \int\displaylimits_{2B \cap\Omega}{\mathbb{F}: \eta ^2 \nabla\vu} - 2\int\displaylimits_{2B \cap\Omega}{\mathbb{F}:\eta \nabla \eta \vu}
\end{aligned}
\end{equation*}
where $\nabla \eta \vu$ is the matrix $\nabla \eta \otimes \vu$.
Next using Young's inequality on the RHS yields,
\begin{equation*}
\begin{aligned}
& \quad 2\int\displaylimits_{2B \cap\Omega}{\eta^2 |\DT\vu|^2} +\int\displaylimits_{2B \cap\Gamma}{\alpha \eta^2 |\vu_{\vt}|^2} \\
\le & \quad \varepsilon \int\displaylimits_{2B \cap\Omega}{\eta ^2 |\DT\vu|^2} + C_\varepsilon \int\displaylimits_{2B \cap\Omega}{|\vu|^2|\nabla \eta |^2} + \varepsilon \int\displaylimits_{2B \cap\Omega}{\eta^2|\pi-\pi_0|^2} + C_\varepsilon \int\displaylimits_{2B \cap\Omega}{|\nabla \eta|^2|\vu|^2}\\
& + \varepsilon \int\displaylimits_{2B \cap\Omega}{\eta^2 |\nabla\vu|^2}+ C_\varepsilon \int\displaylimits_{2B \cap\Omega}{\eta ^2|\mathbb{F}|^2 } + \varepsilon \int\displaylimits_{2B \cap\Omega}{\eta^2 |\mathbb{F}|^2}+ C_\varepsilon \int\displaylimits_{2B \cap\Omega}{|\nabla \eta|^2|\vu|^2}.
\end{aligned}
\end{equation*}
Upon choosing $\varepsilon>0$ suitably and using the properties (\ref{eta}) and that $\alpha \ge 0$, we get then
\begin{equation*}
\int\displaylimits_{B \cap\Omega}{|\DT\vu|^2}\le \frac{C}{r^2}\int\displaylimits_{2B \cap\Omega}{|\vu|^2} + \varepsilon \int\displaylimits_{2B \cap\Omega}{|\pi-\pi_0|^2} + \varepsilon \int\displaylimits_{2B \cap\Omega}{|\nabla\vu|^2} + C \int\displaylimits_{2B \cap\Omega}{ |\mathbb{F}|^2}
\end{equation*}
where the constant $C>0$ is independent of $\alpha$. Now plugging the pressure estimate (\ref{pressure}) gives,
\begin{equation*}
\label{0.}
\int\displaylimits_{B\cap\Omega }{|\DT\vu|^2} \le \frac{C}{r^2}\int\displaylimits_{2B\cap\Omega }{|\vu|^2} + \varepsilon \int\displaylimits_{2B\cap\Omega }{|\nabla\vu|^2} + C \int\displaylimits_{2B\cap\Omega }{ |\mathbb{F}|^2}.
\end{equation*}
Next adding the term $\int\displaylimits_{B \cap\Omega}{|\vu|^2}$ in both sides, choosing $r\le 1$ (as $\Omega$ is bounded, we can do so) hence $\frac{1}{r}\ge 1$ and using Korn inequality, we obtain
\begin{equation*}
\int\displaylimits_{B\cap\Omega }{|\nabla\vu|^2}\le \|\vu\|^2_{\bm{H}^1(B\cap\Omega)} \le C(\Omega) \left( \frac{1}{r^2}\int\displaylimits_{2B\cap\Omega }{|\vu|^2} + \int\displaylimits_{2B\cap\Omega }{ |\mathbb{F}|^2}\right) + \varepsilon \int\displaylimits_{2B\cap\Omega }{|\nabla\vu|^2}.
\end{equation*}
Therefore, using Lemma \ref{epsilon_lemma} with $\beta = 2$, we achieve the desired inequality (\ref{11}).
\hfill
\end{proof}

We state the following boundary H\"{o}lder estimate:

\begin{proposition}
	Let $p>1$ and $\gamma \in (0,1)$. Suppose that
	\begin{equation*}
	\begin{cases}
	-\Delta \bm{v} +\nabla z = \bm{0} ,\quad \div \ \bm{v} =0 &\quad  \textup{ in } \ B(x_0,r)\cap\Omega\\
	\bm{v}\cdot \vn =0, \quad \alpha \bm{v}_{\vt} +2 [(\DT\bm{v})\vn]_{\vt} = \bm{0} &\quad \textup{ on } \ B(x_0,r)\cap\Gamma
	\end{cases}
	\end{equation*}
	for some $x_0\in \Gamma$ and $0<r<r_0$, Then for any $x,y \in B(x_0,r/2)\cap \Omega$,
	\begin{equation}
	\label{bhe}
|\bm{v}(x)-\bm{v}(y)|\le C \left( \frac{|x-y|}{r}\right) ^\gamma \left( \pppvint\displaylimits_{ \ B(x_0,r)\cap\Omega}{|\bm{v}|^2} \right) ^{1/2}
	\end{equation}
	where $C>0$ depends only on $\Omega$, but independent of $\alpha$.
\end{proposition}

\begin{proof}
It can be shown in the exact same way as done in \cite[Theorem 2.8 (a), Part II]{modica}, since we have the corresponding Caccioppoli inequality (\ref{11}) as in \cite[Theorem 2.2, Part II]{modica}.
\hfill
\end{proof}

We will also use the following Sobolev-Poincar\'{e} inequality for certain class of functions.

\begin{lemma}
	\label{Lpoincare}
	Let $\Omega$ be a $\mathcal{C}^1$ bounded domain in $\R^n$ and $1<p<n$. For any $\bm{v}\in\bm{W}^{1,p}(B(x,r)\cap\Omega)$ with $\bm{v}\cdot \vn =0$ on $B(x,r)\cap \Gamma$, the following inequality holds:
	\begin{equation}
	\label{poincare}
	\left( \ \int\displaylimits_{B(x,r)\cap\Omega}{|\bm{v}|^{p*}}\right) ^{1/p*} \le C \left( \ \int\displaylimits_{B(x,r)\cap\Omega}{|\nabla \bm{v}|^{p}} \right) ^{1/p}
	\end{equation}
	where $p*$ is the Sobolev exponent and the constant $C>0$ depends on $\Omega, n$ and $p$, but independent of $r$.
\end{lemma}

\begin{proof}
The Poincar\'{e} inequality for functions $\bm{v}\in\vH{1}$ with $\bm{v}\cdot \vn =0$ on $\Gamma$ is well-known, see for example, \cite[Proposition 2.1]{AO}. The same argument also holds for $p\ne 2$. Then the estimate (\ref{poincare}) follows combining the Sobolev estimate.

Now to see that the above constant does not depend on $r$, we use scaling argument. Without loss of generality, consider $x=0$. Let (\ref{poincare}) holds for $r=1$ \textit{i.e.}
\begin{align*}
\left( \ \int\displaylimits_{B(0,1)\cap\Omega}{|\bm{v}(y)|^{p*}\,\mathrm{d}y}\right) ^{1/p*} \le C \left( \ \int\displaylimits_{B(0,1)\cap\Omega}{|\nabla \bm{v}(y)|^{p}\,\mathrm{d}y} \right) ^{1/p}.
\end{align*}
Under the change of variable $y=\frac{z}{r}$, it becomes, denoting $\bm{w}(z)=\bm{v}(y)$,
\begin{align*}
r^{-\frac{p*}{n}}\left( \ \int\displaylimits_{B(0,r)\cap\Omega}{|\bm{w}(z)|^{p*}\,\mathrm{d}z}\right) ^{1/p*} \le C r^{-\frac{p}{n}} \left( r^p \int\displaylimits_{B(0,r)\cap\Omega}{|\nabla \bm{w}(z)|^{p}\,\mathrm{d}z} \right) ^{1/p}
\end{align*}
which yields
\begin{align*}
\left( \ \int\displaylimits_{B(0,r)\cap\Omega}{|\bm{w}(z)|^{p*}\,\mathrm{d}z}\right) ^{1/p*} \le C \left( \ \int\displaylimits_{B(0,r)\cap\Omega}{|\nabla \bm{w}(z)|^{p}\,\mathrm{d}z} \right) ^{1/p}.
\end{align*}
This proves the claim.
\hfill
\end{proof}

\begin{lemma}[\textbf{weak reverse H\"{o}lder inequality}]
	\label{Lrhi}
	Let $p\ge 2$. Then for any $B(x_0,r)$ with the property that $0<r<\frac{r_0}{8}$ and either $B(x_0,2r)\subset \Omega$ or $ x_0\in\Gamma$, the following weak Reverse H\"{o}lder inequality holds:
	\begin{equation}
	\label{rhi}
	\left( \pppvint\displaylimits_{\ B(x_0,r)\cap\Omega} {|\nabla \vu|^p}\right) ^{1/p} \le C \left( \ \pppvint\displaylimits_{B(x_0,2r)\cap\Omega} {|\nabla\vu|^2}\right) ^{1/2} 
	\end{equation}
	whenever $\vu\in \bm{H}^1(B(x_0,2r)\cap \Omega)$ satisfies
	\begin{equation*}
	\begin{cases}
	-\Delta \vu +\nabla \pi = \bm{0} ,\quad \div \ \vu =0 \quad &\textup{ in } \ B(x_0,2r)\cap\Omega\\
	\vu\cdot \vn =0, \quad \alpha \vu_{\vt} +2 [(\DT\vu)\vn]_{\vt} = \bm{0} &\textup{ on } \ B(x_0,2r)\cap\Gamma \quad(\text{if } x_0\in\Gamma).
	\end{cases}
	\end{equation*}
	The constant $C>0$ at most depends on $\Omega$ and $p$.
\end{lemma}

\begin{proof}
	{\bf case(i) :} $B(x_0,2r)\subset \Omega$.\\
	The weak reverse H\"{o}lder inequality (\ref{rhi}) holds for any $p\ge 2$, by the following interior estimates for Stokes operator \cite[Theorem 2.7 (a)]{GX}:
	\begin{equation*}
	\underset{B(x_0,r)}{\sup} \ |\nabla \vu|\le C \left(\  \pvint\displaylimits_{B(x_0,2r)} {|\nabla \vu|^2}\right) ^{1/2}.
	\end{equation*}
	
	{\bf case(ii) :} $x_0\in\Gamma$.\\
	From the interior gradient estimate for Stokes problem, we can write (eg. see \cite[Theorem 2.7, (3)]{GX})
	\begin{equation*}
	|\nabla \vu(x)|\le \frac{C}{\delta(x)} \left( \pppvint\displaylimits_{ \ B(x, c\delta(x))}{|\vu|^2}\right) ^{1/2}.
	\end{equation*}
	Now for fixed $y\in B(x,2c\delta(x))$, let $\bm{v}(x) = \vu(x) - \vu(y)$. Then $-\Delta \bm{v} +\nabla z = \bm{0} ,\ \div \ \bm{v} =0$ in $B(x,2c\delta(x))$ and thus we may write from the above argument,
	\begin{align*}
	|\nabla \bm{v}(x)|\le \frac{C}{\delta(x)} \left( \pppvint\displaylimits_{ \ B(x,c\delta(x))}{|\bm{v}|^2}\right) ^{1/2}
	\end{align*}
	which gives, along with the boundary H\"{o}lder estimate (\ref{bhe}),
	\begin{align}
	|\nabla \vu(x)| &\le \frac{C}{\delta(x)} \left( \pppvint\displaylimits_{ \ \, B(x,c\delta(x))}{|\vu(z) - \vu(y)|^2\mathrm{d}z}\right) ^{1/2}\nonumber\\
	& = \frac{C}{\delta(x)^{1+\frac{3}{2}}} \left( \ \int\displaylimits_{B(x,c\delta(x))}{|\vu(z) - \vu(y)|^2\mathrm{d}z}\right) ^{1/2}\nonumber\\
	& \le \frac{C}{\delta(x)^{1+\frac{3}{2}}} \left[ \ \int\displaylimits_{B(x,2c\delta(x))}\left( \frac{|z-y|}{r}\right) ^{2\gamma} \ \left( \ \pppvint\displaylimits_{B(x_0,2r)\cap\Omega}{|\vu|^2}\right) \mathrm{d}z\right] ^{1/2}\nonumber\\
	& \le \frac{C}{\delta(x)^{1+\frac{3}{2}}} \left( \ \pppvint\displaylimits_{B(x_0,2r)\cap\Omega}{|\vu|^2}\right)^{1/2} \frac{1}{r^\gamma} \left( \ \int\displaylimits_{B(x,2c\delta(x))}{|z-y|^{2\gamma}\mathrm{d}z }\right)^{1/2}.\label{0}
	\end{align}
	Let us now calculate the last integral in the last inequality. Substituting $w = \frac{z-y}{4c\delta(x)}$, we get
	\begin{align*}
	\int\displaylimits_{B(x,2c\delta(x))}{|z-y|^{2\gamma}\mathrm{d}z } \le C \int\displaylimits_{B(0,1)} {w^{2\gamma} (\delta(x))^{2\gamma +3} \mathrm{d}w}
	& = C \ (\delta(x))^{2\gamma +3} \int\displaylimits_{0}^{1}\int\displaylimits_0^{2\pi} {r^{2\gamma} r^2 \,\,\mathrm{d}r \,\,\mathrm{d}\theta}\\
	& = C \ (\delta(x))^{2\gamma +3} \ \frac{2\pi}{2\gamma +3} \le C_\gamma (\delta(x))^{2\gamma +3}.
	\end{align*}
	Plugging the value of the above integral in (\ref{0}), along with the Sobolev-Poincar\'{e} inequality (\ref{poincare}), we then obtain
	\begin{align*}
	|\nabla \vu(x)| & \le \frac{C_\gamma}{(\delta(x))^{1+\frac{3}{2}}} \left( \ \pppvint\displaylimits_{B(x_0,2r)\cap\Omega}{|\vu|^2}\right)^{1/2} \frac{1}{r^\gamma} (\delta(x))^{\gamma + \frac{3}{2}}\\
	& = C_\gamma \ \frac{(\delta(x))^{\gamma -1}}{r^\gamma} r^{-3/2} \left( \ \int\displaylimits_{B(x_0,2r)\cap\Omega}{|\vu|^2}\right) ^{1/2}\\
	& \le C_\gamma \ \frac{(\delta(x))^{\gamma -1}}{r^\gamma} r^{1-3/2} \left( \ \int\displaylimits_{B(x_0,2r)\cap\Omega}{|\vu|^6}\right)^{1/6}\\
	& \le C_\gamma \left( \frac{r}{\delta(x)}\right) ^{1-\gamma}\left( \ \pppvint\displaylimits_{B(x_0,2r)\cap\Omega}{|\nabla \vu|^2}\right)^{1/2}.
	\end{align*} 
	Since $\gamma \in (0,1)$ is arbitrary, we thus have,
	\begin{equation*}
	|\nabla \vu(x)| \le C_\gamma \left( \frac{r}{\delta(x)}\right) ^{\gamma}\left( \ \pppvint\displaylimits_{B(x_0,2r)\cap\Omega}{|\nabla \vu|^2}\right)^{1/2}.
	\end{equation*}	
	Finally it yields choosing $\gamma$ so that $p\gamma <1$,
	\begin{equation*}
	\left( \pppvint\displaylimits_{ \ B(x_0,r)\cap \Omega}{|\nabla \vu|^p}\right) ^{1/p} \le C_p \left( \ \pppvint\displaylimits_{B(x_0,2r)\cap\Omega}{|\nabla \vu|^2}\right)^{1/2}.
	\end{equation*}
	This completes the proof.	
	\hfill
\end{proof}

With the following abstract lemma which is proved in \cite[Theorem 2.2]{geng}, we are now in a position to prove Theorem \ref{T1}.

\begin{lemma}
	\label{G0}
	Let $\Omega$ be a bounded Lipschitz domain in $\R^3$ and $p>2$. Let $G\in \L{2}$ and $f\in\L{q}$ for some $2<q<p$. Suppose that for each ball $B$ with the property that $|B|\le \beta |\Omega|$ and either $2B \subset \Omega$ or $B$ centers on $\Gamma$, there exist two integrable functions $G_B$ and $R_B$ on $2B\cap \Omega$ such that $|G|\le |G_B| + |R_B|$ on $2B\cap \Omega$ and
	\begin{equation}
	\label{G1}
	\begin{aligned}
	& \left(\frac{1}{|2B\cap \Omega|}\int\displaylimits_{2B\cap \Omega}{|R_B|^p} \right) ^{1/p}\\
	& \le C_1 \left[ \left(\frac{1}{|\gamma B \cap \Omega|}\int\displaylimits_{\gamma B \cap \Omega}{|G|^2}\right)^{1/2} + \underset{B \subset B'}{sup} \left(\frac{1}{|B'\cap \Omega|}\int\displaylimits_{B'\cap \Omega}{|f|^2}\right)^{1/2}\right] 
	\end{aligned}
	\end{equation}
	and
	\begin{equation}
	\label{G2}
	\left(\frac{1}{|2B\cap \Omega|}\int\displaylimits_{2B\cap \Omega}{|G_B|^2} \right) ^{1/2} \le C_2 \ \underset{B \subset B'}{sup} \left(\frac{1}{|B'\cap \Omega|}\int\displaylimits_{B'\cap \Omega}{|f|^2}\right)^{1/2}
	\end{equation}
	where $C_1, C_2 >0$ and $0<\beta <1<\gamma$. Then we have,
	\begin{equation}
	\label{G3}
	\left(\frac{1}{|\Omega|}\int\displaylimits_{ \Omega}{|G|^q} \right) ^{1/q} \le C \left[ \left(\frac{1}{|\Omega|}\int\displaylimits_{\Omega}{|G|^2}\right)^{1/2} + \left(\frac{1}{|\Omega|}\int\displaylimits_{\Omega}{|f|^q}\right)^{1/q}\right]
	\end{equation}
	where $C>0$ depends only on $C_1, C_2, n, p, q, \beta, \gamma$ and $\Omega$. 
\end{lemma}

\begin{proof}[\textbf{Proof of Theorem \ref{T1}}]
	Given any ball $B$ with either $2B \subset \Omega$ or $B$ centers on $\Gamma$, let $\varphi\in C_c^\infty(8B)$ is a cut-off function such that $0\le \varphi \le 1$ and
	\begin{equation*}
	\varphi =
	\begin{cases}
	& 1 \quad \text{ on } 4B\\
	& 0 \quad \text{ outside } 8B
	\end{cases}
	\end{equation*}
	and we decompose $(\vu,\pi)=(\bm{v},\pi_1)+(\bm{w},\pi_2)$ where $(\bm{v},\pi_1),(\bm{w},\pi_2)$ satisfy
	\begin{equation}
	\label{01}
	\left\{
	\begin{aligned}
	-\Delta\bm{v} +\nabla \pi_1= \div \ (\varphi\mathbb{F}),\quad \div\;\bm{v}=0 \ &\text{ in $\Omega$}\\
	\bm{v}\cdot\vn=0, \quad 2\left[(\DT\bm{v})\vn \right]_{\vt}+\alpha\bm{v}_{\vt}=- [(\varphi\mathbb{F})\vn]_{\vt} \ &\text{ on $\Gamma$}
	\end{aligned}
	\right.
	\end{equation}
	and
	\begin{equation}
	\label{02}
	\left\{
	\begin{aligned}
	-\Delta\bm{w} +\nabla \pi_2= \div \ ((1-\varphi)\mathbb{F}),\quad \div\;\bm{w}=0 \ &\text{ in $\Omega$}\\
	\bm{w}\cdot\vn=0, \quad 2\left[(\DT\bm{w})\vn \right]_{\vt}+\alpha\bm{w}_{\vt}=- [((1-\varphi)\mathbb{F})\vn]_{\vt} \ &\text{ on $\Gamma$}.
	\end{aligned}
	\right.
	\end{equation}
	Multiplying (\ref{01}) by $\bm{v}$ and integrating by parts, we get,
	\begin{equation*}
	\int\displaylimits_\Omega{|\nabla \bm{v}|^2} + \int\displaylimits_\Gamma {\alpha |\bm{v}_{\vt}|^2} =- \int\displaylimits_\Omega{\varphi \mathbb{F}: \nabla \bm{v}}
	\end{equation*}
	which gives
	\begin{equation}
	\label{L2.}
	\|\nabla \bm{v}\|_{\mathbb{L}^{2}(\Omega)} \le \|\varphi\mathbb{F}\|_{\mathbb{L}^{2}(\Omega)} .
	\end{equation}
	
	{\bf (i)} First we consider the case $4B\subset \Omega$. We want to apply Lemma \ref{G0} with $G = |\nabla \vu|, G_B = |\nabla
	\bm{v}|$ and $ R_B = |\nabla \bm{w}|$. It is easy to see that
	$$
	|G| \le |G_B| + |R_B| .
	$$
	Now we verify (\ref{G1}) and (\ref{G2}). For that, using (\ref{L2.}) we get,
	\begin{equation*}
	\begin{aligned}
	\frac{1}{|2B|} \int\displaylimits_{2B}{|G_B|^2} = \frac{1}{|2B|} \int\displaylimits_{2B}{|\nabla \bm{v}|^2} \le \frac{1}{| 2B\cap\Omega|} \int\displaylimits_{ \Omega}{|\nabla \bm{v}|^2}
	& \le \frac{1}{|2B\cap \Omega|} \int\displaylimits_{\Omega}{|\varphi \mathbb{F}|^2} \\
	& \le \frac{C(\Omega)}{|8B\cap \Omega|} \int\displaylimits_{8B \cap\Omega}{| \mathbb{F}|^2}
	\end{aligned}
	\end{equation*}
	where in the last inequality, we used that $|8B\cap\Omega| \le |\Omega|$. This gives the estimate (\ref{G2}).
	
	Next, from (\ref{02}), we observe that $-\Delta\bm{w} +\nabla \pi_2= \bm{0},\ \div\;\bm{w}=0 \text{ in } 4B$. Hence, by the weak reverse H\"{o}lder inequality in Lemma \ref{Lrhi} (using $2B$ instead of $B$), we have
	\begin{equation*}
	\left( \frac{1}{|2B|} \int\displaylimits_{2B}{|\nabla \bm{w}|^p}\right) ^{1/p} \le C_p(\Omega) \left( \frac{1}{|4B|}\int\displaylimits_{4B }{|\nabla \bm{w}|^2}\right) ^{1/2} 
	\end{equation*}
	which implies together with (\ref{L2.}),
	\begin{equation*}
	\begin{aligned}
	\left( \frac{1}{|2B|} \int\displaylimits_{2B}{|R_B|^p}\right) ^{1/p}
	&\le C_p(\Omega) \left( \frac{1}{|4B|}\int\displaylimits_{4B }{|\nabla \bm{w}|^2}\right) ^{1/2} \\
	& \le C_p(\Omega) \left[ \left( \frac{1}{|4B|}\int\displaylimits_{4B }{|\nabla \vu|^2}\right) ^{1/2} + \left( \frac{1}{|4B|}\int\displaylimits_{4B }{|\nabla \bm{v}|^2}\right) ^{1/2} \right] \\
	& \le C_p(\Omega) \left( \frac{1}{|4B|}\int\displaylimits_{4B }{|G|^2}\right) ^{1/2} +  \left( \frac{1}{|8B\cap \Omega|}\int\displaylimits_{8B \cap\Omega}{|\mathbb{F}|^2}\right) ^{1/2}.
	\end{aligned}
	\end{equation*}
	This gives (\ref{G1}). So from (\ref{G3}), it follows that
	\begin{equation*}
	\left(\frac{1}{|\Omega|}\int\displaylimits_{ \Omega}{|\nabla \vu|^q} \right) ^{1/q} \le C_p(\Omega) \left[ \left(\frac{1}{|\Omega|}\int\displaylimits_{\Omega}{|\nabla \vu|^2}\right)^{1/2} + \left(\frac{1}{|\Omega|}\int\displaylimits_{\Omega}{|\mathbb{F}|^q}\right)^{1/q}\right]
	\end{equation*}
	for any $2<q<p$ where $C_p(\Omega)>0$ does not depend on $\alpha$.
	
	Because of the self-improving property of the weak Reverse H\"{o}lder condition (\ref{rhi}), the above estimate holds for any $q\in(2,\tilde{p})$ for some $\tilde{p}>p$ also and in particular, for $q=p$, which along with the $\bm{L}^2$-estimate (\ref{51}) implies (\ref{Lp}).\\
	
	{\bf (ii)} Next consider $B$ centers on $\Gamma$. We apply again Lemma \ref{G0} with $G=|\nabla \vu|, G_B =|\nabla \bm{v}|$ and $ R_B=|\nabla \bm{w}|$, which yields
	$|G| \le |G_B| + |R_B|
	$
	and (\ref{G2}) as before. Also now $\bm{w}$ satisfies the problem
	\begin{equation*}
	\left\{
	\begin{aligned}
	-\Delta \bm{w} +\nabla \pi_2 = \bm{0} ,\quad \div \ \bm{w} =0 \quad &\text{ in } \ 4B\cap\Omega\\
	\bm{w}\cdot \vn =0, \quad \alpha \bm{w}_{\vt} +2 [(\DT\bm{w})\vn]_{\vt} = \bm{0}\quad &\text{ on } 4B \cap\Gamma.
	\end{aligned}
	\right.
	\end{equation*}
	So by the weak reverse H\"{o}lder inequality in Lemma \ref{Lrhi} and the estimate (\ref{L2.}), we obtain
	(\ref{G1}) as in the previous case. Thus we get from (\ref{G3}),
	$$
	\left(\frac{1}{|\Omega|}\int\displaylimits_{ \Omega}{|\nabla \vu|^q} \right) ^{1/q} \le C_p(\Omega) \left[ \left(\frac{1}{|\Omega|}\int\displaylimits_{\Omega}{|\nabla \vu|^2}\right)^{1/2} + \left(\frac{1}{|\Omega|}\int\displaylimits_{\Omega}{|\mathbb{F}|^q}\right)^{1/q}\right]
	$$
	for any $2<q<p$. This completes the proof together with the $\bm{L}^2$-estimate (\ref{52}).
	\hfill
\end{proof}	

The next proposition will be used to study the complete Stokes problem (\ref{S}). We will improve the following result in Proposition \ref{P2} where we consider data less regular.

\begin{proposition}[{\bf Estimates in \boldmath$\W{1}{p}, p> 2$ for RHS ${f}$}]
\label{P0}
	Let $p>2$, $\bm{f}\in\vL{p}$ and $\alpha\in\Lb{t(p)}$. Then the unique solution $(\vu,\pi)\in\vW{1}{p}\times L^p_0(\Omega)$ of (\ref{S}) with $\mathbb{F}=0$ and $\bm{h} =\bm{0}$, satisfies the following estimates:\\
\textbf{(i)} if $\Omega$ is not axisymmetric, then
	\begin{equation}
	\label{Lp0}
	\|\vu\|_{\vW{1}{p}} + \|\pi\|_{\L{p}} \le C_p(\Omega) \ \|\bm{f} \|_{\vL{p}} 
	\end{equation}
\textbf{(ii)} if $\Omega$ is axisymmetric and $\alpha \ge \alpha_* >0$, then	
	\begin{equation}
	\label{Lp0.}
	\|\vu\|_{\vW{1}{p}} + \|\pi\|_{\L{p}}\le C_p(\Omega,\alpha_*) \ \|\bm{f} \|_{\vL{p}} .
	\end{equation}
\end{proposition}

\begin{proof}
	The result follows using the same argument as in Theorem \ref{T1} and hence we do not repeat it. Note that as $\pi\in L^p_0(\Omega)$,
	\begin{equation}
	\label{pressure1}
\|\pi\|_{\L{p}}\le C \|\nabla \pi \|_{\vW{-1}{p}} \le C \left(\|\vu\|_{\vW{1}{p}} + \|\bm{f}\|_{\vW{-1}{p}} \right) \le C \|\bm{f}\|_{\vL{p}}.	
	\end{equation}
	\hfill
\end{proof}

\begin{proposition}[{\bf Estimates in \boldmath$\W{1}{p}$ for RHS $\mathbb{F}$}]
	\label{P1}
	Let $p\in(1,\infty)$, $\mathbb{F}\in\mathbb{L}^p(\Omega)$ and $\alpha\in\Lb{t(p)}$. Then the solution $(\vu,\pi)\in\vW{1}{p}\times L^p_0(\Omega)$ of (\ref{S}) with $\bm{f}=\bm{0}$ and $\bm{h}= \bm{0}$ satisfies the following estimates:\\
\textbf{(i)} if $\Omega$ is not axisymmetric, then
	\begin{equation}
	\label{Lp2}
	\|\vu\|_{\vW{1}{p}} + \|\pi\|_{\L{p}}\le C_p(\Omega) \ \|\mathbb{F} \|_{\mathbb{L}^p(\Omega)} 
	\end{equation}
\textbf{(ii)} if $\Omega$ is axisymmetric and $\alpha\ge \alpha_*>0$, then
\begin{equation}
\label{Lp2..}
\|\vu\|_{\vW{1}{p}}+ \|\pi\|_{\L{p}} \le C_p(\Omega,\alpha_*) \ \|\mathbb{F} \|_{\mathbb{L}^p(\Omega)} .
\end{equation}
\end{proposition}

\begin{proof}
	For $p>2$, the estimates (\ref{Lp2}) and (\ref{Lp2..}) are proved in Theorem \ref{T1}. Now suppose that $1<p<2$. We prove it in two steps. Also without loss of generality, we consider $\Omega$ is not axisymmetric.
	
	{\bf (i)} We first show that
	\begin{equation}
	\label{S6s2E3}
	\|\nabla \vu\|_{\vL{p}}\le C_p(\Omega)\|\mathbb{F}\|_{\mathbb{L}^p(\Omega)}.
	\end{equation}
	For that, we write
	\begin{equation}
	\label{S6s2E2}
	\|\nabla \vu\|_{\vL{p}} = \underset{0\neq \mathbb{G}\in\mathbb{L}^{p'}(\Omega)}{\sup}\frac{|\int\displaylimits_\Omega{\nabla\vu:\mathbb{G} }|}{\|\mathbb{G}\|_{\mathbb{L}^{p'}(\Omega)}}.
	\end{equation}
	Now, for any matrix $\mathbb{G}\in (\mathcal{D}(\Omega))^{3\times 3}$, let $\left( \bm{v}, \tilde{\pi}\right) \in\vW{1}{p'}\times L^{p'}_0(\Omega)$ be the solution of
	\begin{equation*}
	\begin{cases}
	-\Delta \bm{v} + \nabla \tilde{\pi}= \div \ \mathbb{G}, \quad
	\div\;\bm{v}=0 \ &\text{ in $\Omega$}\\
	\bm{v}\cdot\vn=0, \quad \left[(2\DT\bm{v}+\mathbb{G})\vn\right]_{\vt}+\alpha\bm{v}_{\vt}=\bm{0} \ &\text{ on $\Gamma$}.	
	\end{cases}
	\end{equation*}
	Since $p'>2$, from Theorem \ref{T1}, we have
	$$
	\|\bm{v}\|_{\vW{1}{p'}} \le C_p(\Omega) \|\mathbb{G}\|_{\mathbb{L}^{p'}(\Omega)} .
	$$
	Also if $\vu\in\vW{1}{p}$ is the solution of (\ref{S}) with $\bm{f}=\bm{0}, \bm{h}= \bm{0}$, using the weak formulation of the problems satisfied by $\vu$ and $\bm{v}$, we obtain
	$$
	-\int\displaylimits_\Omega{\mathbb{F} : \nabla\bm{v}} = 2\int\displaylimits_\Omega{\DT\vu : \DT\bm{v}} + \int\displaylimits_\Gamma{\alpha \vu_{\vt} \cdot \bm{v}_{\vt}} =- \int\displaylimits_\Omega{\mathbb{G} : \nabla\vu}
	$$
	which gives,
	$$
	|\int\displaylimits_\Omega{\bm{G} : \nabla\vu}| \le \|\mathbb{F}\|_{\mathbb{L}^p(\Omega)} \| \nabla\bm{v}\|_{\vL{p'}} \le C_p(\Omega) \|\mathbb{F}\|_{\mathbb{L}^p(\Omega)} \|\mathbb{G}\|_{\mathbb{L}^{p'}(\Omega)}
	$$
	and hence (\ref{S6s2E3}) follows from (\ref{S6s2E2}).
	
	{\bf (ii)} Next we prove that 
	\begin{equation}
	\label{Lp2.}
	\|\vu\|_{\vL{p}} \le C_p(\Omega) \|\mathbb{F}\|_{\mathbb{L}^p(\Omega)}.
	\end{equation}
	For that, we write similarly
	\begin{equation}
	\label{S6s2E4}
		\|\vu\|_{\vL{p}} = \underset{0\neq \bm{\varphi} \in\vL{p'}}{\sup}\frac{|\int\displaylimits_\Omega{ \vu\cdot\bm{\varphi} }|}{\|\bm{\varphi}\|_{\vL{p'}}}.
	\end{equation}
	 From Proposition \ref{P0}, we get for any $\bm{\varphi}\in\vL{p'}$, the unique solution $\left( \bm{w},\tilde{\pi}\right) \in \vW{1}{p'}\times L^{p'}_0(\Omega)$ of the problem
	\begin{equation}
	\label{S6e0}
	\begin{cases}
	-\Delta \bm{w} + \nabla \tilde{\pi}= \bm{\varphi}, \quad
	\div\;\bm{w}=0 \ &\text{ in $\Omega$}\\
	\bm{w}\cdot\vn=0, \quad 2\left[(\DT\bm{w})\vn\right]_{\vt}+\alpha\bm{w}_{\vt}=\bm{0} \ &\text{ on $\Gamma$}
	\end{cases}
	\end{equation}
	satisfies
	\begin{equation}
	\label{S6e1}
	\|\bm{w}\|_{\vW{1}{p'}}\le C_p(\Omega) \ \|\bm{\varphi}\|_{\vL{p'}}.
	\end{equation}
	Therefore using the weak formulation of the problems satisfied by $\vu$ and $\bm{w}$, we get,
	\begin{align*}
	\int\displaylimits_\Omega{\vu\cdot \bm{\varphi}} = \int\displaylimits_\Omega{ \vu\cdot \left(- \Delta \bm{w}+ \nabla \tilde{\pi}\right) }& = 2\int\displaylimits_\Omega{ \DT\vu: \DT\bm{w}} - 2\int\displaylimits_\Gamma{\vu \cdot (\DT\bm{w})\vn}\\
	&= 2\int\displaylimits_\Omega{\DT\vu : \DT\bm{w}} + \int\displaylimits_\Gamma{\alpha \vu_{\vt}\cdot \bm{w}_{\vt}} = -\int\displaylimits_\Omega{\mathbb{F} : \nabla\bm{w}}
	\end{align*}
	which implies (\ref{Lp2.}) from the relation (\ref{S6s2E4}), along with (\ref{S6e1}).
	
	 The bound on pressure follows as in the previous proposition.
	This completes proof.
	\hfill
\end{proof}

Now we study the complete problem (\ref{S}).

\begin{theorem}[{\bf Complete estimates in \boldmath$\W{1}{p}$}]
	\label{T2}
	Let $p\in(1,\infty)$ and
	$$
	\bm{f}\in \vL{r(p)}, \mathbb{F}\in\mathbb{L}^p(\Omega), \bm{h}\in \vWfracb{-\frac{1}{p}}{p}, \alpha\in\Lb{t(p)}.
	$$
	Then the solution $(\vu,\pi)\in\vW{1}{p}\times L^p_0(\Omega)$ of (\ref{S}) satisfies the following estimates:\\
\textbf{(i)} if $\Omega$ is not axisymmetric, then
	\begin{equation}
	\label{Lp1}
	\|\vu\|_{\vW{1}{p}} +\|\pi\|_{\L{p}}\le C_p(\Omega) \left( \|\bm{f} \|_{\vL{r(p)}} + \|\mathbb{F}\|_{\mathbb{L}^p(\Omega)} + \|\bm{h}\|_{\vWfracb{-\frac{1}{p}}{p}}\right) 
	\end{equation}
\textbf{(ii)} if $\Omega$ is axisymmetric and $\alpha \ge \alpha_* >0$, then
\begin{equation}
\label{Lp1.}
\|\vu\|_{\vW{1}{p}} +\|\pi\|_{\L{p}}\le C_p(\Omega,\alpha_*) \left( \|\bm{f} \|_{\vL{r(p)}} + \|\mathbb{F}\|_{\mathbb{L}^p(\Omega)} + \|\bm{h}\|_{\vWfracb{-\frac{1}{p}}{p}}\right) .
\end{equation}
\end{theorem}

To prove the above theorem, we also need the following proposition:
\begin{proposition}[{\bf Estimates in \boldmath$\W{1}{p}$ with RHS $f$ and $h$}]
	\label{P2}
	Let $p\in(1,\infty)$ and
	$$
	\bm{f}\in \vL{r(p)}, \bm{h}\in \vWfracb{-\frac{1}{p}}{p}, \alpha\in\Lb{t(p)}.
	$$
	Then the solution $(\vu,\pi)\in\vW{1}{p}\times L^p_0(\Omega)$ of (\ref{S}) with $\mathbb{F}=0$ satisfies the following estimates:\\
\textbf{(i)} if $\Omega$ is not axisymmetric, then
	\begin{equation}
	\label{Lp3}
	\|\vu\|_{\vW{1}{p}} +\|\pi\|_{\L{p}}\le C_p(\Omega) \left( \|\bm{f} \|_{\vL{r(p)}} + \|\bm{h}\|_{\vWfracb{-\frac{1}{p}}{p}}\right) 
	\end{equation}
\textbf{(ii)} if $\Omega$ is axisymmetric and $\alpha\ge \alpha_* >0$, then
\begin{equation}
\label{Lp3.}
\|\vu\|_{\vW{1}{p}} +\|\pi\|_{\L{p}}\le C_p(\Omega,\alpha_*) \left( \|\bm{f} \|_{\vL{r(p)}} + \|\bm{h}\|_{\vWfracb{-\frac{1}{p}}{p}}\right) .
\end{equation}
\end{proposition}

\begin{proof}
Without loss of generality, we only consider the case $\Omega$ is not axisymmetric. The proof is similar to that of Proposition \ref{P0} with obvious modification.\\
	{\bf(i)} To show
	\begin{equation}
	\label{Lp6}
	\begin{aligned}
	\|\nabla\vu\|_{\vL{p}} \le  C_p(\Omega) \left( \|\bm{f}\|_{\vL{r(p)}}+ \|\bm{h}\|_{\vWfracb{-\frac{1}{p}}{p}}\right)
	\end{aligned}
	\end{equation}
we write
\begin{equation}
\label{S6s2E5}
	\|\nabla\vu\|_{\vL{p}} = \underset{0\neq \mathbb{G}\in \mathbb{L}^{p'}(\Omega)}{\sup}\frac{\left| \int\displaylimits_{\Omega}{\nabla\vu:\mathbb{G}}\right| }{\|\mathbb{G}\|_{\mathbb{L}^{p'}(\Omega)}}.
\end{equation}
For any matrix $\mathbb{G}\in (\mathcal{D} (\Omega))^{3\times 3}$, let $\left( \bm{v}, \tilde{\pi}\right) \in\vW{1}{p'}\times L^{p'}_0(\Omega)$ be the solution of
	\begin{equation*}
	\begin{cases}
	-\Delta \bm{v} + \nabla \tilde{\pi}= \div \ \mathbb{G}, \quad
	\div\;\bm{v}=0 \ &\text{ in $\Omega$}\\
	\bm{v}\cdot\vn=0, \quad \left[(2\DT\bm{v}+\mathbb{G})\vn\right]_{\vt}+\alpha\bm{v}_{\vt}=\bm{0} \ &\text{ on $\Gamma$}.	
	\end{cases}
	\end{equation*}
which satisfies estimate (by Proposition \ref{P1})
	$$
	\|\bm{v}\|_{\vW{1}{p'}} \le C_p(\Omega) \|\mathbb{G}\|_{\mathbb{L}^{p'}(\Omega)} .
	$$
	Also, if $(\vu,\pi)\in\vW{1}{p}\times L^p_0(\Omega)$ is a solution of (\ref{S}) with $\mathbb{F} = 0$, from the weak formulation of the problems satisfied by $\vu$ and $\bm{v}$, we get
	$$
	-\int\displaylimits_{\Omega}{\mathbb{G}: \nabla\vu} = 2\int\displaylimits_{\Omega}{\DT\vu:\DT\bm{v}} + \int\displaylimits_{\Gamma}{\alpha \vu_{\vt}\cdot \bm{v}_{\vt}} = \int\displaylimits_{\Omega}{\bm{f}\cdot \bm{v}}+\left\langle \bm{h} ,\bm{v}\right\rangle _{\Gamma} .
	$$
	This implies, together with the embedding $\vW{1}{p'}\hookrightarrow \vL{(r(p))'}$ for all $p\in (1,\infty)$ (follows from the definition of $r(p)$),
	\begin{equation*}
	\begin{aligned}
	|\int\displaylimits_\Omega{\mathbb{G}: \nabla\vu}| &\le \|\bm{f}\|_{\vL{r(p)}} \|\bm{v}\|_{\vL{(r(p))'}} + \|\bm{h}\|_{\vWfracb{-\frac{1}{p}}{p}} \|\bm{v}\|_{\vWfracb{\frac{1}{p}}{p'}}\\
	&\le C_p(\Omega) \left( \|\bm{f}\|_{\vL{r(p)}}+ \|\bm{h}\|_{\vWfracb{-\frac{1}{p}}{p}}\right)  \|\bm{v}\|_{\vW{1}{p'}}.
	\end{aligned}
	\end{equation*}
Therefore, (\ref{Lp6}) follows from (\ref{S6s2E5}).
	
{\bf (ii)} Next we prove the following bound as done in (\ref{Lp2.}):
	\begin{equation}
	\label{Lp5}
	\|\vu\|_{\vL{p}} \le C_p(\Omega) \left( \|\bm{f}\|_{\vL{r(p)}}+ \|\bm{h}\|_{\vWfracb{-\frac{1}{p}}{p}}\right)
	\end{equation}
	except that we do not need to assume $p<2$ here as in (\ref{Lp2.}). 
Having
\begin{equation}
	\|\vu\|_{\vL{p}} = \underset{0\neq \bm{\varphi} \in\vL{p'}}{\sup}\frac{|\int\displaylimits_\Omega{ \vu\cdot\bm{\varphi} }|}{\|\bm{\varphi}\|_{\vL{p'}}},
\end{equation}
there exists a unique 
	$\left( \bm{w},\tilde{\pi}\right) \in \vW{1}{p'}\times L^{p'}_0(\Omega)$ of the problem (\ref{S6e0}) for any $ \bm{\varphi} \in\vL{p'}$ satisfying the estimate (\ref{S6e1}) (For $p<2$, the estimate (\ref{S6e1}) can be proved by the exact same argument as in Proposition \ref{P1}). Thus we can write,
	\begin{equation*}
	\int\displaylimits_{\Omega}{\vu \cdot \bm{\varphi}} =\int\displaylimits_{\Omega}{\vu\cdot\left( -\Delta \bm{w} + \nabla \tilde{\pi}\right) } =2 \int\displaylimits_{\Omega}{\DT\vu:\DT\bm{w}} + \int\displaylimits_{\Gamma}{\alpha \vu_{\vt}\cdot \bm{w}_{\vt}} = \int\displaylimits_{\Omega}{\bm{f}\cdot \bm{w} } + \left\langle \bm{h},\bm{w}\right\rangle _{\Gamma} 
	\end{equation*}
	which yields (\ref{Lp5}) as before.
For the pressure estimate, we proceed as in (\ref{pressure1}).
	\hfill
\end{proof}

\begin{proof}[\bf Proof of Theorem  \ref{T2}]
	Let $\bm{u}_1\in\vW{1}{p}$ be the weak solution of
	\begin{equation*}
	\begin{cases}
	-\Delta \bm{u}_1 +\nabla \pi_1=\div\, \mathbb{F},\quad \div\;\bm{u}_1=0 \ &\text{ in $\Omega$}\\
	\bm{u}_1\cdot\vn=0, \quad \left[(2\DT\bm{u}_1+\mathbb{F})\vn\right]_{\vt}+\alpha\bm{u}_{1\vt}=\bm{0} \ &\text{ on $\Gamma$}.
	\end{cases}
	\end{equation*}
	given by Proposition \ref{P1} and $\bm{u}_2\in\vW{1}{p}$ be the weak solution of
	\begin{equation*}
	\begin{cases}
	-\Delta\bm{u}_2 +\nabla \pi_2=\bm{f},\quad \div\;\bm{u}_2=0 \ &\text{ in $\Omega$}\\
	\bm{u}_2\cdot\vn=0, \quad 2\left[(\DT\bm{u}_2)\vn\right]_{\vt}+\alpha\bm{u}_{2\vt}=\bm{h} \ &\text{ on $\Gamma$}.
	\end{cases}
	\end{equation*}
	given by Proposition \ref{P2}. Then $(\vu,\pi)=(\bm{u}_1,\pi_1)+(\bm{u}_2, \pi_2)$ and is the solution of the problem (\ref{S}) which also satisfies the estimate (\ref{Lp1}) and (\ref{Lp1.}).
	\hfill
\end{proof}

\begin{remark}
\label{S6s2R0}
\rm{ Note that it is also possible to deduce uniform estimate (\ref{Lp1}) in the case when $\Omega$ is axisymmetric, $\alpha$ is a constant with no strict positive lower bound $\alpha_*$ and the condition (\ref{50}) is satisfied. Indeed, we may use the $L^2$ estimate (\ref{54}) in (\ref{H1W1p}) and carry forward all consequent results.
}
\end{remark}

In the next result, we improve the dependence of the continuity constant $\gamma$ of the inf-sup condition (\ref{infsup-ineq}) on the parameters and show that it is actually independent of $\alpha$.

\begin{theorem}
\label{P3}
Let $p\in (1,\infty)$ and $ \alpha\in \Lb{t(p)}$. We have the following inf-sup condition:
\begin{equation}
\label{inf_sup}
\underset{\underset{\vu\neq 0}{\vu\in \vVsolT{p}}}{\inf} \ \ \underset{\underset{\bm{\varphi}\neq 0}{\bm{\varphi}\in \vVsolT{p'}}}{\sup} \frac{\left| 2\int\displaylimits_{\Omega}{\DT\bm{u}:\DT\bm{\varphi}}+ \int\displaylimits_{\Gamma}{\alpha \bm{u}_{\vt}\cdot \bm{\varphi}_{\vt}}\right| }{\|\bm{u}\|_{\vVsolT{p}} \ \|\bm{\varphi}\|_{\vVsolT{p'}}} \geq \gamma(\Omega,p)
\end{equation}
when either (i) $\Omega$ is not axisymmetric or (ii) $\Omega$ is axisymmetric and $\alpha \ge \alpha_* >0$.
\end{theorem}

\begin{proof}
It follows the same proof as in Proposition \ref{P2}. Indeed, let $\vu\in \vVsolT{p}$ and $\vu\ne \bm{0}$. Then by Korn inequality, $\|\vu\|_{\vW{1}{p}}\simeq \|\vu\|_{\vL{p}} + \|\DT\vu\|_{\vL{p}}$.

{\bf 1.} We first write
\begin{equation}
\label{S6s2E0}
\|\DT\vu\|_{\vL{p}} = \underset{0\neq \mathbb{G}\in \mathbb{L}^{p'}(\Omega)}{\sup}\frac{\left| \int\displaylimits_{\Omega}{\DT\vu:\mathbb{G}}\right| }{\|\mathbb{G}\|_{\mathbb{L}^{p'}(\Omega)}} = \underset{0\neq \mathbb{G}\in \mathbb{L}^{p'}_s(\Omega)}{\sup}\frac{\left| \int\displaylimits_{\Omega}{\DT\vu:\mathbb{G}}\right| }{\|\mathbb{G}\|_{\mathbb{L}^{p'}(\Omega)}}
\end{equation}
where $\mathbb{L}^{p'}_s(\Omega)$ is the space of all symmetric matrices in $\mathbb{L}^{p'}(\Omega)$. For the last equality, note that, for any matrix $\mathbb{G}$, it can be decomposed as $\mathbb{G} = \frac{1}{2}(\mathbb{G} + \mathbb{G}^T) + \frac{1}{2}(\mathbb{G}-\mathbb{G}^T)$ and $\DT\vu: (\mathbb{G}-\mathbb{G}^T) = 0$. Therefore, denoting $\mathbb{K} = \frac{1}{2}(\mathbb{G}+\mathbb{G}^T)$, we have $\mathbb{K}\in{\mathbb{L}^{p'}_s(\Omega)}$ and $\|\mathbb{K}\|_{\mathbb{L}^{p'}(\Omega)}\le 2\|\mathbb{G}\|_{\mathbb{L}^{p'}(\Omega)}$ which shows that
\begin{equation*}
\underset{0\neq \mathbb{G}\in \mathbb{L}^{p'}(\Omega)}{\sup}\frac{\left| \int\displaylimits_{\Omega}{\DT\vu:\mathbb{G}}\right| }{\|\mathbb{G}\|_{\mathbb{L}^{p'}(\Omega)}} \le \underset{0\neq \mathbb{K}\in \mathbb{L}^{p'}_s(\Omega)}{\sup}\frac{\left| \int\displaylimits_{\Omega}{\DT\vu:\mathbb{K}}\right| }{\|\mathbb{K}\|_{\mathbb{L}^{p'}(\Omega)}}.
\end{equation*}
And the reverse inequality in the above relation is obvious.

Now for any $\mathbb{G}\in\mathbb{L}^{p'}_s(\Omega)$, let $\left( \bm{\varphi}, \tilde{\pi}\right) \in\vW{1}{p'}\times L^{p'}_0(\Omega)$ be the unique solution of
\begin{equation}
\label{S6s2S0}
\begin{cases}
-\Delta \bm{\varphi} + \nabla \tilde{\pi}= \div \ \mathbb{G}, \quad
\div\;\bm{\varphi}=0 \ &\text{ in $\Omega$}\\
\bm{\varphi}\cdot\vn=0, \quad \left[(2\DT\bm{\varphi}+\mathbb{G})\vn\right]_{\vt}+\alpha\bm{\varphi}_{\vt}=\bm{0} \ &\text{ on $\Gamma$}.
\end{cases}
\end{equation}
Since we have either (i) $\Omega$ is not axisymmetric or (ii) $\Omega$ is axisymmetric and $\alpha \ge \alpha_* >0$, the solution also satisfies estimate (by Proposition \ref{P1}),
\begin{equation}
\label{S6s2}
\|\bm{\varphi}\|_{\vW{1}{p'}} \le C_p(\Omega) \|\mathbb{G}\|_{\mathbb{L}^{p'}(\Omega)}.
\end{equation}
Also taking $\vu$ as a test function in the weak formulation of (\ref{S6s2S0}), we obtain
\begin{equation}
\label{S6s2E1}
2\int\displaylimits_{\Omega}{\DT\bm{\varphi}:\DT\vu} + \int\displaylimits_{\Gamma}{\alpha \bm{\varphi_{\vt}\cdot \vu_{\vt}}} = -\int\displaylimits_{\Omega}{\mathbb{G}: \nabla\vu} = -\int\displaylimits_{\Omega}{\mathbb{G}: \DT\vu}
\end{equation}
where in the last equality, we used $\mathbb{G}$ is symmetric.
Thus, from (\ref{S6s2E0}), combining (\ref{S6s2}) and (\ref{S6s2E1}), we get
\begin{equation*}
\begin{aligned}
\|\DT\vu\|_{\vL{p}}  \le C_p(\Omega) \underset{\underset{\bm{\varphi\ne \bm{0}}}{\bm{\varphi\in\vVsolT{p'}}}}{\sup}{\frac{\left| 2\int\displaylimits_{\Omega}{\DT \vu : \DT \bm{\varphi}}+ \int\displaylimits_{\Gamma}{\alpha \vu_{ \vt}\cdot \bm{\varphi}_{\vt}}\right| }{\|\bm{\varphi}\|_{\vW{1}{p'}}}} .
\end{aligned}
\end{equation*}

{\bf 2.} Similarly as in (\ref{Lp2.}), to show
\begin{equation}
\label{S6s2E6}
\|\vu\|_{\vL{p}} \le C_p(\Omega) \underset{\underset{\bm{\varphi\ne \bm{0}}}{\bm{\varphi\in\vVsolT{p'}}}}{\sup}{\frac{\left| 2\int\displaylimits_{\Omega}{\DT \vu : \DT \bm{\varphi}}+ \int\displaylimits_{\Gamma}{\alpha \vu_{ \vt}\cdot \bm{\varphi}_{\vt}}\right| }{\|\bm{\varphi}\|_{\vW{1}{p'}}}}
\end{equation}
we write 
\begin{equation}
\label{S6s2E4.}
\|\vu\|_{\vL{p}} = \underset{0\neq \bm{w} \in\vL{p'}}{\sup}\frac{|\int\displaylimits_\Omega{ \vu\cdot\bm{w} }|}{\|\bm{w}\|_{\vL{p'}}}.
\end{equation}
So for any $\bm{w}\in\vL{p'}$, the unique solution $\left( \bm{\varphi},\tilde{\pi}\right) \in \vW{1}{p'}\times L^{p'}_0(\Omega)$ of the Stokes problem
\begin{equation}
\label{S6e0.}
\begin{cases}
-\Delta \bm{\varphi} + \nabla \tilde{\pi}= \bm{w}, \quad
\div\;\bm{\varphi}=0 \ &\text{ in $\Omega$}\\
\bm{\varphi}\cdot\vn=0, \quad 2\left[(\DT\bm{\varphi})\vn\right]_{\vt}+\alpha\bm{\varphi}_{\vt}=\bm{0} \ &\text{ on $\Gamma$}
\end{cases}
\end{equation}
satisfies (by Proposition \ref{P0})
\begin{equation}
\label{S6e1.}
\|\bm{\varphi}\|_{\vW{1}{p'}}\le C_p(\Omega) \ \|\bm{w}\|_{\vL{p'}}.
\end{equation}
Therefore taking $\vu$ as a test function in the weak formulation of (\ref{S6e0.}), we get,
\begin{equation*}
2\int\displaylimits_{\Omega}{\DT\bm{\varphi}:\DT\vu} + \int\displaylimits_{\Gamma}{\alpha \bm{\varphi}_{\vt}\cdot \vu_{\vt}} =\int\displaylimits_\Omega{ \vu\cdot\bm{w} }
\end{equation*}
which yields (\ref{S6s2E6}) from (\ref{S6s2E4.}) together with (\ref{S6e1.}).

Hence, the constant $\gamma(\Omega,p)$ is simply the inverse of $C_p(\Omega)$.
\hfill
\end{proof}

\begin{theorem}
\label{W^{2,p}estimate}
	Let $p>2$, $\alpha$ be a constant and $\bm{f}\in\vL{p}$. Then the solution $(\vu,\pi)\in\vW{2}{p}\times \W{1}{p}$ of (\ref{S}) with $\mathbb{F}=0$ and $\bm{h}=\bm{0}$ satisfies the following estimates:\\
	\textbf{(i)} if $\Omega$ is not axisymmetric, then
	\begin{equation}
	\|\vu\|_{\vW{2}{p}} +\|\pi\|_{\W{1}{p}} \le C_p(\Omega) \ \|\bm{f} \|_{\vL{p}}
	\end{equation}
	\textbf{(ii)} if $\Omega$ is axisymmetric and $\alpha\ge \alpha_* >0$, then	
	\begin{equation}
	\|\vu\|_{\vW{2}{p}} +\|\pi\|_{\W{1}{p}} \le C_p(\Omega,\alpha_*) \ \|\bm{f} \|_{\vL{p}}.
	\end{equation}
\end{theorem}

\begin{remark}
\rm{If we consider the operator of the form $\div (A(x)\nabla)\vu$ instead of $\Delta \vu$ in the first equation of the system (\ref{S}) and continue all the corresponding estimates, we may obtain the above ${\bm W}^{2,p}$-estimate for all $p\in (1, \infty)$ as done in \cite[Theorem 3.1]{AGG}.}
\end{remark}

\begin{proof}
The proof follows combining the $\bm{H}^2$-estimate and the $\bm{W}^{1,p}$-estimate proved in Theorem \ref{13} and Theorem \ref{T2} respectively. 
\hfill
\end{proof}

\section{Limiting cases}
\label{the limiting case}
\setcounter{equation}{0}
Our objective in this section is to study the limiting behaviour of the solution of \eqref{S} when the friction coefficient $\alpha$ goes to 0 and $\infty$.

\subsection{$\alpha$ tends to $0$}
\begin{theorem}
\label{goes 0}
Let $p\in (1,\infty)$, $\Omega$ be not axisymmetric and $(\vu_\alpha, \pi_\alpha)$ be the solution of \eqref{S} where
$$
\bm{f}\in\vL{r(p)}, \mathbb{F}\in\mathbb{L}^p(\Omega), \bm{h}\in \vWfracb{-\frac{1}{p}}{p}, \alpha \in \Lb{t(p)}.
$$
Then as $\alpha \rightarrow 0$ in $\Lb{t(p)}$, we have the convergence,
$$ (\vu_\alpha, \pi _\alpha) \rightarrow (\vu_0, \pi_0) \quad \text{ in } \quad \vW{1}{p}\times L^{p}_0(\Omega) $$
where $(\vu_0,\pi_0)$ is a solution of the following Stokes problem with Navier boundary condition corresponding to $\alpha = 0,$
\begin{equation}
\label{27}
\begin{aligned}
\begin{cases} 
-\Delta\vu_0 +\nabla \pi_0=\bm{f} + \div \ \mathbb{F},\quad \div\;\vu_0=0 & \text{ in } \ \Omega ,\\
\vu_0\cdot\vn=0, \quad \left[(2\DT\vu_0+\mathbb{F})\vn\right]_{\vt}=\bm{h} & \text{ on } \ \Gamma .
\end{cases}
\end{aligned}
\end{equation}
\end{theorem}

\begin{proof}
Let $\alpha \rightarrow 0$ in $\Lb{t(p)}$. That means there does not exist any $\alpha_*>0$ such that $\alpha\geq \alpha_*$ on $\Gamma$. Now from the estimates \eqref{51} and estimates \eqref{12}, it is clear that $(\vu_\alpha, \pi_\alpha)$ is bounded in $\vW{1}{p}\times L^p_0(\Omega)$ for all $p\in (1,\infty)$. Then there exists $(\vu_0, \pi_0) \in \vW{1}{p}\times L^p_0(\Omega)$ such that
$$\left( \vu_\alpha, \pi_\alpha\right) \rightharpoonup \left( \vu_0,\pi_0\right)  \ \text{ weakly \ in } \ \vW{1}{p}\times \L{p} .$$
It can be easily shown that $(\vu_0,\pi_0)$ is the unique solution of the Stokes problem with Navier boundary condition, corresponding to $\alpha = 0$,
\begin{equation}
\label{limit system 0}
\begin{aligned}
\begin{cases} 
-\Delta\vu_0 +\nabla \pi_0=\bm{f}+\div \ \mathbb{F},\quad \div\;\vu_0=0 & \text{ in } \ \Omega ,\\
\vu_0\cdot\vn=0, \quad \left[(2\DT\vu_0+\mathbb{F})\vn\right]_{\vt}=\bm{h} & \text{ on } \ \Gamma .
\end{cases}
\end{aligned}
\end{equation}
Indeed, since $(\vu_\alpha, \pi_\alpha)$ is the solution of \eqref{S}, it satisfies the weak formulation \eqref{varfor_Stokes_Nbc}. Now as in Lemma \ref{23}, $\vu_\alpha \rightharpoonup \vu_0$ in $\vW{1}{p}$ implies
$$\vu_{\alpha\vt}\rightharpoonup \vu_{0\vt} \quad \text{ in } \ \vLb{s}$$
where $s$ satisfies (\ref{S5Es}) and because $\alpha \rightarrow 0$ in $\Lb{t(p)}$,
$$\alpha \vu_{\alpha \vt} \rightharpoonup \bm{0} \quad \text{ in } \ \vLb{m^\prime} $$
with $m$ satisfies (\ref{S5Em}).
Hence in the weak formulation \eqref{varfor_Stokes_Nbc}, the boundary term in the left hand side goes to $0$. So finally, passing the limit, we deduce,
$$\forall \ \bm{\varphi} \in \vVsolT{p^\prime}, \quad 2\int\displaylimits_{\Omega}{\DT\vu_0:\DT\bm{\varphi}}=\int\displaylimits_{\Omega}{\bm{f}\cdot\bm{\varphi}}-\int\displaylimits_{\Omega}{\mathbb{F}:\nabla \bm{\varphi}}+\left\langle \bm{h},\bm{\varphi}\right\rangle_\Gamma .$$
Now by taking difference between the system \eqref{S} and the limiting system \eqref{limit system 0}, we get,
\begin{align*}
\begin{cases} 
-\Delta(\vu_\alpha - \vu_0) +\nabla (\pi_\alpha-\pi_0)=\bm{0},\quad \div\;(\vu_\alpha-\vu_0)=0 & \text{ in } \ \Omega \ ,\\
(\vu_\alpha-\vu_0)\cdot\vn=0, \quad 2\left[\DT(\vu_\alpha-\vu_0)\vn\right]_{\vt} + \alpha (\vu_{\alpha}- \vu_{0})_{\vt}=-\alpha \vu_{0\vt} & \text{ on } \ \Gamma \ .
\end{cases}
\end{align*}
Once again using the estimates of Proposition \ref{L^2estimate} and Proposition \ref{estimate} for the above system and also using H\"{o}lder inequality and trace theorem, we obtain
\begin{align*}
& \ \|\vu_\alpha - \vu_0\|_{\vW{1}{p}}+ \|\pi_\alpha - \pi_0\|_{\L{p}}\\
& \leq C(\Omega) \ \|\alpha \vu_{0 \vt}\|_{\vWfracb{-\frac{1}{p}}{p}}\\
& \leq C(\Omega) \ \|\alpha\|_{\Lb{t(p)}} \|\vu_0\| _{\vW{1}{p}} \ .
\end{align*}
Therefore, $\vu_\alpha - \vu_0$ and $\pi_\alpha - \pi_0$ both tend to zero in the same rate as $\alpha$.
\hfill
\end{proof}

\begin{remark}
\rm{We can prove also the above theorem for $\Omega$ axisymmetric and $\alpha$ constant, provided the compatibility condition (\ref{50}), with the help of the estimates \eqref{54} and Remark \ref{S6s2R0}.
Indeed, to expect the limiting system to be \eqref{27}, we must assume the compatibility condition since this is the necessary condition for the existence of a solution of the system \eqref{27}.
}
\end{remark}

\subsection{$\alpha$ tends to $\infty$}
Next we study the behaviour of $\vu_\alpha$ where $\alpha$ is a constant and grows to $\infty$.

\begin{theorem}
\label{inf}
Let $p\in (1,\infty)$ and $(\vu_\alpha, \pi_\alpha)$ be the solution of \eqref{S} where
$$
\bm{f}\in\vL{r(p)},\mathbb{F}\in\mathbb{L}^p(\Omega), \bm{h}\in \vWfracb{-\frac{1}{p}}{p} \text{ and } \alpha \text{ constant}.
$$
{\bf i)} Then as $\alpha \rightarrow \infty$, we have the convergence,
$$ (\vu_\alpha, \pi _\alpha) \rightharpoonup (\vu_\infty, \pi_\infty) \quad \text{ in } \quad \vW{1}{p}\times L^{p}_0(\Omega) $$
where $(\vu_\infty,\pi_\infty)$ is the unique solution of the Stokes problem with Dirichlet boundary condition,
\begin{equation}
\label{limit system inf}
\begin{aligned}
\begin{cases} 
-\Delta\vu_\infty +\nabla \pi_\infty=\bm{f} + \div \ \mathbb{F} &\text{ in } \ \Omega ,\\
\div\;\vu_\infty=0 &\text{ in } \ \Omega ,\\
\vu_\infty =\bm{0} &\text{ on } \ \Gamma .
\end{cases}
\end{aligned}
\end{equation}
{\bf ii)} Moreover, for $p\ge 2$, we obtain the strong convergence
$$
(\vu_{ \alpha},\pi_{\alpha})\rightarrow (\vu_\infty,\pi_\infty) \quad \text{ in } \quad \vW{1}{p}\times L^{p}_0(\Omega).
$$
\end{theorem}

\begin{proof}
As $\alpha \rightarrow \infty$, we can consider $\alpha\geq 1$.

{\bf i)} From estimates \eqref{Lp1} or \eqref{Lp1.}, we see that $(\vu_\alpha, \pi_\alpha)$ is bounded in $\vW{1}{p}\times L^p_0(\Omega)$, hence there exists $(\vu_\infty, \pi_\infty)\in \vW{1}{p}\times L^p_0(\Omega)$ such that
$$\left( \vu_\alpha,\pi_\alpha\right) \rightharpoonup \left( \vu_\infty,\pi_\infty\right)  \quad \text{ weakly \ in } \ \vW{1}{p}\times \L{p} .$$
But we can also write the system \eqref{S} as follows :
\begin{equation}
\label{modified S}
\begin{aligned}
\begin{cases}
-\Delta\vu_\alpha +\nabla \pi_\alpha=\bm{f}+\div \ \mathbb{F} & \text{ in } \ \Omega \\
\div\;\vu_\alpha = 0 & \text{ in } \ \Omega \\
\vu_\alpha = \frac{1}{\alpha}\left( \bm{h} - [(2\DT\vu_\alpha+\mathbb{F})\vn]_{\vt}\right)  & \text{ on } \ \Gamma .
\end{cases}
\end{aligned}
\end{equation}
Passing limit in the above system as $\alpha \rightarrow \infty$, we obtain that $(\vu_\infty,\pi_\infty)$ is the solution of the Stokes problem with Dirichlet boundary condition (\ref{limit system inf}).

Indeed, passing limit in the first two equations of \eqref{modified S} is easy. And for the boundary condition, since $\left( \vu_\alpha, \pi_\alpha\right) $ is bounded in $\vE{p}$, by the Green formula \eqref{green}, $2[(\DT\vu_\alpha)\vn]_{\vt}$ is bounded in $\vWfracb{-\frac{1}{p}}{p}$. Hence taking limit as $\alpha \rightarrow \infty$ in the boundary condition of (\ref{modified S}), we obtain the boundary condition of \eqref{limit system inf}.

{\bf ii)} Now taking the difference between the systems \eqref{S} and \eqref{limit system inf}, we get
\begin{equation*}
\label{22}
\begin{aligned}
\begin{cases} 
-\Delta(\vu_\alpha - \vu_\infty) +\nabla (\pi_\alpha - \pi_\infty)=\bm{0}, \quad \div\;(\vu_\alpha - \vu_\infty) = 0 & \text{ in } \ \Omega \\
(\vu_\alpha -\vu_\infty)\cdot \vn = 0, \quad 2\left[\DT (\vu_\alpha -\vu_\infty)\vn \right]_{\vt}+ \alpha(\vu_\alpha -\vu_\infty)_{\vt}=\bm{h} - 2[(\DT\vu_\infty)\vn]_{\vt} & \text{ on } \ \Gamma 
\end{cases}
\end{aligned}
\end{equation*}
Multiplying the above system by $ \vu_\alpha - \vu_\infty$ and integrating by parts yields,
$$
2\int\displaylimits_{\Omega}{|\DT(\vu_\alpha - \vu_\infty)|^2} + \alpha \int\displaylimits_{\Gamma}{|(\vu_\alpha - \vu_\infty)_{\vt}|^2} = \left\langle \bm{h} - 2[(\DT\vu_\infty)\vn]_{\vt}, \vu_\alpha - \vu_\infty \right\rangle _{\vHfracbd{1}{2}\times \vHfracb{1}{2}}.
$$
As $\vu_{ \alpha}\rightharpoonup \vu_\infty$ in $\vHfracb{1}{2}$ weakly and $\bm{h}-2 [(\DT\vu_\infty)\vn]_{\vt}\in\vHfracbd{1}{2}$, this shows the strong convergence of $\vu_\alpha$ to $\vu_\infty$ in $\vH{1}$. And the strong convergence for the pressure term follows from the following estimate
\begin{equation*}
\|\pi_\alpha - \pi_\infty\|_{\L{2}} \le \|\nabla (\pi_\alpha - \pi_\infty)\|_{\vH{-1}} \le C \|\Delta (\vu_{ \alpha}- \vu_\infty)\|_{\vH{-1}}.
\end{equation*}

Next, for $p>2$, as $\DT(\vu_{ \alpha}-\vu_\infty)\rightarrow 0$ in $\vL{2}$, $\DT (\vu_{ \alpha}-\vu_\infty)\rightarrow 0$ for almost every $\bm{x}\in\Omega$ up to a subsequence. Also using the uniform estimate in Theorem \ref{T2} for the above system satisfied by $\vu_{ \alpha}-\vu_\infty$, we can write
\begin{equation*}
\|\vu_{ \alpha}-\vu_\infty\|_{\vW{1}{p}} \le C_p(\Omega) \| \bm{h} - 2[(\DT\vu_\infty)\vn]_{\vt}\|_{\vWfracb{-\frac{1}{p}}{p}}
\end{equation*}
which shows that $\|\DT(\vu_{ \alpha}-\vu_\infty)\|_{\mathbb{L}^p(\Omega)}\le C$. Therefore, dominated convergence theorem yields $\|\DT(\vu_{ \alpha}-\vu_\infty)\|_{\mathbb{L}^p(\Omega)}\rightarrow 0$. This completes the proof.

\hfill
\end{proof}

\subsection{$\alpha$ less regular}
\begin{theorem}
\label{less regular}
Let
$$
\bm{f}\in \vL{\frac{6}{5}}, \mathbb{F}\in\mathbb{L}^2(\Omega), \bm{h}\in \vHfracbd{1}{2} \text{ and } \alpha \in \Lb{\frac{4}{3}}.
$$
Then the Stokes problem \eqref{S} has a solution $\left( \vu,\pi \right)$ in $\vH{1}\times \L{2}$.
\end{theorem}

\begin{proof}
Without loss of generality, we assume $\bm{h} = \bm{0}$.
	
{\bf i)} First let us consider $\Omega$ is not axisymmetric. Let $\alpha_k \in \D{\Gamma}$ such that $\alpha_k \rightarrow \alpha$ in $\Lb{\frac{4}{3}}$. Now if $\left( \vu_k, \pi_k\right)\in \vH{1}\times L^{2}_0(\Omega)$ is the solution of the problem \eqref{S} corresponding to $\alpha_k$, from the estimates \eqref{51} satisfied by $\left( \vu_k, \pi_k\right)$, we can see, there exists $\left( \vu,\pi\right) \in \vH{1}\times L^2_0(\Omega)$ such that
\begin{equation*}
\left( \vu_k,\pi_k\right) \rightharpoonup \left( \vu,\pi\right) \quad  \text{ in } \quad \vH{1}\times \L{2} .
\end{equation*}
Also it is easy to show that
\begin{equation*}
-\Delta\vu +\nabla \pi=\bm{f} \text{ in }\Omega, \quad \div \ \vu = 0 \text{ in } \Omega, \quad \vu \cdot \vn = 0 \text{ on } \Gamma.
\end{equation*}
Now from the Green formula \eqref{green}, we obtain
\begin{equation*}
2\left[(\DT\vu_k)\vn\right]_{\vt} \rightharpoonup 2\left[(\DT\vu)\vn\right]_{\vt} \quad \text{ in } \quad \vHfracbd{1}{2} .
\end{equation*}
Moreover, as $\alpha_k \rightarrow \alpha$ in $\Lb{\frac{4}{3}}$ and $\vu_{k \vt} \rightharpoonup \vu_{\vt}$ in $\Lb{4}$, it gives $$\alpha_k \vu_{k \vt} \rightharpoonup \alpha \vu_{\vt} \quad \text{ in } \quad \vLb{1} .$$
Therefore since $\vu_k$ satisfies the boundary condition $$\left[(2\DT\vu_k+\mathbb{F})\vn\right]_{\vt}+\alpha_k\;\vu_{k \vt}=\bm{0} \ \text{ on } \Gamma ,$$
passing the limit, it yields
$$\left[(2\DT\vu+\mathbb{F})\vn\right]_{\vt}+\alpha\;\vu_{\vt}=\bm{0} \ \text{ on } \Gamma .$$
Hence, $(\vu,\pi)$ becomes the solution of the Stokes problem \eqref{S}.

{\bf ii)} Note that, when $\Omega$ is axisymmetric and $\alpha \geq \alpha_*>0$, we can find a sequence $\alpha_k \in \D{\Gamma}$ such that $\alpha_k \geq \alpha_*$ and $\alpha_k \rightarrow \alpha$ in $\Lb{\frac{4}{3}}$. So we can make use of the estimate \eqref{52} and obtain the same result.
\hfill
\end{proof}

\begin{remark}
\rm{For $\alpha \in \Lb{\frac{4}{3}}$ and $\bm{h} = \bm{0}$, the solution $\vu\in\vH{1}$ satisfies the extra property: $\left[(\DT\vu)\vn\right]_{\vt}\in \vLb{1}$.}
\end{remark}

\begin{remark}
\rm{Let $\bm{f} = \bm{0}, \mathbb{F}=0$ and $\bm{h}\in \vLb{\frac{4}{3}}$ with $\bm{h}\cdot \vn = 0$ on $\Gamma$.\\

\noindent i) If $\alpha \in \Lb{2}$, we have from Theorem \ref{thm_weak_sol_Stokes_Nbc}, $\vu\in\vH{1}$ which means $\alpha \vu_{\vt}\in\vLb{\frac{4}{3}}\hookrightarrow \vHfracbd{1}{2}$. So $[(\DT\vu)\vn]_{\vt} \in \vLb{\frac{4}{3}}$ which in turn implies $\vu\in \vW{1+\frac{3}{4}}{\frac{4}{3}}\hookrightarrow \vH{1}$.\\

\noindent ii) If $\alpha \in \Lb{\frac{4}{3}}$, we have from Theorem \ref{less regular}, $\vu\in\vH{1}$ which gives $\alpha \vu_{\vt}\in \vL{1}$. But from here, we cannot improve the regularity any more.
} 
\end{remark}

\section{Navier-Stokes equations}
\label{Navier-Stokes equation}
\setcounter{equation}{0}

Finally we consider the non-linear problem and study the existence of generalized and strong solutions for the Navier-Stokes equation:
\begin{equation}
\begin{cases}
 -\Delta\vu + \vu\cdot \nabla \vu+ \nabla \pi=\bm{f}+\div \ \mathbb{F},\quad \div\;\vu=0 & \text{in $\Omega$,}\\
 \vu\cdot\vn=0, \qquad \left[(2\DT\vu+\mathbb{F})\vn\right]_{\vt}+\alpha\vu_{\vt}=\bm{h}& \text{on $\Gamma$.}
\end{cases}
\end{equation}

\subsection{Existence and regularity}
\begin{theorem}
\label{35}
Let $p\in (1,\infty)$ and
$$
\bm{f}\in \vL{r(p)}, \mathbb{F}\in \mathbb{L}^p(\Omega),\bm{h}\in \vWfracb{-\frac{1}{p}}{p} \text{ and } \alpha\in \Lb{t(p)}
$$
where $r(p)$ and $t(p)$ are defined by \eqref{def_exponent_rp} and \eqref{def_exponent_tp_alpha} respectively. Then the following two problems are equivalent:\\
(i) Find $(\vu, \pi)\in \vW{1}{p}\times \L{p}$ satisfying \eqref{NS} in the sense of distributions.\\
(ii) Find $\vu\in\vVsolT{p}$ such that for all $\bm{\varphi}\in \vVsolT{p^\prime}$,
\begin{equation}
\label{36}
2\int\displaylimits_{\Omega}{\DT\vu:\DT\bm{\varphi}}+b(\vu,\vu,\bm{\varphi})+\int\displaylimits_{\Gamma}{\alpha\vu_{\vt}\cdot\bm{\varphi}_{\vt}}=\int\displaylimits_{\Omega}{\bm{f}\cdot
	\bm{\varphi}}-\int\displaylimits_{\Omega}{\mathbb{F}:\nabla \bm{\varphi}}+\left\langle \bm{h},\bm{\varphi}\right\rangle_\Gamma .
\end{equation}
where $b(\vu, \bm{v}, \bm{w}) = \int\displaylimits_{\Omega}{\left( \vu \cdot \nabla\right)  \bm{v} \cdot \bm{w}}$.
\end{theorem}
\noindent
The proof is standard and very similar to that of Proposition \ref{34}, hence we omit it. To facilitate the work, we give some properties of the operator $b$ but skip the proof (cf. \cite[Lemma 7.2]{AR}).

\begin{lemma}
The trilinear form $b$ is defined and continuous on $\vVsolT{2} \times \vVsolT{2}\times \vVsolT{2}$. Also, we have
\begin{equation}
\label{S8E1}
b(\vu,\bm{v},\bm{v}) =\bm{0}
\end{equation}
and
\begin{equation*}
b(\vu,\bm{v},\bm{w}) = -b(\vu,\bm{w},\bm{v}) \quad \forall \ \vu, \bm{v}, \bm{w} \in\vVsolT{2} . 
\end{equation*}
Moreover
$$b(\vu,\vu,\bm{\beta}) = 0 \quad \text{ and } \quad b(\bm{\beta},\bm{\beta},\vu) = 0 .$$
\end{lemma}

Now we can prove the existence of the generalized solution of the Navier-Stokes problem \eqref{NS}. First we study the Hilbert case.

\begin{theorem}
\label{39}
Let
$$
\bm{f}\in \vL{\frac{6}{5}}, \mathbb{F}\in\mathbb{L}^2(\Omega), \bm{h}\in \vHfracbd{1}{2} \text{ and } \alpha\in \Lb{2}.
$$
Then the problem \eqref{NS} has a solution $(\vu, \pi)\in \vH{1}\times L^2_0(\Omega)$ satisfying the following estimates:\\
{\bf a)} if $\Omega$ is not axisymmetric, then
\begin{equation}
\label{37}
\|\vu\|_{\vH{1}}+\|\pi\|_{\L{2}}\leq{C(\Omega)}\left(\|\bm{f}\|_{\vL{\frac{6}{5}}}+\|\mathbb{F}\|_{\mathbb{L}^2(\Omega)}+\|\bm{h}\|_{\vHfracbd{1}{2}}\right) .
\end{equation}
\medskip
{\bf b)} if $\Omega$ is axisymmetric and\\
\indent $\left(\textit{i} \right)$ $\alpha \geq \alpha_* >0$ on $\Gamma$, then
\begin{equation}
\label{32}
\|\vu\|_{\vH{1}}+\|\pi\|_{L^{2}(\Omega)}\leq \frac{C(\Omega)}{\min\{2,\alpha_*\}}\left(\|\bm{f}\|_{\vL{\frac{6}{5}}}+\|\mathbb{F}\|_{\mathbb{L}^2(\Omega)}+\|\bm{h}\|_{\vHfracbd{1}{2}}\right) .
\end{equation}
\indent $\left(\textit{ii} \right)$ $\bm
{f}, \bm{h}$ satisfy the condition:
\begin{equation*}
\int\displaylimits_{\Omega}{ \bm{f} \cdot \bm{\beta}} - \int\displaylimits_{\Omega}{\mathbb{F}: \nabla\bm{\beta}} + \left\langle \bm{h}, \bm{\beta}\right\rangle_\Gamma = 0
\end{equation*}
\qquad then the solution $\vu$ satisfies $\int\displaylimits_{\Gamma}{\alpha \vu \cdot \bm{\beta}} = 0$ and
\begin{equation}
\label{41}
\|\DT\vu\|_{\mathbb{L}^2(\Omega)}^2+\int\displaylimits_{\Gamma}{\alpha |\vu_{\vt}|^2}+ \|\pi\|^2_{\L{2}}\leq C(\Omega)\left(\|\bm{f}\|_{\vL{\frac{6}{5}}}+\|\mathbb{F}\|_{\mathbb{L}^2(\Omega)}+\|\bm{h}\|_{\vHfracbd{1}{2}}\right)^2 .
\end{equation}
\qquad In particular, if $\alpha$ is a constant, then $\int\displaylimits_{\Gamma}{\vu \cdot \bm{\beta}} = 0$ and
\begin{equation}
\label{42}
\|\vu\|_{\vH{1}}+\|\pi\|_{\L{2}}\leq{C(\Omega)}\left(\|\bm{f}\|_{\vL{\frac{6}{5}}}+\|\mathbb{F}\|_{\mathbb{L}^2(\Omega)}+\|\bm{h}\|_{\vHfracbd{1}{2}}\right) .
\end{equation}
\end{theorem}

\begin{remark}
	\rm{Note that in the case of $\vu\cdot \vn \ne 0$ on $\Gamma$ when $\Omega$ has multiply connected boundary, the existence of solution of Navier-Stokes equation with Dirichlet boundary condition is not yet clear in full generality, e.g. see \cite{Korobkov}. So we do not consider that case here for the Navier boundary condition also. }
\end{remark}

\begin{proof}
{\bf 1. Existence:} The existence of solution of \eqref{36} can be shown using standard arguments i.e. by Galerkin method, we construct an approximate solution and then pass to the limit. Nonetheless, we state it briefly for completeness.

For each fixed integer $m\ge 1$, define an approximate solution $\vu_m$ of (\ref{36}) by
\begin{equation}
\label{S8E0}
\begin{aligned}
&\vu_m = \sum_{i=1}^{m}\xi _{i,m}\bm{v}_i, \,\,\,\,\xi _{i,m}\in\R\\
&2\int\displaylimits_{\Omega}{\DT\vu_m:\DT\bm{v}_k}+b(\vu_m,\vu_m,\bm{v}_k)+\int\displaylimits_{\Gamma}{\alpha\vu_{\vt m}\cdot\bm{v}_{\vt k}}=\int\displaylimits_{\Omega}{\bm{f}\cdot
	\bm{v}_k}-\int\displaylimits_{\Omega}{\mathbb{F}:\nabla \bm{v}_k}+\left\langle \bm{h},\bm{v}_k\right\rangle_\Gamma,\\
& \hspace{11cm} \text{ for } k = 1,..., m
\end{aligned}
\end{equation}
and $V_m := \left\langle \bm{v}_1, ... , \bm{v}_m\right\rangle $ is the space spanned by the vectors $\bm{v}_1, ... , \bm{v}_m$ and $ \{\bm{v}_i\}_i$ is the Hilbert basis of $\vVsolT{2}$. Note that $V_m$ is equipped with the scalar product $ (,)$ induced by $\vVsolT{2}$. Let us also define the mapping $P_m : V_m \rightarrow V_m$ as for all $ \bm{w},\bm{v}\in V_m$,
\begin{equation*}
\left( P_m(\bm{w}), \bm{v} \right) =  2\int\displaylimits_{\Omega}{\DT\bm{w}:\DT\bm{v}}+b(\bm{w},\bm{w},\bm{v})+\int\displaylimits_{\Gamma}{\alpha\bm{w}_{\vt }\cdot\bm{v}_{\vt}}-\int\displaylimits_{\Omega}{\bm{f}\cdot
	\bm{v}}+\int\displaylimits_{\Omega}{\mathbb{F}:\nabla \bm{v}}-\left\langle \bm{h},\bm{v}\right\rangle_\Gamma.
\end{equation*}
The continuity of the mapping is obvious. Also, using (\ref{S8E1}), we get
\begin{equation*}
\begin{aligned}
\left( P_m(\bm{w}), \bm{w} \right) &= 2\|\DT\bm{w}\|^2_{\mathbb{L}^2(\Omega)} + \int\displaylimits_{\Gamma}{\alpha |\bm{w}_{\vt}|^2} - \int\displaylimits_{\Omega}{\bm{f}\cdot
	\bm{w}}+\int\displaylimits_{\Omega}{\mathbb{F}:\nabla \bm{w}}-\left\langle \bm{h},\bm{w}\right\rangle_\Gamma\\
&\ge C(\alpha) \|\bm{w}\|_{\vH{1}}\left\lbrace  \|\bm{w}\|_{\vH{1}} - C(\Omega)\left( \|\bm{f}\|_{\vL{\frac{6}{5}}} + \|\mathbb{F}\|_{\mathbb{L}^2(\Omega)} + \|\bm{h}\|_{\vHfracbd{1}{2}}\right)\right\rbrace . 
\end{aligned}
\end{equation*}
Hence, $ \left( P_m(\bm{w}), \bm{w} \right) > 0$ for all $ \|\bm{w}\|_{V_m} = k$ where $k> C(\Omega)\left( \|\bm{f}\|_{\vL{\frac{6}{5}}} +\|\mathbb{F}\|_{\mathbb{L}^2(\Omega)}+ \|\bm{h}\|_{\vHfracbd{1}{2}}\right)$. Therefore, the hypothesis of Brower's theorem is satisfied and there exists a solution $\vu_m$ of (\ref{S8E0}).

Next as $\bm{u}_m$ is a solution of (\ref{S8E0}), we have
\begin{equation*}
2\|\DT\vu_m\|^2_{\mathbb{L}^2(\Omega)}+\int\displaylimits_{\Gamma}{\alpha|\vu_{\vt m}|^2}=\int\displaylimits_{\Omega}{\bm{f}\cdot
	\bm{u}_m}-\int\displaylimits_{\Omega}{\mathbb{F}:\nabla \bm{u}_m}+\left\langle \bm{h},\bm{u}_m\right\rangle_\Gamma
\end{equation*} 
which yields the a priori estimate
\begin{equation*}
\|\vu_m\|_{\vH{1}} \le C(\alpha ) \left( \|\bm{f}\|_{\vL{\frac{6}{5}}}+\|\mathbb{F}\|_{\mathbb{L}^2(\Omega)}+ \bm{h}\|_{\vHfracbd{1}{2}}\right) .
\end{equation*}
Since the sequence $\vu_m$ remains bounded in $\vVsolT{2}$, there exists some $\vu\in\vVsolT{2}$ and a subsequence which we still call $\vu_m$ such that
$$
\vu_m \rightharpoonup \vu \text{ in } \vVsolT{2}.
$$
Now due to the compact embedding of $\H{1}$ into $\L{2}$, we can pass to the limit in (\ref{S8E0}) and obtain, for any $\bm{v}\in \vVsolT{2}$,
\begin{equation*} 2\int\displaylimits_{\Omega}{\DT\vu:\DT\bm{v}}+b(\vu,\vu,\bm{v})+\int\displaylimits_{\Gamma}{\alpha\vu_{\vt}\cdot\bm{v}_{\vt}}=\int\displaylimits_{\Omega}{\bm{f}\cdot
	\bm{v}} - \int\displaylimits_{\Omega}{\mathbb{F}:\nabla \bm{v}}+\left\langle \bm{h},\bm{v}\right\rangle_\Gamma
\end{equation*}
and thus $\vu$ is a solution of (\ref{36}).

{\bf 2. Estimates:} The estimates can be shown as for the Stokes problem in Theorem \ref{L^2estimate}.
\hfill
\end{proof}

\begin{proposition}
The solution of the problem (\ref{NS}), given by Theorem \ref{39} is unique provided
\begin{equation}
\label{38}
\|\bm{f}\|_{\vL{\frac{6}{5}}}+ \|\mathbb{F}\|_{\mathbb{L}^2(\Omega)}+\|\bm{h}\|_{\vHfracbd{1}{2}} < \frac{1}{C(\Omega)} 
\end{equation}
where the constant $C(\Omega)$ is the continuity constant of the linear form $b$.
\end{proposition}

\begin{remark}
\rm{Interestingly, in the case of $\alpha \equiv 0$, there is no uniqueness of the solution of the system \eqref{NS} even for small data. But in our case, when $\alpha\neq 0$ on some $\Gamma_0 \subseteq \Gamma$ with $|\Gamma_0| >0$, there is indeed uniqueness of the solution under the assumption of small data as in the case of Dirichlet boundary condition. The reason of this behaviour is the presence of a non-trivial kernel of the Stokes operator for $\alpha \equiv 0$.}
\end{remark}

\begin{proof}
Taking $\bm{\varphi} = \vu$ in (\ref{36}) and using (\ref{S8E1}), we obtain that any solution of (\ref{36}) satisfies the estimate
\begin{equation}
\label{S8E3}
\|\vu\|_{\vH{1}} \le C(\alpha) \left( \|\bm{f}\|_{\vL{\frac{6}{5}}}+ \|\mathbb{F}\|_{\mathbb{L}^2(\Omega)}+ \bm{h}\|_{\vHfracbd{1}{2}}\right).
\end{equation}
Now if $\vu_1$ and $\vu_2$ are two different solutions of (\ref{36}), let $\vu = \vu_1 - \vu_2$ and subtracting the equations (\ref{36}) corresponding to $\vu_1$ and $\vu_2$, we get
\begin{equation}
\label{S8E2}
\forall \bm{\varphi}\in \vVsolT{2}, \quad 2\int\displaylimits_{\Omega}{\DT\vu:\DT\bm{\varphi}}+b(\vu_1,\vu,\bm{\varphi})+ b(\vu,\vu_2,\bm{\varphi})+\int\displaylimits_{\Gamma}{\alpha\vu_{\vt}\cdot\bm{\varphi}_{\vt}}= 0.
\end{equation}
Taking $\bm{\varphi} = \vu$ in (\ref{S8E2}) and once again using (\ref{S8E1}) implies
\begin{equation*}
2\|\DT\vu\|^2_{\mathbb{L}^2(\Omega)} + \int\displaylimits_{\Gamma}{\alpha |\vu_{ \vt}|^2}= -b(\vu,\vu_2,\vu)
\end{equation*}
which yields, using the continuity of $b$ and the estimate (\ref{S8E3}) for $\vu_2$,
\begin{equation*}
\|\vu\|^2_{\vH{1}} \le \frac{C(\Omega)}{C(\alpha)} \|\vu\|^2_{\vH{1}} \|\vu_2\|_{\vH{1}} \le C(\Omega) \|\vu\|^2_{\vH{1}} \left( \|\bm{f}\|_{\vL{\frac{6}{5}}}+ \|\mathbb{F}\|_{\mathbb{L}^2(\Omega)}+ \bm{h}\|_{\vHfracbd{1}{2}}\right).
\end{equation*}
Thus considering the condition (\ref{38}), the above inequality implies $\|\vu\|_{\vH{1}} = 0$ that is $\vu_1 = \vu_2$.
\hfill
\end{proof}

Next we prove the existence of solution of the system \eqref{NS} in $\vW{1}{p}$ using the Hilbert case and the Stokes regularity result.

\begin{corollary}
\label{reg_NS}
Let $p>\frac{3}{2}$ and $\bm{f},\mathbb{F}, \bm{h}$ and $\alpha$ satisfy the assumptions as in Theorem \ref{35}.\\
{\bf i)} There exists a solution $(\vu,\pi)\in \vW{1}{p}\times L^p_0(\Omega)$  of (\ref{NS}).\\
{\bf ii)} Moreover, for any $p\in( 1,\infty)$, if $ \mathbb{F}=0$ and
$$
\bm{f}\in\vL{p}, \bm{h}\in\vWfracb{1-\frac{1}{p}}{p} \text{ and } \alpha\in\Wfracb{1-\frac{1}{q}}{q}
$$
\quad with $q>\frac{3}{2}$ if $p\le \frac{3}{2}$ and $q=p$ otherwise, then $(\vu,\pi)\in\vW{2}{p}\times\W{1}{p}$.
\end{corollary}

\begin{proof}
{\bf i)} First let us consider $p>2$. Then we have existence of a solution $ (\vu,\pi)\in \vH{1}\times L^2_0(\Omega)$ from Theorem \ref{39} using the Hilbert case.
Since $\vu \in \vH{1}$, the non-linear term $(\vu \cdot \nabla) \vu \in \vL{\frac{3}{2}}\hookrightarrow \vL{r(p)}$ if $p\leq 3$. Hence, using the regularity result for Stokes problem in Corollary \ref{thm_W1p_regularity_Stokes_Nbc}, we obtain that $(\vu,\pi)\in \vW{1}{p}\times \L{p}$.
For $p>3$, repeating the same argument with $\vu\in \vW{1}{3}$, we deduce the required regularity.

And to obtain existence result for $p\in (\frac{3}{2},2)$, we follow the exact same construction as in \cite[Theorem 1.1]{serre}. Note that we replaced the space $\vW{-1}{p}$ for the given data in \cite{serre} by $\vL{r(p)}$. For example, we used the following lemma instead of \cite[Lemma 1.2]{serre}:

If there exists $ (\bm{v},\tilde{\pi}) \in \vW{1}{p}\times \L{p}$ such that
\begin{equation*}
\begin{cases}
-\Delta\bm{v} + \bm{v}\cdot \nabla \bm{v}+ \nabla \tilde{\pi}-\bm{f}&\in \vL{r(p)} \\
\div\;\bm{v}=0 &\text{ in } \Omega\\
\bm{v}\cdot\vn=0 &\text{ on } \Gamma\\ 2\left[(\DT\bm{v})\vn\right]_{\vt}+\alpha\;\bm{v}_{\vt}=\bm{0} &\text{ on } \Gamma
\end{cases}
\end{equation*}
for $p\le q\le 2$, then there exists $ (\bm{w},\bar{\pi})\in \vW{1}{p}\times \L{p}$ such that
\begin{equation*}
\begin{cases}
-\Delta\bm{w} + \bm{w}\cdot \nabla \bm{w}+ \nabla \bar{\pi}-\bm{f}&\in \vL{r(s)} \\
\div\;\bm{w}=0 &\text{ in } \Omega\\
\bm{w}\cdot\vn=0 &\text{ on } \Gamma\\ 2\left[(\DT\bm{w})\vn\right]_{\vt}+\alpha\;\bm{w}_{\vt}=\bm{0} &\text{ on } \Gamma
\end{cases}
\end{equation*}
where $\frac{1}{s} = \frac{1}{q} + \frac{1}{p} - \frac{2}{3}$ (thus $s>q$).\\
The rest of the proof follows the same argument as in \cite{serre} without any further changes.

{\bf ii)} Next to prove the strong regularity result, consider the more regular data. For $p\in (1,\frac{3}{2}]$, since the Sobolev exponent $p^*\in (\frac{3}{2},3]$ and thus $r(p^*) = p$, we have
$
\bm{f}\in\vL{r(p^*)}, \bm{h}\in \vWfracb{-\frac{1}{p^*}}{p^*}
$
and hence using the above regularity result for weak solution of \eqref{NS}, we obtain
 $\left( \vu,\pi\right) \in \vW{1}{p^*}\times \L{p^*}$. Now for $p\in (1,\frac{3}{2})$, $(\vu \cdot \nabla) \vu \in \L{s}$ with
$$
\frac{1}{s} = \frac{2}{p}-1
$$
which implies $s>p$ and thus using Theorem \ref{thm_W2p_regularity_Stokes_Nbc} again, we obtain $\left( \vu,\pi\right)  \in \vW{2}{p}\times \W{1}{p}$. For $p=\frac{3}{2}$, since $\vW{1}{3}\hookrightarrow \vL{m}$ for any $ m\in (1,\infty)$, we have $(\vu \cdot \nabla) \vu \in \L{s}$ with $\frac{1}{s} = \frac{1}{3} + \frac{1}{m}$, so choosing $m>3$ gives $ s> \frac{3}{2}$ and thus $\left( \vu,\pi\right)  \in \vW{2}{\frac{3}{2}}\times \W{1}{\frac{3}{2}}$.

Now for $p> \frac{3}{2}$, having $\vu\in\vW{2}{\frac{3}{2}}$ gives $\sum_i u_i \partial _i u \in \vL{3-\epsilon}$ which yields $\vu\in \vW{2}{3-\epsilon}$. Further repeating the argument, we get $\vu\in\vW{2}{p}$.
\hfill 
\end{proof}

Finally we discuss the limiting behaviour of the Navier-Stokes system \eqref{NS} as $\alpha$ goes to $0$ or $\infty$.

\subsection{Limiting cases}
\begin{theorem}
\label{NS goes 0}
Let $p\geq2$, $\Omega$ be not axisymmetric and $(\vu_\alpha, \pi_\alpha)$ be a solution of \eqref{NS} where
$$\bm{f}\in\vL{r(p)}, \mathbb{F}\in\mathbb{L}^p(\Omega), \bm{h}\in \vWfracb{-\frac{1}{p}}{p} \text{ and } \alpha\in \Lb{t(p)}.$$
Then as $\|\alpha\|_{\Lb{t(p)}} \rightarrow 0$, we have the convergence,
$$ (\vu_\alpha, \pi _\alpha) \rightarrow (\vu_0, \pi_0) \quad \text{ in } \quad \vW{1}{p}\times L^{p}_0(\Omega) $$
where $(\vu_0,\pi_0)$ is a solution of the following Navier-Stokes problem
\begin{equation}
\label{45}
\begin{aligned}
\begin{cases} 
	-\Delta\vu_0 + \vu_0 \cdot \nabla \vu_0 + \nabla \pi_0=\bm{f}+\div \ \mathbb{F},\quad \div\;\vu_0=0 & \text{ in } \ \Omega ,\\
	\vu_0\cdot\vn=0, \quad \left[(2\DT\vu_0+\mathbb{F})\vn\right]_{\vt}=\bm{h} & \text{ on } \ \Gamma .
\end{cases}
\end{aligned}
\end{equation}
\end{theorem}

\begin{proof}
{\bf i)} We assume for ease of calculation, $ \mathbb{F}=0$ and $\bm{h} = \bm{0}$. As $\alpha \rightarrow 0$ in $\Lb{t(p)}$, there does not exist any $\alpha_*>0$ such that $\alpha\geq \alpha_*$ on $\Gamma_0\subseteq \Gamma$. Therefore, $(\vu_\alpha, \pi_\alpha)$ satisfies the estimates \eqref{37}. For $2<p\leq 3$, using the Stokes estimate \eqref{12}, we obtain
\begin{equation*}
\begin{aligned}
\|\vu_\alpha\|_{\vW{1}{p}} + \|\pi_\alpha\|_{\L{p}}& \leq C(\Omega) \left( \|\bm{f}\|_{\vL{r(p)}} + \|\vu_\alpha\cdot \nabla\vu_\alpha\|_{\vL{r(p)}}\right)\\
& \leq C(\Omega) \left( \|\bm{f}\|_{\vL{r(p)}} + \|\vu_\alpha\|^2_{\vH{1}}\right)\\
& \leq C(\Omega) \left( 1+ \|\bm{f}\|_{\vL{r(p)}}\right) \|\bm{f}\|_{\vL{r(p)}} 
\end{aligned}
\end{equation*}
and for $p>3$,
\begin{equation*}
\begin{aligned}
\|\vu_\alpha\|_{\vW{1}{p}} + \|\pi_\alpha\|_{\L{p}}
& \leq C(\Omega) \left( \|\bm{f}\|_{\vL{r(p)}} + \|\vu_\alpha\cdot \nabla\vu_\alpha\|_{\vL{r(p)}}\right) \\
& \leq C(\Omega) \left( \|\bm{f}\|_{\vL{r(p)}} + \|\vu_\alpha\|^2_{\vW{1}{3}}\right)\\
& \leq C(\Omega) \left[ 1+ \left( 1+ \|\bm{f}\|_{\vL{r(p)}}\right)^2 \|\bm{f}\|_{\vL{r(p)}} \right] \|\bm{f}\|_{\vL{r(p)}} .
\end{aligned}
\end{equation*}
Then $(\vu_\alpha, \pi_\alpha)$ is bounded in $\vW{1}{p}\times \L{p}$ with respect to $\alpha$. So there exists $(\vu_0, \pi_0) \in \vW{1}{p}\times \L{p}$ such that
$$(\vu_\alpha,\pi_\alpha)\rightharpoonup( \vu_0,\pi_0) \ \text{ weakly \ in } \ \vW{1}{p}\times \L{p} .$$
Now, as in Theorem \ref{goes 0}, passing the limit as $\alpha \rightarrow 0$ in $\Lb{t(p)}$ in the variational formulation satisfied by $(\vu_\alpha,\pi_\alpha)$, we get that $\vu_0$ satisfies the equation
$$
2\int\displaylimits_{\Omega}{\DT\vu_0 : \DT\bm{\varphi}} + b(\vu_0,\vu_0,\bm{\varphi}) = \int\displaylimits_{\Omega}{\bm{f} \cdot \bm{\varphi}}\qquad \forall \ \bm{\varphi} \in \vVsolT{p^\prime} .
$$
Indeed, $\vu_\alpha \rightharpoonup \vu_0 \text{ weakly in } \vW{1}{p}$
implies
$\vu_\alpha \rightarrow \vu_0 \text{ in } \vL{s}$
where
\begin{equation*}
\label{s}
s\in
\begin{cases}(1,p^*) & \text{if\quad}p<3\\
(1,\infty) & \text{if\quad}p =3\\
(1,\infty] & \text{if\quad}p<3 .
\end{cases}
\end{equation*}
Also, $\nabla \vu_\alpha \rightharpoonup \nabla\vu_0 \text{ weakly in } \vL{p}$. Therefore, $\vu_\alpha \cdot \nabla \vu_\alpha \rightharpoonup \vu_0\cdot \nabla \vu_0 \text{ weakly in } \vL{q}$ where
$$\frac{1}{q} = \frac{1}{p}+ \frac{1}{s} $$
and note that, $\bm{\varphi} \in \vW{1}{p^\prime} \hookrightarrow \vL{q^\prime}$. Hence, $b(\vu_\alpha,\vu_\alpha,\bm{\varphi}) \rightarrow b(\vu_0,\vu_0,\bm{\varphi})$ as $\alpha \rightarrow 0$ in $\Lb{t(p)}$.
Therefore, $(\vu_0,\pi_0)$ is a solution of the problem (\ref{45}).

{\bf ii)} Next we show that the convergence that $(\vu_\alpha,\pi_\alpha) \rightharpoonup (\vu_0,\pi_0) \text{ weakly in } \vW{1}{p}\times \L{p}$ occurs in fact in strong sense. Taking the difference between the systems \eqref{NS} and \eqref{45}, we get,
\begin{align*}
\begin{cases} 
-\Delta(\vu_\alpha - \vu_0) +\nabla (\pi_\alpha-\pi_0)=\vu_0 \cdot \nabla \vu_0 - \vu_\alpha \cdot \nabla \vu_\alpha & \text{ in } \Omega\\
\div\;(\vu_\alpha-\vu_0)=0 & \text{ in } \Omega \\
(\vu_\alpha-\vu_0)\cdot\vn=0, \quad 2\left[\DT(\vu_\alpha-\vu_0)\vn\right]_{\vt} + \alpha \vu_{\alpha \vt}=\bm{0} & \text{ on } \Gamma 
\end{cases}
\end{align*}
Note that $\vu_0 \cdot \nabla \vu_0 - \vu_\alpha \cdot \nabla \vu_\alpha = \div (\vu_{ \alpha}\otimes \vu_{ \alpha} - \vu_0\otimes \vu_0) $. Thus using the Stokes estimate (\ref{12}) for the above system gives
\begin{equation*}
\begin{aligned}
&\quad \|\vu_{ \alpha} - \vu_0\|_{\vW{1}{p}}+ \|\pi_\alpha - \pi_0\|_{\L{p}}\\
&\le C\left( \|\vu_{ \alpha}\otimes \vu_{ \alpha} - \vu_0\otimes \vu_0\|_{\vL{p}} + \|\alpha \vu_{0 \vt}\|_{\vWfracb{-\frac{1}{p}}{p}}\right) \\
 &= C \left( \|(\vu_{ \alpha} - \vu_0)\otimes \vu_{ \alpha}+ \vu_0\otimes (\vu_{ \alpha}-\vu_0)\|_{\vL{p}}+ \|\alpha \vu_{0 \vt}\|_{\vWfracb{-\frac{1}{p}}{p}}\right)\\
& \le C \left[ \|\vu_{ \alpha} - \vu_0\|_{\vL{s}} \left( \|\vu_{ \alpha}\|_{\vW{1}{p}} + \|\vu_0\|_{\vW{1}{p}}\right) + \|\alpha\|_{\Lb{t(p)}} \|\vu_0\|_{\vW{1}{p}}\right] 
\end{aligned}
\end{equation*} 
where $s$ is defined as in (\ref{s}). Now since $\vu_{ \alpha}$ is bounded in $\vW{1}{p}$ and by compactness, $\vu_\alpha\rightarrow \vu_0$ in $\vL{s} $, we obtain the strong convergence of $\vu_{ \alpha}$ to $\vu_0$ in $\vW{1}{p}$ as $\alpha\rightarrow 0$.
\hfill	
\end{proof}

\begin{remark}
	\rm{As in the Stokes case, we can prove also the above theorem for $\Omega$ axisymmetric and $\alpha$ constant, provided the compatibility condition (\ref{50}),	with the help of the estimates \eqref{41} and Remark \ref{S6s2R0}. Indeed, to expect the limiting system to be \eqref{45}, we must assume the above compatibility condition since this is the necessary condition for the existence of a solution of the system \eqref{45}.
	}
\end{remark}

\begin{theorem}
\label{NS inf}
Let $p\ge 2$ and $(\vu_\alpha, \pi_\alpha)$ be the solution of \eqref{NS} with $\bm{f}\in\vL{r(p)}$, $\mathbb{F}\in\mathbb{L}^p(\Omega)$, $\bm{h}\in\vWfracb{-\frac{1}{p}}{p}$ and $ \alpha$ a constant.
Then as $\alpha \rightarrow \infty$, we have the convergence,
	$$ (\vu_\alpha, \pi _\alpha) \rightarrow (\vu_\infty, \pi_\infty) \quad \text{ in } \quad \vW{1}{p}\times L^{p}_0(\Omega) $$
where $(\vu_\infty,\pi_\infty)$ is a solution of the Navier-Stokes problem with Dirichlet boundary condition,
	\begin{equation}
	\label{46}
	\begin{aligned}
	\begin{cases}
	-\Delta\vu_\infty +\vu_\infty \cdot \nabla \vu_\infty + \nabla \pi_\infty=\bm{f}+ \div \ \mathbb{F} &\text{ in } \ \Omega \\
	\div\;\vu_\infty=0 &\text{ in } \ \Omega \\
	\vu_\infty =\bm{0} &\text{ on } \ \Gamma .
	\end{cases}
	\end{aligned}
	\end{equation}
\end{theorem}

\begin{proof}
{\bf i)} Without loss of generality, assume $\mathbb{F}=0$ and $\bm{h} = \bm{0}$. As $\alpha \rightarrow \infty$, we can consider $\alpha \geq 1$ and then we have estimates \eqref{37} and \eqref{32}. Also as done in Theorem \ref{NS goes 0}, using the Stokes estimate (\ref{Lp1}) and (\ref{Lp1.}), we can write, for $2<p\le 3$,
\begin{equation*}
\begin{aligned}
\|\vu_\alpha\|_{\vW{1}{p}} + \|\pi_\alpha\|_{\L{p}} &\le C_p(\Omega) \left( \|\bm{f}\|_{\vL{r(p)}} + \|\vu_{ \alpha}\cdot \nabla \vu_{ \alpha}\|_{\vL{r(p)}}\right)\\
&\le C_p(\Omega) \left( \|\bm{f}\|_{\vL{r(p)}} + \|\vu_{ \alpha}\|^2_{\vH{1}}\right)\\
&\le C_p(\Omega) \left( 1+ \|\bm{f}\|_{\vL{r(p)}} \right) \|\bm{f}\|_{\vL{r(p)}}
\end{aligned}
\end{equation*}
and then similar for $p>3$ as well. This shows that $(\vu_\alpha, \pi_\alpha)$ is bounded in $\vW{1}{p}\times \L{p}$ for all $p\ge 2$. Hence there exists $(\vu_\infty, \pi_\infty)\in \vW{1}{p}\times \L{p}$ such that
$$(\vu_\alpha,\pi_\alpha)\rightharpoonup (\vu_\infty,\pi_\infty) \quad \text{ weakly \ in } \ \vW{1}{p}\times \L{p} .$$
Now rewriting the system \eqref{NS} as,
\begin{equation}
\label{modified NS}
\begin{aligned}
\begin{cases}
-\Delta\vu_\alpha +\vu_\alpha \cdot \nabla \vu_\alpha + \nabla \pi_\alpha=\bm{f} & \text{ in } \ \Omega \\
\div\;\vu_\alpha = 0 & \text{ in } \ \Omega \\
\vu_\alpha = -\frac{2}{\alpha} [(\DT\vu_\alpha)\vn]_{\vt}  & \text{ on } \ \Gamma 
\end{cases}
\end{aligned}
\end{equation}
and as in Theorem \ref{inf}, letting $\alpha \rightarrow \infty$ in the above system, we obtain that $(\vu_\infty,\pi_\infty)$ satisfies the Navier-Stokes problem \eqref{46}.\\

{\bf ii)} Therefore, $\left( \vu_\alpha-\vu_\infty\right) $ satisfies the system
\begin{equation*}
\begin{cases}
-\Delta\left( \vu_\alpha - \vu_\infty\right) + \nabla \left( \pi_\alpha - \pi_\infty\right)=\vu_\infty \cdot \nabla \vu_\infty - \vu_\alpha \cdot \nabla \vu_\alpha & \text{in $\Omega$,}\\
\div\;\left( \vu_\alpha - \vu_\infty\right)=0 & \text{in $\Omega$,}\\
\left( \vu_\alpha - \vu_\infty\right)\cdot\vn=0, \quad 2\left[\DT\left( \vu_\alpha - \vu_\infty\right)\vn\right]_{\vt}+\alpha\left( \vu_\alpha - \vu_\infty\right)_{\vt}= - 2 [(\DT\vu_\infty)\vn]_{\vt} & \text{on $\Gamma$}.
\end{cases}
\end{equation*}
Multiplying by $\left( \vu_\alpha-\vu_\infty\right) $ and integrating by parts, we get,
\begin{equation*}
\begin{aligned}
&2\int\displaylimits_{\Omega}{|\DT(\vu_\alpha - \vu_\infty)|^2}+ \alpha \int\displaylimits_{\Gamma}{|(\vu_\alpha - \vu_\infty)_{\vt}|^2}  \\
= & \-b\left( \vu_\alpha - \vu_\infty, \vu_\infty, \vu_\alpha - \vu_\infty \right) - \left\langle 2[(\DT \vu_\infty) \vn]_{\vt}, \left( \vu_\alpha - \vu_\infty\right) \right\rangle _{\vHfracbd{1}{2}\times \vHfracb{1}{2}}.
\end{aligned}
\end{equation*}
But as $\alpha\rightarrow \infty$, by compactness $\vu_\alpha \rightarrow \vu_\infty$ in $\vL{4}$ and thus
\begin{equation*}
\begin{aligned}
b\left( \vu_\alpha - \vu_\infty, \vu_\infty, \vu_\alpha - \vu_\infty \right) \le \|\vu_\alpha - \vu_\infty\|^2_{\vL{4}} \|\nabla \vu_\infty\|_{\vL{2}} \rightarrow 0.
\end{aligned}
\end{equation*}
Also since $ \vu_{ \alpha}\rightharpoonup \vu_{ \infty}$ weakly in $ \vHfracb{1}{2}$ and $[(\DT\vu_\infty)\vn]_{\vt}\in\vHfracbd{1}{2}$, it implies,
\begin{equation*}
\left\langle 2[(\DT \vu_\infty) \vn]_{\vt}, \left( \vu_\alpha - \vu_\infty\right) \right\rangle _{\vHfracbd{1}{2}\times \vHfracb{1}{2}}\rightarrow 0.
\end{equation*}
Therefore along with the fact that $\vu_{ \alpha}\rightarrow \vu_\infty$ in $\vL{2}$, we obtain the strong convergence $\vu_{ \alpha}\rightarrow \vu_\infty$ in $\vH{1}$. For the pressure term, we can write
\begin{equation*}
\begin{aligned}
\|\pi_\alpha - \pi_\infty\|_{\L{2}}
&\le C \|\nabla (\pi_\alpha - \pi_\infty)\|_{\vH{-1}} \\
&\le C \|-\Delta(\vu_\alpha - \vu_\infty) + (\vu_\alpha\cdot \nabla \vu_\alpha - \vu_\infty \cdot \nabla \vu_\infty)\|_{\vH{-1}}\\
&\le C \|-\Delta(\vu_\alpha - \vu_\infty)\|_{\vH{-1}}+ \|\vu_{ \alpha}\otimes \vu_{ \alpha} - \vu_\infty\otimes \vu_\infty\|_{\vL{2}}
\end{aligned}
\end{equation*}
and thus $\pi_\alpha \rightarrow \pi_\infty $ in $ \L{2}$.

Now similar to the Stokes case, for $p>2$, $\DT(\vu_{ \alpha}-\vu_\infty)\rightarrow 0$ in $\vL{2}$ implies $\DT (\vu_{ \alpha}-\vu_\infty)\rightarrow 0$ for almost every $\bm{x}\in\Omega$ up to a subsequence. Also using the uniform estimate of Theorem \ref{T2} for the above system satisfied by $\vu_{ \alpha}-\vu_\infty$, we can write
\begin{equation*}
\begin{aligned}
& \quad \,\,\|\vu_{ \alpha}-\vu_\infty\|_{\vW{1}{p}}\\
&\le C_p(\Omega) \left( \|\vu_{ \alpha}\otimes \vu_{ \alpha} - \vu_\infty\otimes \vu_\infty\|_{\vL{p}}+ \| 2[(\DT\vu_\infty)\vn]_{\vt}\|_{\vWfracb{-\frac{1}{p}}{p}}\right) \\
& \le C_p(\Omega)\left[ \|\vu_{ \alpha} - \vu_\infty\|_{\vL{s}} \left( \|\vu_{ \alpha}\|_{\vW{1}{p}} + \|\vu_\infty\|_{\vW{1}{p}}\right) + \|\vu_\infty\|_{\vW{1}{p}}\right] 
\end{aligned}
\end{equation*}
with $s$ defined as in (\ref{s}) so that $\vW{1}{p}\underset{compact}{\hookrightarrow} \vL{s}$. This shows that $\|\DT(\vu_{ \alpha}-\vu_\infty)\|_{\mathbb{L}^p(\Omega)}\le C$ and therefore, dominated convergence theorem yields $\|\DT(\vu_{ \alpha}-\vu_\infty)\|_{\mathbb{L}^p(\Omega)}\rightarrow 0$. This completes the proof.
\hfill
\end{proof}

\bibliographystyle{plain}
\bibliography{Master}
\end{document}